\documentclass[a4paper,10 pt,reqno]{amsart}
\usepackage{latexsym,enumerate}
\usepackage{amsmath,amssymb, amsbsy}
\usepackage[dvips]{graphicx}
\usepackage{color}
\usepackage[latin1]{inputenc}
\usepackage[active]{srcltx}
\usepackage{wrapfig}
\usepackage{graphicx}

\usepackage{mathrsfs}

\usepackage[usenames,dvipsnames]{pstricks}
\usepackage{epsfig}
 \usepackage{pst-grad} 
 \usepackage{pst-plot} 
\usepackage{epic}

\newcommand{\dint}{\dyle\int}

\newcommand{\p}{\partial}

\newcommand{\re}{{I\!\!R}}
\newcommand{\ren}{\re^N}

\newcommand{\dyle}{\displaystyle}
\newcommand{\ene}{{I\!\!N}}

\newcommand{\irn}{\int_{\re^N}}
\newcommand{\io}{\int\limits_\O}

\newcommand{\limit}{\lim\limits}

\renewcommand{\a }{\alpha }
\renewcommand{\b }{\beta }
\renewcommand{\d }{\delta }
\newcommand{\D }{\Delta }
\newcommand{\e }{\varepsilon }

\newcommand{\g }{\gamma}

\renewcommand{\l }{\lambda }

\newcommand{\n }{\nabla }

\newcommand{\s }{\sigma }

\renewcommand{\t }{\tau }

\renewcommand{\O }{\Omega }

\newcommand{\inn}{\mbox{ in }}

\renewcommand{\ge }{\geqslant}
\renewcommand{\geq }{\geqslant}
\renewcommand{\le }{\leqslant}
\renewcommand{\leq }{\leqslant}

\newtheorem{Theorem}{Theorem}[section]
\newtheorem{Corollary}[Theorem]{Corollary}
\newtheorem{Lemma}[Theorem]{Lemma}
\newtheorem{Proposition}[Theorem]{Proposition}
\theoremstyle{definition}
\newtheorem{Definition}[Theorem]{Definition}
\newtheorem{remarks}[Theorem]{Remarks}
\newtheorem{remark}[Theorem]{Remark}
\newcommand{\cqd}{{\unskip\nobreak\hfil\penalty50
        \hskip2em\hbox{}\nobreak\hfil\mbox{$\square$ \qquad}
        \parfillskip=0pt \finalhyphendemerits=0\par\medskip}}

\catcode`@=11 \@addtoreset{equation}{section}
\renewcommand\theequation{\thesection.\@arabic\c@equation}
\catcode`@=12

\begin{document}

\title[]{On the KPZ equation with fractional diffusion: global regularity and existence results}
\author[B. Abdellaoui, I. Peral, A. Primo,  F. Soria]{Boumediene Abdellaoui, Ireneo Peral${}^{(*)}$, Ana Primo, Fernando Soria }
\address{\hbox{\parbox{5.7in}{\medskip\noindent {$*$ Laboratoire d'Analyse Nonlin\'eaire et Math\'ematiques
Appliqu\'ees. \hfill \break\indent D\'epartement de Math\'ematiques, Universit\'e Abou Bakr Belka\"{\i}d, Tlemcen, \hfill\break\indent Tlemcen 13000,
Algeria.}}}}
\address{\hbox{\parbox{5.7in}{\medskip\noindent{Departamento de Matem\'aticas,\\ Universidad Aut\'onoma de Madrid,\\
        28049, Madrid, Spain. \\[3pt]
        \em{E-mail:\,}{\tt boumediene.abdellaoui@inv.uam.es, ireneo.peral@uam.es, ana.primo@uam.es, fernando.soria@uam.es
         }.}}}}

\date{\today}
\thanks{ This work was partially supported by research grants MTM2016-80474-P and  PID2019-110712GB-I00, MINECO, Spain.
The first author has been also partially supported by an Erasmus grant from Autonomous University of Madrid and by the DGRSDT, Algeria. \\
During the final stages of this work, Ireneo Peral sadly passed away after a very short period of illness. He worked very hard and made fundamental contributions to this project, but the responsibility for the correctness of the presentation rests solely and exclusively with the other three authors. The paper is dedicated to him as a tribute for his guidance and teachings.} 

\begin{abstract}
In this work we analyze the existence of solutions to the fractional quasilinear problem,
$$
(P) \left\{
\begin{array}{rcll}
u_t+(-\Delta )^s u &=&|\nabla u|^{\alpha}+ f &\inn \Omega_T\equiv\Omega\times (0,T),\\ u(x,t)&=&0 & \inn(\mathbb{R}^N\setminus\Omega)\times [0,T),\\
u(x,0)&=&u_{0}(x) & \inn\Omega,\\
\end{array}\right.
$$
where $\Omega$ is a $C^{1,1}$ bounded domain in $\mathbb{R}^N$, $N> 2s$ and  $\frac{1}{2}<s<1$. We will assume  that $f$ and $u_0$ are non negative functions
satisfying some additional hypotheses that  will be specified later on.

Assuming certain regularity on $f$, we will prove the existence of a solution to problem $(P)$ for values $\alpha<\dfrac{s}{1-s}$, as well as the non existence of such a
solution when $\alpha>\dfrac{1}{1-s}$. This behavior clearly exhibits a deep difference with the local case.
\vskip 2mm
\hfill {\it To Ireneo, our teacher, colleague and friend, in memoriam.}

\end{abstract}

\maketitle
{\footnotesize
\tableofcontents
}

\section{Introduction}
In the paper \cite{KPZ}, Kardar, Parisi and Zhang describe the following  model for the growth of surfaces
$$u_t-\Delta u=c\sqrt{1+|\nabla u|^2}+f, $$
where $f$ represents in general  a stochastic process. After a Taylor expansion for small size of the gradient, they consider instead the so called KPZ equation,
$$u_t-\Delta u=c|\nabla u|^2+g.$$
It is important to say that, from a physical point of view, the KPZ equation is a relevant case of study because, among other things,  it defines a new \textit{ universality class} for a lot of models in Statistical
Mechanics (see for instance \cite{barstan} and \cite{cor}).
In that sense, the behavior of the so called Hopf-Cole class of solutions has been deeply researched in the seminal paper by M. Hairer  \cite{Hairer}.

\

We will restrict ourselves to the deterministic setting, that is, when  the source term is a function with a suitable summability.
In the local critical case $\a=2$ with $s=1$ there is a large literature of results. For instance,  in the paper \cite{ADP} (see also the references therein) a classification of the
solutions was found showing in particular an extreme case of non-uniqueness.

\

The relevant facts in the KPZ model are that the growth is driven in the direction of the gradient of the interface while the diffusion comes from the classical Laplacian
(remember that behind all this one finds always a Brownian motion).

\

 There is another question to take into account, which is that the diffusion to consider  may  change according with the medium.  For instance in the paper by
 Barenblatt,  Bertsch,  Chertock, and  Prostokishin, \cite{baren}, the growth in a porous medium was considered.  This model has been studied for instance in the
 papers, \cite{APW} and \cite{APW1}.  The diffusion for a power law in the gradient (the p-Laplacian) has been also studied, see for instance  \cite{ADPS},
 \cite{BG} and the references therein.

\

The main goal of this work is to study a non local version of the  Kardar-Parisi-Zhang equation. More precisely, the idea is to consider diffusion driven by the
\textit{fractional Laplacian} (so that  behind it what one finds is a Levy process). More specifically, we deal with the problem
\begin{equation}\label{grad}
\left\{
\begin{array}{rcll}
u_t+(-\Delta )^s u&=&|\nabla u|^{\alpha}+ f & \inn \Omega_T\equiv\Omega\times (0,T),\\ u(x,t)&=&0 & \inn(\mathbb{R}^N\setminus\Omega)\times [0,T),\\
u(x,0)&=&u_{0}(x) & \inn\Omega,\\
\end{array}\right.
\end{equation}
where $\Omega$ is a $C^{1,1}$ bounded domain in $\mathbb{R}^N$, $N> 2s$ and  $\frac{1}{2}<s<1$. We suppose  that $f$ and $u_0$ are non negative functions
satisfying some hypotheses that we will precise later.

\

By $(-\Delta)^s$ we mean the  fractional Laplacian operator of order $2s$ given as the multiplier of the Fourier transform with symbol $|\xi|^{2s}$. That is,  for every $u\in \mathscr{S}(\mathbb{R}^N)$, the Schwartz class, $$(-\Delta)^{s}u=\mathcal{F}^{-1}(|\xi|^{2s}\mathcal{F}(u)),\, \xi\in\mathbb{R}^{N}, s\in(0,1),$$
where $\mathcal{F}$ denotes Fourier transform and $\mathcal{F}^{-1}$ its inverse.

\

As was indicated by M. Riesz in his foundational paper \cite{Riesz}, since  the formal homogeneous  kernel corresponding to the multiplier $|\xi|^{2s}$ is a constant multiple of
$|x|^{-N-2s}$, therefore not in $L^1_{loc}$, the definition cannot be a convolution but rather a principal value given by the following expression
\begin{equation}\label{fraccionario}
(-\Delta)^{s}u(x):=a_{N,s}\mbox{ P.V. }\int_{\mathbb{R}^{N}}{\frac{u(x)-u(y)}{|x-y|^{N+2s}}\, dy},\quad s\in(0,1),
\end{equation}
where
$$a_{N,s}:=2^{2s-1}\pi^{-\frac N2}\frac{\Gamma(\frac{N+2s}{2})}{|\Gamma(-s)|}$$
is the normalization constant related  to the definition through the Fourier transform. This formula is obtained by analytic continuation of the Riesz
potentials in the complex plane.  See the details for instance in \cite{Landkof} and \cite{PSbook}.
The hypothesis $s>\frac 12$ is a natural assumption to allow the presence of a power of  the gradient as a nonlinear perturbation.
The stationary problem associated to problem \eqref{grad} has recently been studied in \cite{CV2} and \cite{AP}.

\

 The interest of the fractional  Laplacian is motivated, in addition  to the mathematical relevance, by the fact that it has recently been used  in a number of equations
 modeling concrete phenomena. Among others, we mention crystal dislocation \cite{DP_P_V_2015}, \cite{DP_F_V_2014}, mathematical finances \cite{Apple_2004} and
 quantum mechanics, see \cite{L_2000}.

For the nonlocal case $s\in (\frac 12,1)$ and for regular data, the authors in \cite{W,W1} proved the existence of a regular solution using semi-group theory and probabilistic tools.  More precisely, the authors in \cite{W} treat the case $f=0$ and $\O=\ren$, under regular assumption on $u_0$, in order to get global estimates. In their approach, the fact that $\O=\ren$  turns out to be a fundamental key,  which is lost in the case of a bounded domain.  As we will see later, the work on bounded domains, $\O$, generates a loss of regularity near the boundary  and, as a consequence, non existence results holds for large values of $q$.

\

The main goal of this paper is to consider  a general class of data. It is important to remark that monotony arguments have serious limitations in order to pass to
the limit in the approximating problems. To overcome these difficulties we will  follow the arguments used for the elliptic case in \cite{AP}. In particular, we will use
apriori estimates and the Schauder fixed point theorem, which in the stationary case  are inspired by results in  \cite{CON2} and \cite{CON1} for local operators.

\

We briefly sketch now the main results in the paper. First,
via a fixed point argument we obtain the following results for general data.
\begin{Theorem}\label{maria}
Suppose in the problem \eqref{grad} that $\alpha<\dfrac{N+2s}{N+1}$, then { there exists $T:=T(\O,s)>0$ such that } for
all $(f, u_0)\in L^m(\O_T)\times L^1(\O)$ with $1\le m<\dfrac{1}{s}$, problem
\eqref{grad} has a solution $u\in L^{q}(0,T;W_{0}^{1,q}(\Omega))$ for all $q<\dfrac{N+2s}{N+1}$ and $T_k(u)\in L^{2}(0,T;H^s_0(\Omega))$ for all $k>0$, { where $T_k(\s)$ is given by
\begin{gather*}\label{f-trun}
    T_k(\s)=\left\{\begin{array}{cl}
    \s\,,&\hbox{ if }|\s|\le k\,;\\[2mm]
    k\frac{\s}{|\s|}\,,&\hbox{ if }|\s|> k\,;
    \end{array}\right.
   \end{gather*}
}
\end{Theorem}

Notice that, even for $f\in L^m(\O_T)$ with $1<m<\dfrac{1}{s}$, the existence result holds with the same assumption on $\a$ as in $L^1$ data. This assumption does
not appear in the local case $s=1$ where the relation $\a\leftrightarrow m$ is strictly increasing.  In the non local case, this  limitation is due to the fact
that the global regularity for the gradient term imposes many restrictions on the parameters $s,m,\a$ and makes a fundamental difference between the local and the
nonlocal case.

In the case of $L^1$-data the above existence result is optimal, in the sense that  if $\a>\frac{N+2s}{N+1}$, then we can find $f\in L^1(\O_T)$ or $u_0\in
L^1(\O)$ such that problem \eqref{grad} has no solution in the space $L^\a(0,T; W^{1,\a}_0(\O))$.

\

For large value of $\a$ a serious limitation appears as a consequence of the lack
 of regularity for the linear problem near the boundary.
 This loss of regularity allows us to  get the following non existence
 result which makes more significant the difference between the local and
 the nonlocal case. However,  this is coherent with the local case; indeed,
 one sees in the threshold that $\dfrac 1{1-s}\to \infty$ as $s\to 1$.

\begin{Theorem}\label{nonint}
Suppose that $\a>\dfrac{1}{1-s}$, and let $(f,u_0)\in L^\infty(\O_T)\times L^\infty(\O)$ be nonnegative functions with $(f,u_0)\neq (0,0)$. Then, the problem
$$
\left\{
\begin{array}{rcll}
u_t+(-\Delta )^s u&=&|\nabla u|^{\alpha}+ f & \inn \Omega_T\equiv\Omega\times (0,T),\\ u(x,t)&=&0 & \inn(\mathbb{R}^N\setminus\Omega)\times [0,T),\\
u(x,0)&=&u_{0}(x) & \inn\Omega,\\
\end{array}\right.
$$
has no solution $u$ such that $u\in L^\a(0,T;W^{1,\a}_0(\O))$.
\end{Theorem}

In the same direction we prove a general non existence result  of  weak solutions in a suitable weighted Sobolev space for a range of the parameter $\a$. This
gives a fundamental difference related to the local case, $s=1$, where existence of a solution is proved for all $\a>1$ under suitable regularity assumptions on
the data, see for example \cite{BS}.

\

For the existence result, we will distinguish two   types of problems according to the integrability of the gradient term. In the first case we look for the global integrability of the gradient term in the whole domain $\Omega_T$.
\begin{Theorem}\label{hhh}
Assume that $\dfrac{2s-1}{1-s}>\dfrac{(N+2s)^2}{N+1}$ and that $\dfrac{N+2s}{N+1}\le \a<\dfrac{2s-1}{(1-s)(N+2s)}$.

Suppose that $u_0=0$, $f\in L^{m}(\Omega_T)$ with $m\ge \dfrac{1}{s}$ satisfies  one of the following conditions:
\begin{enumerate}
\item[(I)] either $\dfrac{N+2s}{2s-1}\le m$,
\item[(II)] or $\dfrac{N+2s}{\a'}\dfrac{1}{(2s-1)-(1-s)(N+2s)}<m<\dfrac{N+2s}{2s-1}$,
\end{enumerate}
then {there exists $T:=T(\O,s)>0$ such that} problem \eqref{grad} has a solution $u\in L^{\a}(0,T;W_{0}^{1,\a}(\Omega))$  and moreover $u\in L^{\g}(0,T;W_{0}^{1,\g}(\Omega))$ for all $\g<\dfrac{1}{1-s}$ if $(I)$ holds and $u\in L^{\g}(0,T;W_{0}^{1,\g}(\Omega))$ for all
$\g<\dfrac{m(N+2s)}{(N+2s)(m(1-s)+1)-m(2s-1)}$ if $(II)$ holds.
\end{Theorem}

Notice that the hypothesis on $s$ means that $s$ must be close to $1$ and $\alpha<<\dfrac{s}{1-s}$. The above conditions are needed  to  ensure  the global integrability of the gradient term in the whole $\O_T$.

\noindent For the complete range of the parameter $\alpha$, that is $\dfrac{N+2s}{N+1}\le \alpha<\dfrac{s}{1-s}$ and without any restriction in  the order of the operator  $s>\dfrac 12$, to have  a weak solution to the problem requires a natural weight that in fact is a power of the distance to the boundary of $\Omega$. { For simplicity, throughout this paper we denote $\d(x):=\text{dist}(x,\p\O)$, with $x\in \O$, such a distance.} Hence, the existence of a distributional solution in this case will be obtained in a weighted Sobolev space.
More precisely we have the following result.
\begin{Theorem}\label{fix001}
For every $s\in(\dfrac 12, 1)$, assume that $\dfrac{N+2s}{N+1}\le \a<\dfrac{s}{1-s}$. Let $f$ be a nonnegative function such that $f\in L^m(\O)$ with
$m>\max\bigg\{\dfrac{N+2s}{s(2s-1)}, \dfrac{N+2s}{s-\a(1-s)}\bigg\}$, then {{there exists $T:=T(\O,s)>0$ such that}}  problem \eqref{grad} has
a distributional solution
$u\in L^{\a}(0,T;W^{1,\a}_{loc}(\O))\cap L^1(0,T;W^{1,1}_0(\O))$.
Moreover $u\d^{1-s}\in L^{m\a}(0,T;W^{1,m\a}_0(\O)).$
\end{Theorem}

\

In the common values of $\alpha$ in Theorem \ref{hhh} we require  $s$ to be close to $1$. Also,  the integrability of the datum, $m$, is bigger than in  Theorem \ref{fix001}, where there is no restriction in $s$. The optimality of the results remain open.

\

If  the source term { $f$ is null, then we can prove the existence of a solution using a suitable change of function and  Theorem \ref{hhh}. More precisely we have }
\begin{Theorem}\label{hhh2}
In the problem  \eqref{grad}, let us consider $f=0$. Assume that $\dfrac{2s-1}{1-s}>\dfrac{(N+2s)^2}{N}$ and that $\dfrac{N+2s}{N+1}\le \a<\dfrac{2s}{(1-s)(N+2s)+1}$.
Let $u_0$ be a nonnegative measurable function such that $u_0\in L^\s(\O)$ with $\s>\dfrac{(\a-1)N}{(2s-\a)-\a(1-s)(N+2s)}$, then {there exists $T:=T(\O,s)>0$ such that} the problem
\begin{equation}\label{gradu0}
\left\{
\begin{array}{rcll}
u_t+(-\Delta )^s u&=&|\nabla u|^{\alpha} &\inn \Omega_T,\\ u(x,t)&=&0 &\inn(\mathbb{R}^N\setminus\Omega)\times [0,T),\\ u(x,0)&=& u_0(x) & \inn\Omega,
\end{array}\right.
\end{equation}
has a solution $u\in L^{\g}(0,T;W_{0}^{1,\g}(\Omega))$, for all $\g< \dfrac{\s(N+2s)}{(1-s)\s(N+2s)+N+\s}$.
\end{Theorem}

\

 A similar result with weights could be obtained by application of Theorem \ref{fix001}. This problem will be studied in a forthcoming paper, in which we   will also look for the asymptotic behavior of the solutions with respect to $t$.

\

As a direct application of the arguments developed,  we  will treat the problem with \textit{drift}, that is, the nonlinearity  on the gradient is substituted by a term of the form  $B(x,t)\,\cdot\, \nabla u$.
 The existence of a solution is known in the literature under
 regularity conditions on the data $f$ and $u_0$.

 Here we will prove the existence of a solution under natural
 condition on the field $B$ and general data  $(f,u_0)\in L^1(\O_T)\times L^1(\O)$.  Using a suitable Harnack inequality we are able to  prove  a
 comparison principle and then the uniqueness of the solution follows
 for the linear problem with drift. This will be the key in order to show
 the uniqueness of a \emph{good solution} to problem \eqref{grad}.

\
{
\begin{remark}
 Notice that the above existence results also hold true without any additional assumption on $T$ if we alternatively assume that the corresponding norm of the data are small. This  can be seen  if we substitute the data $f,u_0$ with $\l f,\l u_0$ with $|\l|$ small.
 However in the case of Theorem  \ref{maria}, using the uniqueness result proved in Theorem \ref{uniqapr}, we are able to show the existence of a minimal solution (and in some cases a unique solution) without any rescription on the norm of the data. See Theorem \ref{uniqq} below.
 \end{remark}
}

The paper is organized as follows.
In Section \ref{sec2} we collect some  tools that will be used systematically in the paper. We begin by specifying  the sense in which the solutions are
understood, and we state some  tools as the Kato's inequality and the gradient regularity for an associated elliptic problem.

In order to prove gradient regularity for the solution based on the representation  formula, we need to show gradient regularity for the fractional heat kernel. This is done in Section \ref{sec3} where we also consider the general linear fractional heat equation with $L^1$ data. In this case we are able to show the uniqueness of the solution and to prove the strong convergence of the solution of the approximating problems in a suitable Sobolev space without using the Landes regularizing approximation. General regularity results of the gradient and the Hardy-Sobolev term $\dfrac{u}{\delta^s}$ are obtained, in particular,  close to  $\partial \Omega\times(0,T)$.
This will be useful in order to complete the regularity schema of this class of equations, as is done in \cite{AP} for the elliptic equation.

As a consequence of the results in Section 3, Section 4 is devoted to prove the surprising nonexistence result.  We consider the Problem \eqref{grad} in Section \ref{sec4}. We begin with the case of $L^1$ data and for all $\a<\a_0=\dfrac{N+2s}{N+1}$ and  prove the existence of a weak solution with suitable regularity. We also
prove that the condition on $\a$ is optimal in the sense that for $\a>\a_0$, then there exists $f\in L^1(\O_T)$ such that problem \eqref{grad} has no solution.

Problem \eqref{grad} with general $\alpha<\dfrac{s}{1-s}$ is treated in Subsection \ref{sub:sec41}. Under suitable hypotheses on $f$, we are able to show the
existence of a weak solution. We also treat the case where $f\equiv 0$ and $u_0\gneqq 0$. Following closely the ideas as in the case $f\gneqq 0$, we
 prove the existence of a solution that lives in a suitable Sobolev space.

In the last section we consider the linear problem with \textit{drift}. The existence of a weak solution is proved for all data in $L^1$.
According to additional hypotheses on the \textit{drift} term, we are able to show the uniqueness of the weak solution. As a consequence we prove a  general comparison principle, using a suitable singular Gronwall-Bellman inequality, that allows us to show the existence of a minimal solution to problem \eqref{grad} under suitable hypothesis on $\alpha$.

{ We thank A. Younes for pointing  us some misprints in an earlier version of this work.}

\section{Preliminaries and functional setting}\label{sec2}

Let us begin by some results from fractional Sobolev spaces that will be used in this paper. We refer to \cite{dine} for more details and proofs.

Assume that $s\in (0,1)$ and $p>1$. For a measurable $\O\subset \ren$,  the fractional Sobolev Space $W^{s,p}(\Omega)$ is defined by
$$
W^{s,p}(\Omega)\equiv \Big\{ \phi\in L^p(\O):\dint_{\O}\dint_{\O}|\phi(x)-\phi(y)|^pd\nu<+\infty\Big\},
$$
where, for simplicity, we set
\begin{equation}\label{thekernel}
d\nu=\dyle\frac{dxdy}{|x-y|^{N+ps}}.
\end{equation}

Notice that $W^{s,p}(\O)$ is a Banach Space endowed with the norm
$$
\|\phi\|_{W^{s,p}(\O)}= \Big(\dint_{\O}|\phi(x)|^pdx\Big)^{\frac 1p} +\Big(\dint_{\O}\dint_{\O}|\phi(x)-\phi(y)|^pd\nu\Big)^{\frac 1p}.
$$
The space $W^{s,p}_{0} (\O)$ is defined as the completion of $\mathcal{C}^\infty_0(\O)$ with respect to the previous norm.

If $\O$ is a bounded regular domain, we can endow $W^{s,p}_{0}(\O)$ with the equivalent norm
$$
||\phi||_{W^{s,p}_{0}(\O)}= \Big(\int_{\O}\dint_{\O}|\phi(x)-\phi(y)|^pd\nu\Big)^{\frac 1p}.
$$

The next Sobolev inequality is proved in \cite{Adams}, see also \cite{dine} and \cite{Ponce} for an elementary proof.
\begin{Theorem} \label{Sobolev}(Fractional Sobolev inequality)
Assume that $0<s<1, p>1$ satisfy $ps<N$. There exists a positive constant $S\equiv S(N,s,p)$ such that for all $v\in C_{0}^{\infty}(\ren)$,
$$
\iint_{\re^{2N}} \dfrac{|v(x)-v(y)|^{p}}{|x-y|^{N+ps}}\,dxdy\geq S \Big(\dint_{\mathbb{R}^{N}}|v(x)|^{p_{s}^{*}}dx\Big)^{\frac{p}{p^{*}_{s}}},
$$
where $p^{*}_{s}= \dfrac{pN}{N-ps}$.
\end{Theorem}
We will denote by $H^s(\ren)$ the Hilbert space $W^{s,2}(\ren)$. If $u\in H^s(\ren)$, we define
$$
(-\Delta)^s u(x)={ P.V. } \dint_{\ren}\dfrac{u(x)-u(y)}{|x-y|^{N+2s}}{dy}.
$$
For $w, v\in H^s(\ren)$, we have
$$
\langle (-\Delta)^s w,v\rangle =\dfrac 12\iint_{\re^{2N}}\dfrac{(w(x)-w(y))(v(x)-v(y))}{|x-y|^{N+2s}}{dxdy}.
$$
If $H^s_0(\O)$ is the closure of $\mathcal{C}_0^\infty(\Omega)$ with respect to the norm $H^s(\ren)$ and if $w, v\in H^s_0(\O)$, then
$$
\langle (-\Delta)^s w,v\rangle =\dfrac 12\iint_{D_\O}\dfrac{(w(x)-w(y))(v(x)-v(y))}{|x-y|^{N+2s}}{dxdy},
$$
where $D_{\O}=(\ren\times \ren)\setminus (\mathcal{C}\O\times \mathcal{C}\O)$.

Since we are considering parabolic problems, we need to define the corresponding parabolic spaces.
 For $q\ge 1$, the space $L^{q}(0,T; W^{s,q}_0(\O))$ is defined as the set of functions $\phi$ such that
$\phi\in L^q(\O_T)$ with $||\phi||_{L^{q}(0,T; W^{s,q}_0(\O))}<\infty$ where
$$
||\phi||_{L^{q}(0,T; W^{s,q}_0(\O))}=\Big(\int_0^T\iint_{D_{\O}}|\phi(x,t)-\phi(y,t)|^qd\nu\,dt\Big)^{\frac 1q}.
$$
It is clear that $L^{q}(0,T; W^{s,q}_0(\O))$ is a Banach Space.

Consider now the linear problem
\begin{equation}\label{eq:def-0}
\left\{
\begin{array}{rcll}
u_t+(-\D)^s u&=& \dyle f  & \text{ in } \O_{T}=\Omega \times (0,T)  , \\ u&=&0 & \text{ in }(\ren\setminus\O) \times (0,T), \\ u(x,0)&=&u_0(x)& \mbox{  in  }\O,
\end{array}%
\right.
\end{equation}
where $\O\subset \ren$ is a bounded regular domain. If the data $(f,u_0)\in L^2(\O_T)\times L^2(\O)$, then we can deal with energy solution. More precisely we
have the next definition.
\begin{Definition}\label{energy}
Assume  $(f,u_0)\in L^2(\O_T)\times L^2(\O)$, then we say that $u$ is an energy solution to problem \eqref{eq:def-0} if $u\in L^{2}(0,T; H^s_0(\O))\cap
\mathcal{C}([0,T], L^2(\O))$, $u_t\in L^{2}(0,T; H^{-s}(\O))$, and for all $v\in  L^{2}(0,T; H^s_0(\O))$ we have
\begin{equation*}
\begin{array}{lll}
&\dyle\int_0^T\langle u_t, v\rangle dt +\dfrac 12\int_0^T\iint_{D_{\O}}(u(x,t)-u(y,t))(v(x,t)-v(y,t))d\nu\ dt\\ &\dyle=\iint_{\O_T} fv dx\,dt
\end{array}
\end{equation*}
and $u(x,.)\to u_0$ strongly in $L^2(\O)$, as $t\to 0$.
\end{Definition}
Notice that the existence of energy solution follows using classical arguments, see for instance \cite{LPPS}.

If the datum lies in $L^1$, we need to define a more general concept of solution. Let us begin by the next definitions.

Assume that $\alpha,\beta\in (0,1)$, we define the set
\begin{equation*}\begin{split}
\mathcal{T}:=\{&\phi:\mathbb{R}^N\times [0,T]\rightarrow\mathbb{R},\,\hbox{ s.t. }-\phi_t+(-\Delta)^s\phi=\varphi,\, \varphi\in L^\infty(\Omega\times (0,T))\cap
\mathcal{C}^{\alpha, \beta}(\Omega\times (0,T)),\\ &\phi=0\inn (\ren\setminus \Omega)\times   {(0,T]},\phi(x,T)=0 \inn \Omega \}.
\end{split}\end{equation*}
From \cite{LPPS}, we know that if $\phi\in \mathcal{T}$, then $\phi\in L^\infty(\Omega\times (0,T))$ and $\phi\in \mathcal{T}$ satisfies the equation in a
pointwise sense.

We are now able to state the meaning of weak solution.

\begin{Definition}\label{veryweak}  Assume that $(f,u_{0})\in L^1(\O_T)\times L^{1}(\Omega)$.
We  say that $u\in \mathcal{C}([0,T); {L}^{1}(\O))$,  is a weak solution to problem \eqref{eq:def-0} if for all $\phi\in \mathcal{T}$ we have
\begin{equation}\label{eq:subsuper}
\iint_{\O_T}\,u\big(-\phi_t\, +(-\Delta)^{s}\phi\big)\,dx\,dt=\\ \iint_{\O_T}\,f\phi\,dxdt +\int_\Omega{u_0(x)\phi(x,0)\,dx}.
\end{equation}
\end{Definition}
The next existence result is proved in \cite{LPPS}, ({ see also \cite{AABP} and \cite{BWZ} for some different approaches.})
\begin{Theorem}\label{th1}
Suppose that $(f,u_0)\in L^1(\O_T)\times L^1(\O)$, then problem \eqref{eq:def-0} has a unique weak solution $u$ such that $u\in \mathcal{C}([0,T];L^1(\O))\cap
L^m(\O_T)$ for all $m\in [1, \frac{N+2s}{N})$, \;$|(-\D)^{\frac{s}{2}} u|\in L^{r}(\O_T)$ for all $r\in[1, \frac{N+2s}{N+s})$ and $T_k(u)\in L^2(0,T,H^s_{
0}(\Omega))$ for all $k>0$ where $T_k(\s)=\max \{-k, \min\{k,\s\}\}$. Moreover $u\in {L^q(0,T,W^{s,q}_{ 0}(\O))}$  for all $1\le q<\frac{N+2s}{N+s}$. In addition
we have
\begin{equation}\label{main-estim}
\begin{array}{lll}
&\dyle ||u||_{\mathcal{C}([0,T],L^1(\O))}+ ||u||_{L^m(\O_T)}+||(-\D)^{\frac{s}{2}} u||_{L^{r}(\O_T)}+ ||u||_{L^q(0,T,W^{s,q}_{ 0}(\O))}\\ &
\dyle \le
C(\O_T)\bigg(||f||_{L^1(\O_T)}+||u_0||_{L^1(\O)}\bigg).
\end{array}
\end{equation}
\end{Theorem}

\begin{remark}\label{optimal}
The regularity condition obtained in Theorem \ref{th1} is optimal in the sense that if $m\ge \dfrac{N+2s}{N}$, then we can find $f\in L^1(\O_T)$ and $u_0\in
L^1(\O)$ such that $u^m\notin L^1(\O_T)$. This fact will be used in Theorem \ref{maria} in order to show the optimality of the condition imposed on $\a$.
\end{remark}

{In the case of having data in $L^1_{loc}$  the natural concept is
the usual \textit{distributional solution}, that is, given by the following
definition.
\begin{Definition}\label{distribu}
Assume that $(f,u_0)\in L^1_{loc}(\O_T)\times L^1_{loc}(\O)$. We say that $u\in L^1(\O_T)\cap\,\mathcal{C}([0,T], L^1_{loc}(\O))$ is a distributional solution to Problem \eqref{eq:def-0} if for all $\varphi \in \mathcal{C}^\infty_0(\O_T)$, for all $\eta\in \mathcal{C}^\infty_0(\O)$, we have
$$
\iint_{\O_T} u(-\varphi_t+(-\Delta)^s\varphi)\,dx\,dt=\iint_{\O_T}f(x,t)\varphi(x,t)\,dx\,dt, $$
and
$$
\io u(x,t)\,\eta(x)dx\to \int_{\O}u_0(x)\,\eta(x)\,dx \mbox{ as  }t\to 0.
$$
\end{Definition}
}
\

\

Finally, the next Kato type inequality will be useful in order to show apriori estimates and the positivity of the solution. { The proof follows exactly as in the elliptic case proved in \cite{CV1}. { (See also \cite{PSbook}).}
\begin{Theorem}\label{Conv}
Let $\phi\in \mathcal{C}^{2}(\re)$ be a convex function. Assume $u\in L^{2}(0,T; H^s_0(\O))\,\cap\, \mathcal{C}([0,T], L^2(\O))$. Define $v=\phi(u)$ and suppose
that $|v_t+(-\D)^s v|\in L^1(\O_T)$, then
\begin{equation}\label{ka}
v_t+(-\D)^s v\leq \phi'(u) (u_t+(-\D)^s (u)).
\end{equation}
\end{Theorem}

\section{Gradient regularity of the solutions to the linear problem with $\frac 12<s<1$.}\label{sec3}

In this section the principal goal is to study the gradient regularity of the solution to the linear problem
\begin{equation}\label{eq:def}
\left\{
\begin{array}{rcll}
u_t+(-\D)^s u&=& \dyle f  & \text{ in } \O_{T}=\Omega \times (0,T)  , \\ u&=&0 & \text{ in }(\ren\setminus\O) \times (0,T), \\ u(x,0)&=&u_0(x)& \mbox{  in  }\O,
\end{array}%
\right.
\end{equation}
where $\O\subset \ren$ is a bounded regular domain and assuming  that $s\in (\frac 12,1)$.

We emphasize that the hypothesis $\frac 12< s<1$ is assumed across the whole section.

Under this hypothesis the heat kernel lies in the space $L^q((0,T), W^{1,q}_0(\O))$  for suitable $q$ as we will see later. This regularity is the starting point  to study the problem \eqref{eq:def}.

Denote by $P_{\Omega}$ the  kernel of the heat equation $\dfrac{d}{dt}+(-\Delta)^s$ in $\Omega\times(0,T)$ with Dirichlet boundary condition. The solution,  $u$, to problem \eqref{eq:def} is represented by
$$
u(x,t)=\dint_{\Omega} u_0(y) P_{\Omega} (x,y, t)\,dy+ \iint_{\O_t}f(y,\sigma) P_{\Omega} (x,y, t-\sigma)\,dy\,d\sigma,
$$
where $\O_t=\O\times (0,t)$.
As in the elliptic case, in order to get apriori estimates on the gradient of $u$, we need some sharp properties of the heat kernel $P_{\Omega}$. These properties
are resumed in the next lemma whose proof can be found in \cite{BJ}, \cite{CZ} and  \cite{KS}.
 \begin{Lemma}\label{estimmm}
Assume that  $s\in (\frac 12,1)$, then for all $x,y\in \O$ and for all $0<t<T$,
\begin{equation}\label{green1}
P_{\Omega}(x,y,t)\backsimeq \Big( 1\wedge \frac{\delta^s(x)}{\sqrt{t}}\Big)\times \Big( 1\wedge \frac{\delta^s(y)}{\sqrt{t}}\Big)\times
\Big(t^{-\frac{N}{2s}}\wedge \frac{t}{|x-y|^{N+2s}}\Big)
\end{equation}
and
\begin{equation}\label{green2}
|\nabla_x P_{\Omega}(x,y,t)|\leq C \Big( \dfrac{1}{\delta(x) \wedge t^{\frac{1}{2s}}}\Big)\, P_{\Omega}(x,y,t).
\end{equation}
Setting
\begin{equation}\label{green0}
\dint_{0}^{\infty} P_{\Omega} (x,y,
t)\,dt=\mathcal{G}_s(x,y)
\end{equation}
where $\mathcal{G}_s(x,y)$ is the Green function of the fractional Laplacian operator with Dirichlet condition, that is,
\begin{equation}\label{green00}
\mathcal{G}_s(x,y)\simeq \frac{1}{|x-y|^{N-2s}}\bigg(\frac{\d^s(x)}{|x-y|^{s}}\wedge 1\bigg) \bigg(\frac{\d^s(y)}{|x-y|^{s}}\wedge 1\bigg).
\end{equation}
\end{Lemma}
\begin{remark} Notice that
$$ P_{\Omega} (x,y,t)\le t^{-\frac{N}{2s}}\wedge \frac{t}{|x-y|^{N+2s}},$$
therefore
\begin{equation}\label{upper-kernel}
P_{\Omega} (x,y,t)\le \dfrac{2t}{t^\frac{N+2s}{2s}+|x-y|^{N+2s}}\le C(s,N)\dfrac{t}{(t^\frac 1{2s}+|x-y|)^{N+2s}}.
\end{equation}
Hereafter we will call
\begin{equation}\label{upper-kernel11}
H(x,t):=\dfrac{t}{(t^\frac 1{2s}+|x|)^{N+2s}}.
\end{equation}
\end{remark}
We begin by proving the following basic result.
\begin{Proposition}\label{propo1}
Assume that  $s\in (\frac 12,1)$, then for all $q<\dfrac{N+2s}{N+1}$, we have
$$
\dint_{0}^{T} \dint_{\Omega} \dint_{\Omega}|\nabla_x P_{\Omega}|^q\, dx\,dy\,dt\le C(\O)(T^{\frac{N+2s-q(N+s)}{2s}}+T^{\frac{N+2s-q(N+1)}{2s}}).
$$
\end{Proposition}
\begin{proof}
From \eqref{green1} and \eqref{green2} we obtain that $$ |\nabla_x P_{\Omega}(x,y,t)|\leq C \Big( \dfrac{1}{\delta(x) \wedge t^{\frac{1}{2s}}}\Big)\,\Big( 1\wedge
\frac{\delta^s(x)}{\sqrt{t}}\Big)\times \Big( 1\wedge \frac{\delta^s(y)}{\sqrt{t}}\Big)\times \dfrac{t}{\bigg(t^{\frac{1}{2s}}+|x-y|\bigg)^{N+2s}}.
$$
Hence, according with \eqref{upper-kernel},
\begin{equation}\label{rrr}
|\nabla_x P_{\Omega} (x,y,t)|\leq \left\{\begin{array}{rcl}
 \Big( 1\wedge \dfrac{\delta^s(y)}{\sqrt{t}}\Big)\cdot \dfrac{\sqrt{t}}{(\delta(x))^{1-s}} \cdot \dfrac{C}{(t^{\frac{1}{2s}} +|x-y|)^{(N+2s)}}, \mbox{ if }
 \delta(x)<t^{\frac{1}{2s} },\\
 \\
 C \Big( 1\wedge \dfrac{\delta^s(y)}{\sqrt{t}}\Big)\cdot \dfrac{t^{\frac{2s-1}{2s}}}{(t^{\frac{1}{2s}}+ |x-y|)^{N+2s}}, \mbox{ if } \delta(x)\geq t^{\frac{1}{2s}
 }.
\end{array}
\right.
\end{equation}
Thus
\begin{eqnarray*}
&\dyle \!\!\!\! \iiint_{\O_T\times \O}|\nabla_x P_{\Omega}|^q\, dx\,dy\,dt\le \\ \\ &\dyle \iiint_{\{(0,T)\times \O\times \O\}\cap
\{\delta(x)<t^{\frac{1}{2s}}\}}|\nabla_x P_{\Omega}|^q\, dx\,dy\,dt+\!\!  \iiint_{\{(0,T)\times \O\times \O\}\cap \{\delta(x)\ge t^{\frac{1}{2s}}\}}|\nabla_x
P_{\Omega}|^q\, dx\,dy\,dt:= I_1+I_2.
\end{eqnarray*}
Using \eqref{rrr}, it holds that
$$
\begin{array}{rcl}
I_1 &\leq& \dint_{\Omega}\dfrac {dx}{(\delta (x))^{(1-s)q}}\Bigg(\iint_{\O_T}\dfrac{t^{\frac{q}{2}}\,dydt} {(t^{\frac{1}{2s}}+|x-y|)^{q(N+2s)}}
\Bigg)\\
\\
&\leq& \dint_{\Omega}\dfrac {dx}{(\delta (x))^{(1-s)q}} \Bigg(\dint_{0}^T \dint_{{ \ren}}  \dfrac{t^{\frac{q}{2}- \frac{1}{2s} q(N+2s)}\,dydt}
{\Big(1+\frac{|x-y|}{t^{\frac{1}{2s}}}\Big)^{q(N+2s)}} \Bigg)\\
\\
&=&  C\dint_{\Omega}\dfrac {dx}{(\delta (x))^{(1-s)q}}\cdot \dint_{0}^{T} t^{\frac{N}{2s}} t^{\frac{qs-q(N+2s)}{2s}}\,dt\cdot \dint_{0}^{\infty}
\dfrac{\theta^{N-1}}{(1+\theta) ^{(N+2s)q}}\,d\theta,
\end{array}
$$
with $\theta=\dfrac{|x-y|}{t^{\frac{1}{2s}}}$.
Since $q<\dfrac{N+2s}{N+s}<\dfrac{1}{1-s}$, then $(1-s)q<1$ and $\dfrac{N}{2s}+\dfrac{qs-q(N+2s)}{2s}>-1$. Hence
$$
I_1\le C(\O)T^{\frac{N+2s-q(N+s)}{2s}}.
$$
Respect to $I_2$, we have
$$
\begin{array}{rcl}
I_2 &\leq& C  \dint_{0}^{T} \dint_{\Omega} \dint_{\Omega} \dfrac{t^{q\frac{2s-1}{2s}}}{(t^{\frac{1}{2s}}+ |x-y|)^{(N+2s)q}}\,dxdy dt\\ &&\\ &=&  C \Big(
\dint_{0}^{T}  t^{\frac{q(2s-1)+N-(N+2s)q}{2s}}\,dt\Big) \dint_{0}^{\infty} \dfrac{\theta^{N-1}}{(1+\theta)^{q(N+2s)}}\,d\theta.
\end{array}
$$
Since $q<\frac{N+2s}{N+1}$, then $\frac{q(2s-1)+N-(N+2s)q}{2s}> -1$. Hence
$$
I_2\le C(\O)T^{\frac{N+2s-q(N+1)}{2s}}.
$$
Combining the above estimates on $I_1$ and $I_2$, the result follows.
\end{proof}
We start with the following elementary result about the linear problem \eqref{eq:def} without source term.
\begin{Proposition}\label{first11}
Suppose that $f\equiv 0$ and $u_0\in L^\rho(\O)$. If $u$ is the unique weak solution to problem \eqref{eq:def}, then for all $r\ge \rho$ and for all $t>0$, we
have
\begin{equation}\label{sem100}
||u(\cdot,t)||_{L^r(\O)}\le C(\O)t^{-\frac{N}{2s}(\frac{1}{\rho}-\frac{1}{r})}||u_0||_{L^\rho(\O)}.
\end{equation}
\end{Proposition}
\begin{proof}
We have the representation formula for the solution,
$$
u(x,t)=\dint_{\Omega}u_0(y) P_{\Omega} (x,y, t)\,dy.
$$
From Theorem \ref{th1} we obtain that $u\in L^m(\O_T)$ for all $m<\frac{N+2s}{N}$.
To prove the  estimate \eqref{sem100} we take advantage of the linearity of the problem  and use a duality argument.
Let $\phi\in \mathcal{C}^\infty_0(\O)$, then
$$
||u(.,t)||_{L^r(\O)} =\dyle \sup_{\{||\phi||_{L^{r'}(\O)}\le 1\}}\io \phi(x) u(x,t)dx= \dyle
\sup_{\{||\phi||_{L^{r'}(\O_T)}\le 1\}}\io \phi(x) \dint_{\Omega}u_0(y) P_{\Omega} (x,y, t)\,dy \,dx.
$$
Using estimate \eqref{green1}, it holds that
$$
\begin{array}{rcl}
||u(.,t)||_{L^r(\O)} &\le& \dyle \sup_{\{||\phi||_{L^{r'}(\O)}\le 1\}}\iint_{\O\times \O}
|u_0(y)||\phi(x)|\dfrac{t}{(t^{\frac{1}{2s}}+|x-y|)^{N+2s}}\,dy\,dx\\ &= & \dyle \sup_{\{||\phi||_{L^{r'}(\O)}\le 1\}} \iint_{\O\times \O} |\phi(x)||u_0(y)|
H(x-y,t)dydx
\end{array}
$$
with $H(x,\s)$ defined in \eqref{upper-kernel11}. Using Young inequality, we get
$$
||u(.,t)||_{L^r(\O)}\le C\dyle \sup_{\{||\phi||_{L^{r'}(\O)}\le 1\}}||\phi||_{L^{r'}(\O)} ||u_0||_{L^\rho(\O)}||H(., t)||_{L^a(\O)}
$$
with $\frac{1}{r'}+\frac{1}{\rho}+\frac{1}{a}=2$. By a direct computations we reach that
$$
||H(.,t)||^a_{L^a(\O)}\le ||H(.,t)||^a_{L^a(\ren)}= t^{a-\frac{a(N+2s)}{2s}}\dyle \dint_{\ren}
\dfrac{1}{\Big(1+\dfrac{|x|}{t^{\frac{1}{2s}}}\Big)^{a(N+2s)}}\,dx.
$$
Setting $z=\dfrac{x}{t^{\frac{1}{2s}}}$, then $||H(.,t)||^a_{L^a(\O)}\le C t^{a+\frac{N}{2s}-\frac{a(N+2s)}{2s}}.$
Thus
\begin{equation}\label{L-a-H}
||H(., t)||_{L^a(\O)}\le C(\O) t^{\frac{N}{2s}(\frac{1}{a}-1)}=C t^{\frac{N}{2s}(\frac{1}{r}-\frac{1}{\rho})}.
\end{equation}
Hence
$$
||u(.,t)||_{L^r(\O)} \le C(\O) t^{\frac{N}{2s}(\frac{1}{r}-\frac{1}{\rho})}||u_0||_{L^\rho(\O)}.
$$
\end{proof}
Next we  state the  following  compactness result that could be seen as the parabolic extension of the result in \cite{CV2} and which have an interest in  itself.
\begin{Theorem} \label{gradiente}
Assume that $(f,u_0)\in L^1(\O_T)\times L^{1}(\Omega)$. Let $u$ be the unique solution to problem \eqref{eq:def}, then for all $q<\frac{N+2s}{N+1}$,
$$
||u||_{\mathcal{C}([0,T],L^1(\O))}+ ||\nabla u||_{L^{q}(\Omega_T)}\le C(q,\Omega_T)\bigg(||f||_{L^{1}(\Omega_T)}+||u_0||_{L^1(\O)}\bigg).
$$
Moreover, for $q<\frac{N+2s}{N+1}$ fixed, setting $\hat{K}: L^{1}(\Omega_T)\times L^1(\O)\to L^q(0,T; W_{0}^{1,q}(\Omega))$, $\hat{K}(f, u_0)=u$, the unique
solution to problem \eqref{eq:def}, then $\hat{K}$ is a compact operator.
\end{Theorem}
\begin{proof}

Without loss of generality we can assume that the data $u_0,f$ are nonnegative, since  thanks to the linearity of the operator the general case can be obtained by decomposing the datum into its positive
and negative parts and then dealing with two data separately.

Since $(u_0,f)\in L^1(\O)\times L^{1}(\Omega_T)$, then $u\in L^{m}(\Omega_T)$ for all $m<\frac{N+2s}{N}$ (see \cite{LPPS}). From the representation formula, setting $\O_t=\O\times (0,t)$ with $t<T$,  we get
$$
u(x,t)=\dint_{\Omega}u_0(y) P_{\Omega} (x,y, t)\,dy\,+ \dyle \iint_{\O_t} f(y,\sigma) P_{\Omega} (x,y, t-\sigma)\,dy\,d\sigma.
$$
Hence
$$
\begin{array}{rcl}
|\nabla u(x,t)|^q&\leq&
 C(\O_T) \Big(\dint_{\Omega} u_0(y)|\nabla_x P_{\Omega} (x,y, t)|\,dy\,+
 \iint_{\O_t} f(y,\sigma) |\nabla_x P_{\Omega} (x,y,t-\sigma)|\,dy\,d\sigma\Big)^q\\
 \\ &\leq& C(\O_T) \Big(\dint_{\Omega} u_0(y)\frac{|\nabla_x P_{\Omega} (x,y, t)|}{P_{\Omega} (x,y, t)}\,P_{\Omega} (x,y, t) dy\,\\
  \\
&+ & \dyle \iint_{\O_t} f(y,\sigma) \frac{|\nabla_x P_{\Omega} (x,y, t-\sigma)|}{P_{\Omega} (x,y, t-\sigma)}P_{\Omega} (x,y,
t-\sigma)\,dy\,d\sigma\Big)^q\\ \\
&\leq& C(\O_T) \Big(\dint_{\Omega} u_0(y) h(x,y,t){P_{\Omega} (x,y, t)}\,dy\Big)^q\\
  \\
&+ & \dyle C(\O_T)\Big(\iint_{\O_t} f(y,\sigma) h(x,y,t-\s){P_{\Omega} (x,y, t-\sigma)}\,dy\,d\sigma\Big)^q \\ \\ &:=& J_1(x,t)+J_2(x,t)
\end{array}
$$
with $h(x,y,t)=\dfrac{|\nabla_x P_{\Omega} (x,y, t)|}{P_{\Omega} (x,y, t)}\leq C \Big( \dfrac{1}{\delta(x) \wedge t^{\frac{1}{2s}}}\Big)$.\\

To estimate $J_1$ we decompose the integral as follows
\begin{eqnarray*}
J_1(x,t) & = & C \Big(\dint_{\Omega} u_0(y) h(x,y,t){P_{\Omega} (x,y, t)}\, dy\Big)^q\\ \\
& \le & C \Big(\dint_{\{\Omega\cap \{\d(x)>t^{\frac{1}{2s}}\}\}} u_0(y) h(x,y,t){P_{\Omega} (x,y, t)}\,dy\Big)^q \\ & + &
\Big(\dint_{\{\Omega\cap \{\d(x)\le t^{\frac{1}{2s}}\}\}} u_0(y) h(x,y,t){P_{\Omega} (x,y, t)}\, dy\Big)^q\\
&\le & \frac{C}{t^{\frac{q}{2s}}}\Big(\dint\limits_{\{\Omega\cap \{\d(x)>t^{\frac{1}{2s}}\}\}} u_0(y){P_{\Omega} (x,y, t)}\,dy\Big)^q  +
\frac{C}{\d^q(x)}\Big(\dint\limits_{\{\Omega\cap \{\d(x)\le t^{\frac{1}{2s}}\}\}} u_0(y){P_{\Omega} (x,y, t)}\,dy\Big)^q\\
& = & J_{11}(x,t)+ J_{12}(x,t).
\end{eqnarray*}
By estimate \eqref{green1} and H\"older inequality, we obtain that
\begin{eqnarray*}
J_{11}(x,t) & = & \frac{C}{t^{\frac{q}{2s}}}\Big(\dint_{\{\Omega\cap \{\d(x)>t^{\frac{1}{2s}}\}\}} u_0(y){P_{\Omega} (x,y, t)}\,dy\Big)^q \\
& \le & C\Big(\dint_{\{\Omega\cap \{\d(x)>t^{\frac{1}{2s}}\}\}} u_0(y)\frac{t^{1-\frac{1}{2s}}}{(t^{\frac{1}{2s}}+|x-y|)^{N+2s}}\,dy\Big)^q\\
&\le & C||u_0||^{q-1}_{L^1(\O)}\io u_0(y)\frac{t^{q(1-\frac{1}{2s})}}{(t^{\frac{1}{2s}}+|x-y|)^{q(N+2s)}}\,dy.
\end{eqnarray*}
Thus
$$
\iint_{\O_T}J_{11}(x,t) dxdt \le C||u_0||^{q-1}_{L^1(\O)}\io u_0(y)\bigg(\iint_{\O_T}\frac{t^{q(1-\frac{1}{2s})}}{(t^{\frac{1}{2s}}+|x-y|)^{q(N+2s)}}\,dx\,dt\bigg)\,dy.
$$
Setting, $z=\dfrac{x-y}{t^{\frac{1}{2s}}}$, then
\begin{eqnarray*}
\dyle \iint_{\O_T}J_{11}(x,t) dxdt & \le &  C||u_0||^{q}_{L^1(\O)}\int_0^T t^{q(1-\frac{1}{2s})-\frac{q(N+2s)}{2s}+\frac{N}{2s}} \int_{\ren}\frac{dz}{(1+|z|)^{q(N+2s)}}dz\,dt\\
&\le & C||u_0||^{q}_{L^1(\O)}\int_0^T t^{\g_1} dt
\end{eqnarray*}
where $\g_1=q(1-\frac{1}{2s})-\frac{q(N+2s)}{2s}+\frac{N}{2s}=\frac{N}{2s}-q\frac{N+1}{2s}$. Since $q<\frac{N+2s}{N+1}$, then $\g_1>-1$. Thus
$$
\dyle \iint_{\O_T}J_{11}(x,t) dxdt\le CT^{\g_1+1}||u_0||^{q}_{L^1(\O)}.
$$

We deal now with $J_{12}$ which is more involved.

By estimate \eqref{green1}, we obtain that $P_{\Omega} (x,y, t)\le Ct^{-\frac{N}{2s}}$.
Hence $\bigg(t^{\frac{N}{2s}}P_{\Omega} (x,y, t)\bigg)^{q-1}\le C$ and then
\begin{equation}\label{one0}
P^q_{\Omega} (x,y, t)\le Ct^{-(q-1)\frac{N}{2s}}P_{\Omega} (x,y, t)\le \frac{C}{(\d(x))^{(q-1)N}}P_{\Omega} (x,y, t)\:\: \mbox{  if }\d(x)\le t^{\frac{1}{2s}}.
\end{equation}
Therefore using H\"older inequality, it holds that
\begin{equation}\label{newww}
\begin{array}{lll}
J_{12}(x,t)& = & \dfrac{C}{\d^q(x)}\Big(\dint_{\{\Omega\cap \{\d(x)\le t^{\frac{1}{2s}}\}\}} u_0(y){P_{\Omega} (x,y, t)}\,dy\Big)^q\\
&\le & \dfrac{C||u_0||^{q-1}_{L^1(\O)}}{\d^q(x)}\dint_{\{\Omega\cap \{\d(x)\le t^{\frac{1}{2s}}\}\}} u_0(y){P^q_{\Omega} (x,y, t)}\,dy\\
& \le & \dfrac{C||u_0||^{q-1}_{L^1(\O)}}{\d^q(x)}\dint_{\{\Omega\cap \{\d(x)\le t^{\frac{1}{2s}}\}\}} u_0(y)\frac{1}{(\d(x))^{(q-1)N}}P_{\Omega} (x,y, t)dy\\
& \le & C||u_0||^{q-1}_{L^1(\O)}\dint_{\{\Omega\cap \{\d(x)\le t^{\frac{1}{2s}}\}\}} u_0(y)\frac{P_{\Omega} (x,y, t)}{(\d(x))^{q(N+1)-N}}dy.
\end{array}
\end{equation}
Thus
$$
\iint_{\O_T}J_{12}(x,t)dxdt\le C||u_0||^{q-1}_{L^1(\O)}\io u_0(y)\io \frac{1}{(\d(x))^{q(N+1)-N}}\bigg(\int_0^T P_{\Omega} (x,y, t)dt\bigg)dxdy.
$$
Recall that, from \eqref{green0},
$$
\dint_{0}^{\infty} P_{\Omega} (x,y,
t)\,dt=\mathcal{G}_s(x,y),
$$
the Green function of the fractional Laplacian in $\Omega$.
Then
$$
\iint_{\O_T}J_{12}(x,t)dxdt\le C||u_0||^{q-1}_{L^1(\O)}\io u_0(y)\bigg(\io \frac{\mathcal{G}_s(x,y)}{(\d(x))^{q(N+1)-N}}dx\bigg)dy.
$$
Let $\varphi(x):=\dyle\io \frac{\mathcal{G}_s(x,y)}{\d^{q(N+1)-N}}dx$, then $\varphi$ is the unique solution to the problem
\begin{equation}\label{varphi}
\left\{
\begin{array}{rcll}
(-\D)^s \varphi &=& \dfrac{1}{(\d(x))^{q(N+1)-N}} & \text{ in } \O, \\
\varphi &=&0 & \text{ in }(\ren\setminus\O). \\
\end{array}%
\right.
\end{equation}
Since $q(N+1)-N<2s$, from \cite{AP}, see also \cite{Adm}, it follows that $\varphi\in L^\infty(\O)$. Hence we conclude that
$$
\iint_{\O_T}J_{12}(x,t)dxdt\le C||u_0||^{q}_{L^1(\O)}.
$$
Combing the above estimate we deduce that
$$
\iint_{\O_T}J_1(x,t) dxdt \le C||u_0||^{q}_{L^1(\O)}.
$$
We treat now $J_2$. As in the previous estimates, we have
\begin{eqnarray*}
J_2(x,t) & = & \Big(\iint_{\{\O\times (0,t)\}}f(y,\sigma) h(x,y,t-\s){P_{\Omega} (x,y, t-\sigma)}\,dy\,d\sigma\Big)^q\\
&=& \Big(\iint_{\{\O\times (0,t) \cap\{\d(x)>(t-\s)^{\frac{1}{2s}}\}\}\}}f(y,\sigma) h(x,y,t-\s){P_{\Omega} (x,y, t-\sigma)}\,dy\,d\sigma\Big)^q\\
&+ & \dyle \Big(\iint_{\{\O\times (0,t) \cap\{\d(x)\le (t-\s)^{\frac{1}{2s}}\}\}\}}f(y,\sigma) h(x,y,t-\s){P_{\Omega} (x,y, t-\sigma)}\,dy\,d\sigma\Big)^q\\
&\le & C\Big(\iint_{\{\O\times (0,t) \cap\{\d(x)>(t-\s)^{\frac{1}{2s}}\}\}\}}f(y,\sigma) \frac{P_{\Omega} (x,y, t-\sigma)}{(t-\s)^{\frac{1}{2s}}}\,dy\,d\sigma\Big)^q\\
&+ & \frac{C}{\d^q(x)}\dyle \Big(\iint_{\{\O\times (0,t) \cap\{\d(x)\le (t-\s)^{\frac{1}{2s}}\}\}\}}f(y,\sigma){P_{\Omega} (x,y, t-\sigma)}\,dy\,d\sigma\Big)^q\\ \\
&=& J_{21}(x,t) + J_{22}(x,t).
\end{eqnarray*}
Using estimate \eqref{green1} and by H\"older inequality, we get
\begin{eqnarray*}
J_{21}(x,t) &=& \Big(\iint_{\{\O\times (0,t) \cap\{\d(x)>(t-\s)^{\frac{1}{2s}}\}\}\}}f(y,\sigma) \frac{(t-\s)^{1-\frac{1}{2s}}}{((t-\s)^{\frac{1}{2s}}+|x-y|)^{N+2s}} \,dy\,d\sigma\Big)^q\\
& \le  & C||f||^{q-1}_{L^1(\O_T)}\iint_{\{\O\times (0,t)\}\}}f(y,\sigma) \frac{(t-\s)^{q(1-\frac{1}{2s})}}{((t-\s)^{\frac{1}{2s}}+|x-y|)^{q(N+2s)}} \,dy\,d\sigma.
\end{eqnarray*}
Integrating in $\O_T$, following closely the computations used in the estimate of the term $J_{11}$, we deduce that
$$
\dyle \iint_{\O_T}J_{21}(x,t) dxdt\le CT^{\g_1+{{2}}}||f||^{q}_{L^1(\O_T)}
$$
with $\g_1=\dfrac{N}{2s}-q\dfrac{N+1}{2s}>-1$.

Now respect to $J_{22}$, using estimate \eqref{one0} and by H\"older inequality, it follows that
\begin{equation}\label{newww0}
\begin{array}{lll}
J_{22}(x,t)& \le & \dfrac{C||f||^{q-1}_{L^1(\O_T)}}{\d^q(x)}\dyle\iint_{\{\O\times (0,t)\}\}}f(y,\sigma){P^q_{\Omega} (x,y, t-\s)}\,dy,\,d\s\\
& \le & \dfrac{C||f||^{q-1}_{L^1(\O_T)}}{\d^q(x)}
\dyle\iint_{\{\O\times (0,t)\}\}}\frac{f(y,\sigma)}{(\d(x))^{(q-1)N}}P_{\Omega} (x,y, t-\s)dy\,d\s \\
& \le & C||f||^{q-1}_{L^1(\O_T)} \dyle \iint_{\{\O\times (0,t)\}\}}f(y,\sigma)
\frac{P_{\Omega} (x,y, t-\s)}{(\d(x))^{q(N+1)-N}}dy.
\end{array}
\end{equation}
Integrating in $\O_T$ as above, there results that
$$
\iint_{\O_T}J_{22}(x,t)dxdt\le C||f||^{q-1}_{L^1(\O_T)}\iint_{\O_T}f(y,\s)\Big(\io \frac{1}{(\d(x))^{q(N+1)-N}}\int_{\s}^T P_{\Omega} (x,y, t-\s)dt\,dx\Big)dy\,d\sigma.
$$
It is clear that, from \eqref{green0},
$$
\dint_{\s}^{T} P_{\Omega} (x,y,
t-\s)\,dt= \dint_{0}^{T-\s} P_{\Omega} (x,y,
\eta)\,d\eta\le \mathcal{G}_s(x,y).
$$
Thus
$$
\iint_{\O_T}J_{22}(x,t)dxdt\le C||f||^{q-1}_{L^1(\O_T)}\iint_{\O_T}f(y,\s)\,\varphi(y)dy,
$$
where $\varphi$ is the unique solution to problem \eqref{varphi}. Since $q(N+1)-N<2s$, we have that $\varphi\in L^{\infty}(\O)$ and then
$$
\iint_{\O_T}J_{22}(x,t)dxdt\le C||f||^{q}_{L^1(\O_T)}.
$$
Hence
$$
\iint_{\O_T}J_2(x,t) dxdt \le C||f||^{q}_{L^1(\O_T)}.
$$
Therefore we conclude that
\begin{equation}\label{main001}
\begin{array}{lll}
\dyle \Big(\iint_{\O_T}|\nabla u(x,t)|^q\, dx\,dt \Big)^{\frac{1}{q}} & \le & \dyle
\Big(\iint_{\O_T} (J_1(x,t)+J_2(x,t))\, dx\,dt \Big)^{\frac{1}{q}}\\ \\ & \le &
C(\O,T)\bigg(\|f\|_{L^{1}(\Omega_T)}+||u_0||_{L^1(\O)}\bigg).
\end{array}
\end{equation}
Fixed $q_0<\frac{N+2s}{N+1}$, we define
$$\hat{K}: L^{1}(\Omega_T)\times L^1(\O)\to L^{q_0}(0,T; W_{0}^{1,q_0}(\Omega)),$$
by $\hat{K}(f,u_0)=u$, where $u$ is
the unique solution to \eqref{eq:def}. From \eqref{main001} we deduce that $\hat{K}$ is well defined and continuous. Let show that $\hat{K}$ is a compact operator.

Denote $w$ and $\tilde{w}$ the unique solutions to the problems
\begin{equation}\label{ww11}
\left\{
\begin{array}{rcll}
w_t+(-\D^s) w &=&  0 & \text{ in } \O_{T}  , \\ w&=&0 & \text{ in }(\ren\setminus\O) \times (0,T), \\
w(x,0)&=& u_0 & \mbox{  in  }\O,
\end{array}%
\right.
\end{equation}
 and
 \begin{equation}\label{ww111}
\left\{
\begin{array}{rcll}
\tilde{w}_t+(-\D^s) \tilde{w} &=& f  & \text{ in } \O_{T}  , \\ \tilde{w}&=&0 & \text{ in }(\ren\setminus\O) \times (0,T), \\
\tilde{w}(x,0)&=& 0 & \mbox{  in  }\O,
\end{array}%
\right.
\end{equation}
respectively. It is clear that $w+\tilde{w}=u$.

From Proposition \ref{first11}, we know that $w\in L^{m}(\Omega_T)$ for all $m<\frac{N+2s}{N}$ and for all $\theta>1$,
$$
t^{\frac{N}{2s}(1-\frac{1}{\theta})}||w(.,t)||_{L^{\theta}}\le C||u_0||_{L^1(\O)}.
$$
Then
\begin{equation}\label{ff1}
\sup_{t\in [0,T]} t^{\frac{N}{2s}(\theta-1)} \io w^\theta(x,t)dx\le C(\O_T)||u_0||^\theta_{L^1(\O)}.
\end{equation}
Now, going back to the definition of $J_{11}$, we get
$$
J_{11}(x,t)=\frac{C}{t^{\frac{q}{2s}}}\Big(\dint_{\{\Omega\cap \{\d(x)>t^{\frac{1}{2s}}\}\}} u_0(y){P_{\Omega} (x,y, t)}\,dy\Big)^q\le \frac{C}{t^{\frac{q}{2s}}}w^q(x,t).
$$
Choosing $\theta=q$ in \eqref{ff1},
\begin{eqnarray*}
\iint_{\O_T} J_{11}(x,t) dxdt & \le &  C\iint_{\O_T}t^{-\frac{q}{2s}}  w^q(x,t) t^{\frac{N(q-1)}{2s}}  t^{-\frac{N(q-1)}{2s}}dxdt\\
& \le &  C\int_0^T t^{-\frac{q}{2s}-\frac{N(q-1)}{2s}}\bigg(\io w^q(x,t) t^{\frac{N(q-1)}{2s}} dx\bigg)dt.
\end{eqnarray*}
 By hypothesis $q<\dfrac{N+2s}{N+1}$,  that is, $\dfrac{q}{2s}+\dfrac{N(q-1)}{2s}<1$, then since
$$\bigg(\io w^q(x,t) t^{\frac{N(q-1)}{2s}} dx\bigg)\in L^\infty(0,T),$$ there exists  $a>1$ such that $a(\dfrac{q}{2s}+\dfrac{N(q-1)}{2s})<1$ and
\begin{equation}\label{j11}
\iint_{\O_T} J_{11}(x,t) dxdt\le C  \bigg(\int_0^T\bigg(\io w^q(x,t) t^{\frac{N(q-1)}{2s}} dx\bigg)^{a'}dt\bigg)^{\frac{1}{a'}}.
\end{equation}

Respect to the term $J_{21}$, we have
\begin{eqnarray*}
J_{21}(x,t)& = & C\Big(\iint_{\{\O\times (0,t) \cap\{\d(x)>(t-\s)^{\frac{1}{2s}}\}\}\}}f(y,\sigma) \frac{P_{\Omega} (x,y, t-\sigma)}{(t-\s)^{\frac{1}{2s}}}\,dy\,d\sigma\Big)^q\\
&\le & C\Big(\iint_{\{\O\times (0,t) \}}f(y,\sigma)P_{\Omega} (x,y, t-\sigma)\,dy\,d\sigma\Big)^{q-1}\iint_{\{\O\times (0,t)\}}f(y,\sigma) \frac{P_{\Omega} (x,y, t-\sigma)}{(t-\s)^{\frac{q}{2s}}}\,dy\,d\sigma\\
\\
&\le & C {{\tilde{w}}}^{q-1}(x,t) \iint_{\{\O\times (0,t) \}}f(y,\sigma) \frac{P_{\Omega} (x,y, t-\sigma)}{(t-\s)^{\frac{q}{2s}}}\,dy\,d\sigma.
\end{eqnarray*}
Thus integrating in $\O_T$ and by estimate \eqref{green1}, we obtain that
\begin{eqnarray*}
\iint_{\O_T} J_{21}(x,t)dxdt &\le & C   \iint_{\O_T}f(y,\sigma)  \bigg(\iint_{\{\O\times (\s,T) \}}{{\tilde{w}}}(x,t)^{q-1}\frac{(t-\s)^{1-\frac{q}{2s}}}{((t-\s)^{\frac{1}{2s}}-|x-y|)^{N+2s}}dxdt\bigg)dy\,d\sigma.
\end{eqnarray*}
Since $q<\dfrac{N+2s}{N+1}$ and  { using the fact that $s>\frac 12$}, we get the existence of $\frac{2s-1}{N+1}<r<\dfrac{N+2s}{N}$ such that $q<\dfrac{N+2s+2sr}{N+2s+r}$. Hence $r>q-1$ and
$$
\frac{r}{r-(q-1)}\bigg(\frac{2s-q}{2s}-\frac{N+2s}{2s}\bigg)+\frac{N}{2s}>-1.
$$
Therefore, using H\"older inequality
\begin{eqnarray*}
&\dyle \iint\limits_{\O_T} J_{21}(x,t)dxdt\le \\
& \dyle C   \iint\limits_{\O_T}f(y,\sigma)  \bigg(\iint\limits_{\O_T}{{\tilde{w}}}^r(x,t) dxdt\bigg)^{\frac{q-1}{r}}
\bigg(\iint\limits_{\{\O\times (\s,T)\}}\frac{(t-\s)^{(\frac{2s-q}{2s})
\frac{r}{r-(q-1)}}}{((t-\s)^{\frac{1}{2s}}+|x-y|)^{(N+2s)\frac{r}{r-(q-1)}}}dxdt
\bigg)^{\frac{r-(q-1)}{r}}dyd\sigma.
\end{eqnarray*}
Setting $z=\dfrac{x-y}{(t-\s)^{\frac{1}{2s}}}$, it holds that
\begin{eqnarray*}
&\dyle \iint_{\O_T} J_{21}(x,t)dxdt\le \\
& \dyle C  ||{{\tilde{w}}}||^{q-1}_{L^r(\O_T)}\iint_{\O_T}f(y,\sigma) \bigg\{\bigg(\int_{\s}^T(t-\s)^{\gamma_2} \irn \frac{1}{(1+|z|)^{(N+2s)\frac{r}{r-(q-1)}}}
dzdt\bigg)^{\frac{r-(q-1)}{r}}\bigg\}dy d\sigma,
\end{eqnarray*}
where $\g_2=\dfrac{r}{r-(q-1)}\bigg(\dfrac{2s-q}{2s}-\dfrac{N+2s}{2s}\bigg)+\dfrac{N}{2s}$. Since $\g_2>-1$, then
\begin{equation}\label{j21}
\iint_{\O_T} J_{21}(x,t)dxdt\le C  T^{\frac{(r-(q-1))(\g_2+1)}{r}}||f||_{L^1(\O_T)}||{{\tilde{w}}}||^{q-1}_{L^r(\O_T)}.
\end{equation}

Respect to the terms $J_{12}$ and $J_{22}$ we have that

\begin{equation}\label{j12}
J_{12}(x,t)=\frac{C}{\d^q(x)}\Big(\dint_{\{\Omega\cap \{\d(x)\le t^{\frac{1}{2s}}\}\}} u_0(y){P_{\Omega} (x,y, t)}\,dy\Big)^q=C\frac{w^q(x,t)}{\d^q(x)},
\end{equation}
and
\begin{equation}\label{j22}
J_{22}(x,t)=\frac{C}{\d^q(x)}\dyle \Big(\iint_{\{\O\times (0,t) \cap\{\d(x)\le t^{\frac{1}{2s}}\}\}\}}f(y,\sigma){P_{\Omega} (x,y, t-\sigma)}\,dy\,d\sigma\Big)^q=C\frac{\tilde{w}^q(x,t)}{\d^q(x)}.
\end{equation}
Let $\{f_n, u_{n0}\}_n$ be a bounded sequence in $L^{1}(\Omega_T)\times L^1(\O)$ and define $u_n=\hat{K}(f_n, u_{n0})$. Using the previous
estimates, it follows that, for all $q<\dfrac{N+2s}{N+1}$,
$$
||\nabla u_n||_{L^{q}(\Omega_T)}\le C(q,\Omega)(||f_n||_{L^{1}(\Omega_T)}+||u_{n0}||_{L^{1}(\Omega)})\le C. $$
Hence,  there exists $u\in L^q(0,T; W_{0}^{1,q}(\Omega))$ for all $q<\dfrac{N+2s}{N+1}$ such that, up to a subsequence, $u_n\rightharpoonup u$ weakly in $L^q(0,T; W_{0}^{1,q}(\Omega))$, $u_n\to u$ strongly in $L^\s(\Omega_T)$ for all $\s<\dfrac{N+2s}{N}$ and $u_n\to u$ a.e in $\O_T$.
Fixing the above subsequence, we define $(w_n,\tilde{w}_n)$ be the solutions to problems \eqref{ww11} and \eqref{ww111} with data $u_{n0}$ and $f_n$ respectively. Then the sequences $\{w^q_n(x,t) t^{\frac{N(q-1)}{2s}}\}_n$, $\{w_n\}_n$ and  $\{\dfrac{\tilde{w}_n}{\d}\}_n$ are bounded in $L^\infty((0,T);L^\s(\O))$, $L^r(\O_T)$ and in $L^\s(\O_T)$ for all $\s<\dfrac{N+2s}{N+1}$ and $r<\dfrac{N+2s}{N}$ respectively. Hence using Vitali's lemma we deduce that the sequences $\{w^q_n(x,t) t^{\frac{N(q-1)}{2s}}\}_n$, $\{w_n\}_n$ and $\{\dfrac{\tilde{w}_n}{\d}\}_n$ are strongly converging in $L^a((0,T);L^\s(\O))$, in $L^r(\O_T)$ and in $L^\s(\O_T)$ for all $\s<\dfrac{N+2s}{N+1}$, for all $r<\dfrac{N+2s}{N}$ and for all $a>1$.

By the linearity of the operator, it follows that $(u_i-u_j)$
solves
\begin{equation}\label{eq:def00}
\left\{
\begin{array}{rcll}
(u_i-u_j)_t+(-\D^s) (u_i-u_j)&=& \dyle f_i-f_j  & \text{ in } \O_{T}, \\ u_i-u_j&=&0 & \text{ in }(\ren\setminus\O) \times (0,T), \\ (u_i-u_j)(x,0)&=&
u_{i0}-u_{j0} & \mbox{  in  }\O.
\end{array}%
\right.
\end{equation}
Going back to the first formula in \eqref{main001} and by estimates \eqref{j11}, \eqref{j21}, \eqref{j12}, \eqref{j22}, we get
\begin{eqnarray*}
& \dyle\Big(\iint_{\O_T}|\nabla (u_i-u_j)(x,t)|^q\, dx\,dt \Big)^{\frac{1}{q}} \le \\
& C(\O,T)\dyle \bigg(\bigg(\int_0^T\bigg(\io |w_i(x,t)-w_j(x,t)|^q t^{\frac{N(q-1)}{2s}} dx\bigg)^{a'}dt\bigg)^{\frac{1}{a'q}} +
\|f_i-f_j\|_{L^{1}(\Omega_T)}||\tilde{w}_i-\tilde{w}_j||^{q-1}_{L^r(\O_T)} \\
& + ||\dfrac{w_i-w_j}{\d}||_{L^q(\O_T)} + ||\dfrac{\tilde{w}_i-\tilde{w_j}}{\d}||_{L^q(\O_T)}\bigg).
\end{eqnarray*}
Letting $i,j\to \infty$, it holds that
$$
\Big(\iint_{\O_T}|\nabla (u_i-u_j)(x,t)|^q\, dx\,dt \Big)^{\frac{1}{q}}\to 0.
$$
Then the operator $\hat{K}$ is compact.
\end{proof}

\

\begin{remark}\label{mainrr}

\

\begin{enumerate}

\item Thanks to the above computations, we can prove that the constant $C(T,\O)$, that appears in estimate \eqref{main001} satisfies $C(T,\O)\to 0$ as $T\to
    0$. This fact will be used below in order to show existence result
    for problem \eqref{grad} using a \textit{Fixed Point Theorem}.
    \item Using an approximation argument and by the linearity of the operator, we can prove that the
    result of Theorem \ref{gradiente} holds if $f$ is a bounded Radon measure.
    \item It is worthy to point out that the same arguments give the elliptic case by a slightly different method as in \cite{CV2}.
\end{enumerate}
\end{remark}
Following the same representation argument as above we get the next technical regularity result.
\begin{Proposition}
Assume that $(f,u_0)\in L^1(\O_T)\times L^1(\O)$ and let $u$ be the unique solution to the problem \eqref{eq:def}, then $T_k(u)\in L^\s(0,T;W^{1,\s}_0(\O))$ for
all $1\le \s<2s$. Moveover we have
$$
||T_k(u)||^\s_{L^\s(0,T;W^{1,\s}_0(\O))}\le C(\O,T)k^{\s-1}(||u_0||_{L^1(\O)}+||f||_{L^1(\O_T)}).
$$
In addition, if $u_n=\hat{K}(f_n, u_{n0})$, then, up to a subsequence, it follows that $T_k(u_n)\to T_k(u)$ strongly in $L^\s(0,T;W^{1,\s}_0(\O))$, for all $1\le
\s<2s$.
\end{Proposition}
\begin{proof}
Without loss of generality we can assume that $f\ge 0$ in $\O_T$ and $u_0\ge 0$ in $\O$. Fix $1<\rho<2s$, as in the proof of Theorem \ref{gradiente}, we have
$$
\begin{array}{rcl}
& & \dyle |\nabla u(x,t)|^\rho \leq C\Big(\dint_{\Omega}\, u_0(y) P_{\Omega} (x,y, t) \,dy\Big)^{\frac{\rho}{\rho'}} \Big(\dint_{\Omega} h^\rho(x,y,t) |u_0(y)|P_{\Omega} (x,y,
t)\,dy\Big)\\  \\ &+ &  \dyle C\Big(\iint_{\O_t}\, f(y,\s) P_{\Omega} (x,y, t-\sigma) \,dy\,d\sigma \Big)^{\frac{\rho}{\rho'}} \Big(
\iint_{\O_t} h^\rho(x,y,t-\s)f(y,\s) P_{\Omega} (x,y, t-\sigma)\,dy\,d\sigma\Big)\\ \\ &\le & \dyle C\, u^{\frac{\rho}{\rho'}}(x,t) \Big(\dint_{\Omega} h^\rho(x,y,t)
|u_0(y)|P_{\Omega} (x,y, t)\,dy  +  \iint_{\O_t} h^\rho(x,y,t-\s) f(y,\s) P_{\Omega} (x,y, t-\sigma)\,dy\,d\sigma\Big).
\end{array}
$$
Thus,
$$
|\n T_k(u)|^{\rho}\le C\, k^{\rho -1}\Big(\dint_{\Omega} h^\rho(x,y,t) |u_0(y)|P_{\Omega} (x,y, t)\,dy  +  \iint_{\O_t}h^\rho(x,y,t-\s) f(y,\s) P_{\Omega} (x,y,
t-\sigma)\,dy\,d\sigma\Big)\chi_{\{u<k\}}.
$$
Hence
\begin{equation}\label{mainii}
\begin{array}{lll}
&\,&\dyle \iint_{\O_T}|\nabla T_k(u(x,t))|^\rho\, dx\,dt \\
&\le & \dyle C(\O,T)\, k^{\rho -1}\Big(\io u_0(y) \bigg(\iint_{\O_T}h^\rho(x,y,t)P_{\Omega} (x,y,
t)\,dxdt\bigg)dy\\ &+& \dyle \iint_{\O_T}f(y,\s)\bigg(\int_\s^T\io h^\rho(x,y,t-\s) P_{\Omega} (x,y, t-\s)\,dx\,dt\bigg)dy\\ &\le & Ck^{\rho-1}( J_1+J_2).
\end{array}
\end{equation}
Recall that $h(x,y,t)\leq C \Big( \dfrac{1}{\delta(x) \wedge t^{\frac{1}{2s}}}\Big)$, hence using the fact that $\rho<2s$, we reach that
\begin{eqnarray*}
J_1&= & \dyle \io u_0(y) \bigg(\iint_{\{\O_T\cap \{\d(x)\ge t^{\frac{1}{2s}}\}\}}h^\rho P_{\Omega} (x,y,
t)\,dxdt\bigg)dy \\ & + & \dyle  \io u_0(y) \bigg(\iint_{\{\O_T\cap \{\d(x)< t^{\frac{1}{2s}}\}\}}h^\rho(x,y,t) P_{\Omega} (x,y,
t)\,dxdt\bigg)dy\\ &= & I_1+I_2.
\end{eqnarray*}
Now, by Lemma \ref{estimmm}, it holds that
$$
I_1\leq C \io u_0(y)\iint_{\O_T}\dfrac{1}{t^{\frac{\rho}{2s}}}\dfrac{t}{ (t^{\frac{1}{2s}}+|x-y|) ^{N+2s}}\,\,dxdt\, dy.
$$
Setting, $z=\dfrac{x-y}{t^{\frac{1}{2s}}}$, it follows that
\begin{eqnarray*}
I_1 & \le & C ||u_0||_{L^1(\O)}\dint_{0}^{T} t^{1-\frac{\rho}{2s}-\frac{N+2s}{2s}+\frac{N}{2s}}dt\dint_{\ren} \frac{dz}{(1+|z|)^{(N+2s)}}\\
& \le &  C||u_0||_{L^1(\O)}\dint_{0}^{T} t^{-\frac{\rho}{2s}}dt.
\end{eqnarray*}
Since $\frac{\rho}{2s}<1$, then
$$ I_1\le
CT^{1-\frac{\rho}{2s}}||u_0||_{L^1(\O)}.$$
We deal now with $I_2$. We have
\begin{eqnarray*}
I_2 & \leq & C \io u_0(y)\iint_{\O_T}\dfrac{1}{\d^\rho(x)}\,P_{\Omega} (x,y,
t)\,dxdt\, dy\\
&\le & C \io u_0(y)\dint_{\Omega} \dfrac{1}{\d^\rho(x)}\,(\dint_{0}^{\infty} P_{\Omega} (x,y,
t)\,dt)dx\, dy.
\end{eqnarray*}
Recall that
$$
\dint_{0}^{\infty} P_{\Omega} (x,y,
t)\,dt=\mathcal{G}_s(x,y),
$$
then
$$
I_2\leq C \io u_0(y)\iint_{\O_T}\dfrac{1}{\d^\rho(x)}\,\mathcal{G}_s(x,y) dx\, dy= C \io u_0(y)\varphi (y) dy,
$$
where $\varphi$ is the solution to problem \eqref{varphi}. Since $\rho<2s$, then $\varphi\in L^\infty(\O)$, hence
$$ I_2\le
C||u_0||_{L^1(\O)}.$$
Hence we conclude that $J_1\le
C(\O,T)||u_0||_{L^1(\O)}.$

In the same way we obtain that
$$
J_2\le C(\O,T)||f||_{L^1(\O_T)}.
$$
Therefore
$$
\iint_{\O_T}|\nabla T_k(u(x,t))|^\s\, dx\,dt\le\dyle C(\O,T)\, k^{\s -1}(||u_0||_{L^1(\O)}+||f||_{L^1(\O_T)}).
$$
Define $u_n=\hat{K}(f_n, u_{n0})$, then according with the proof of Theorem \ref{gradiente}, we know that, up to a subsequence, $u_n\to u$ strongly in $L^q(0,T;
W_{0}^{1,q}(\Omega))$ for all $q<\dfrac{N+2s}{N+1}$. Now, using the fact that the sequence $\{T_k(u_n)\}_n$ is bounded in $L^\s(0,T;W^{1,\s}_0(\O))$ for all $1\le
\s<2s$ and using Vitali's Lemma, it follows that $T_k(u_n)\to T_k(u)$ strongly in $L^\s(0,T;W^{1,\s}_0(\O))$.
\end{proof}
\begin{remark}
If $u_0=0$ and $f\in L^m(\O_T)$ with $m>1$, we can improve the regularity results obtained previously. Notice that, related to the
fractional Laplacian, a kind of regularity holds in the local gradient where we are close to the boundary. This can be seen in the explicit example
$w(x)=(1-|x|^2)_+^s$ which solves
$$
(-\D)^s w=1 \mbox{  in  }B_1(0) \mbox{  and  }w=0 \mbox{  in  }\ren\backslash B_1(0).
$$
However, as it was observed in \cite{AP}, we can show a complete regularity schema using in a suitable weighted Sobolev space. The main tool will be a universal control of the term $\dfrac{u}{\d^s}$ that holds for any $s\in (0,1)$ and without using the classical Hardy-Sobolev inequality.
\end{remark}
Before considering the  case $m>1$, we enunciate the next regularity result that will be used through the paper and that clarifies the regularity of the solution in the space  for fixed time.
\begin{Proposition}\label{more-tregu}
Suppose that $f\in L^1(\O_T)$ and $u_0=0$. Let $u$ be the unique weak solution to problem \eqref{eq:def}, then for all $1<q<\dfrac{N+2s}{N+1}$ and for  all $\eta>0$,
\begin{equation}\label{time-regu}
\io |\n u(x,t)|^q dx \le C(\O_T)||f||^{q-1}_{L^1(\O_t)}\int_0^t\io |f(y,\s)|(t-\s)^{\hat{\g}-\eta} dy\,d\sigma,
\end{equation}
where $\hat{\g}:=\dfrac{N}{2s}-q\dfrac{N+1}{2s}\in (-1,0)$. In particular, we obtain that for all $\eta>0$,
\begin{equation}\label{time-regu1}
\bigg(\io |\n u(x,t)|^q dx\bigg)^{\frac{1}{q}}\le C(\O_T)\bigg(\int_0^t\io |f(y,\s)|(t-\s)^{\hat{\g}-\eta} dy\,d\sigma\bigg)^{\frac{1}{q}}.
\end{equation}
\end{Proposition}
\begin{proof}
From Theorem \ref{gradiente}, we know that $u\in L^q((0,T), W^{1,q}_0(\O))$ for all $q<\frac{N+2s}{N+1}$. Then we have
$$
u(x,t)=\dint_{0}^{t} \dint_{\Omega} f(y,\sigma) P_{\Omega} (x,y, t-\sigma)\,dy\,d\sigma.
$$
Hence
\begin{eqnarray*}
|\nabla u(x,t)| & \leq &  \dyle \dint\limits_{0}^{t} \dint_{\Omega}| f(y,\sigma)| \frac{|\nabla_x P_{\Omega} (x,y, t-\sigma)|}{P_{\Omega} (x,y, t-\sigma)}P_{\Omega} (x,y, t-\sigma)\,dy\,d\sigma\\&\le & C\iint\limits_{\{\O\times (0,t) \cap\{\d(x)>(t-\s)^{\frac{1}{2s}}\}\}\}}|f(y,\sigma)| \frac{P_{\Omega} (x,y, t-\sigma)}{(t-\s)^{\frac{1}{2s}}}\,dy\,d\sigma\\
&+ & \dyle\Bigg(\iint\limits_{\{\O\times (0,t) \cap\{\d(x)\le (t-\s)^{\frac{1}{2s}}\}\}\}}|f(y,\sigma)|{P_{\Omega} (x,y, t-\sigma)}\,dy\,d\sigma\Bigg)\,\frac{C}{\d(x)}.
\end{eqnarray*}
Therefore, fixing $1<q<\dfrac{N+2s}{N+1}$, we reach that
\begin{equation}\label{vvv0}
\begin{array}{lll}
|\nabla u(x,t)|^q &\le & \dyle C\Big(\iint\limits_{\{\O\times (0,t) \cap\{\d(x)>(t-\s)^{\frac{1}{2s}}\}\}\}}|f(y,\sigma)| \frac{P_{\Omega} (x,y, t-\sigma)}{(t-\s)^{\frac{1}{2s}}}\,dy\,d\sigma\Big)^q\\
&+ & \dyle \frac{C}{\d^q(x)}\dyle \Big(\iint\limits_{\{\O\times (0,t) \cap\{\d(x)\le (t-\s)^{\frac{1}{2s}}\}\}\}}|f(y,\sigma)|{P_{\Omega} (x,y, t-\sigma)}\,dy\,d\sigma\Big)^q \\  \\  &=& I_1(x,t) + I_{2}(x,t).
\end{array}
\end{equation}
Using estimate \eqref{upper-kernel} and by H\"older inequality, we get
\begin{eqnarray*}
I_{1}(x,t) &=& \Big(\iint\limits_{\{\O\times (0,t) \cap\{\d(x)>(t-\s)^{\frac{1}{2s}}\}\}\}}|f(y,\sigma)| \frac{(t-\s)^{1-\frac{1}{2s}}}{((t-\s)^{\frac{1}{2s}}+|x-y|)^{N+2s}} \,dy\,d\sigma\Big)^q\\
& \le  & C||f||^{q-1}_{L^1(\O_t)}\iint\limits_{\{\O\times (0,t)\}\}}|f(y,\sigma)| \frac{(t-\s)^{q(1-\frac{1}{2s})}}{((t-\s)^{\frac{1}{2s}}+|x-y|)^{q(N+2s)}} \,dy\,d\sigma.
\end{eqnarray*}
Integrating in $\O$, following the same change of variable as in the proof of Theorem \ref{gradiente},
$$
\begin{array}{lll}
\dyle \io I_{1}(x,t) dx & \le & \dyle C||f||^{q-1}_{L^1(\O_t)}\int_0^t\io |f(y,\sigma)|\io \frac{(t-\s)^{q(1-\frac{1}{2s})}}{((t-\s)^{\frac{1}{2s}}+|x-y|)^{q(N+2s)}} dx \,dy\,d\sigma\\
 &\le & \dyle C||f||^{q-1}_{L^1(\O_t)}\int_0^t\io |f(y,\sigma)|(t-\s)^{\frac{N}{2s}-q\frac{N+1}{2s}}dy\,d\sigma.
\end{array}
$$
Setting $\hat{\g}:=\dfrac{N}{2s}-q\dfrac{N+1}{2s}\in (-1,0)$, then
\begin{equation}\label{ii1}
\dyle \io I_{1}(x,t) dx \le \dyle C||f||^{q-1}_{L^1(\O_t)}\int_0^t\io |f(y,\sigma)|(t-\s)^{\hat{\g}}dy\,d\sigma.
\end{equation}

We treat now $I_2$. As above, we get
\begin{eqnarray*}
I_{2}(x,t)& \le & \frac{C||f||^{q-1}_{L^1(\O_t)}}{\d^q(x)}\iint_{\O_t}|f(y,\sigma)|{P^q_{\Omega} (x,y, t-\s)}\,dy,\,d\s\\
& \le & \frac{C||f||^{q-1}_{L^1(\O_t)}}{\d^q(x)}
\iint_{\{\O_t\cap \{|x-y|<\frac 12 \d(x)\}\}}|f(y,\sigma)|P^q_{\Omega} (x,y, t-\s)dy\,d\s \\
& + & \frac{C||f||^{q-1}_{L^1(\O_t)}}{\d^q(x)} \iint_{\{\O_t \cap\{|x-y|\ge \frac 12 \d(x)\}\}}|f(y,\sigma)|
P^q_{\Omega} (x,y, t-\s) dy \,d\sigma=I_{21}(x,t)+ I_{22}(x,t).
\end{eqnarray*}
We begin by estimating $I_{21}$. Notice that $|\d(y)-\d(x)|\le |x-y|$, then $\d(y)\le \dfrac 32 \d(x)$. Thus $\dfrac{\d(y)}{(t-\s)^{\frac{1}{2{ s}}}}\le C(\O_T)$. Hence
using \eqref{upper-kernel},
\begin{eqnarray*}
P^q_{\Omega} (x,y, t-\s) & \le & C(\O_T)\frac{(t-\s)^q}{((t-\s)^{\frac{1}{2s}}+|x-y|)^{q(N+2s)}}.
\end{eqnarray*}
Since $|x-y|\le \frac 13\d(x)$, then
$$
\frac{P^q_{\Omega} (x,y, t-\s)}{\d^q(x)}\le  C(\O_T)\frac{(t-\s)^q}{|x-y|^q((t-\s)^{\frac{1}{2s}}+|x-y|)^{q(N+2s)}}.
$$
Going back to the definition of $I_{21}$, setting $z=\dfrac{x-y}{(t-\s)^{\frac{1}{2s}}}$ and integrating in $\O$, there results that
\begin{eqnarray*}
\io I_{21}(x,t)dx & \le & C(\O_T)||f||^{q-1}_{L^1(\O_t)}\int_0^t\io |f(y,\s)|\bigg(\io \frac{(t-\s)^q}{|x-y|^q((t-\s)^{\frac{1}{2s}}+|x-y|)^{q(N+2s)}}dx\bigg)\,dy\, d\sigma\\
& \le & C(\O_T)||f||^{q-1}_{L^1(\O_t)}\int_0^t\io |f(y,\s)|(t-\s)^{\hat{\g}}\bigg(\irn\frac{dz}{|z|^q(1+|z|)^{q(N+2s)}}\bigg) dy\,d\sigma,
\end{eqnarray*}
where, as above, $\hat{\g}=\dfrac{N}{2s}-q\dfrac{N+1}{2s}\in (-1,0)$. Using the fact that $q<N$ we deduce that
\begin{equation}\label{ii21}
\io I_{21}(x,t)dx \le C(\O_T)||f||^{q-1}_{L^1(\O_t)}\int_0^t\io |f(y,\s)|(t-\s)^{\hat{\g}}dy\,d\sigma.
\end{equation}

Respect to $I_{22}$, since $\d(x)\le 2 |x-y|$, we get $\d(y)\le C|x-y|$, then, in this case
$$
P^q_{\Omega} (x,y, t-\s)\le C(\O_T) \Big( 1\wedge \frac{\delta^s(x)}{\sqrt{(t-\s)}}\Big)^q\times \Big( 1\wedge \frac{\delta^s(y)}{\sqrt{(t-\s)}}\Big)^q
\frac{(t-\s)^q}{((t-\s)^{\frac{1}{2s}}+|x-y|)^{q(N+2s)}}.
$$
Since, for all $\theta\in (0,1)$,
$$
\Big( 1\wedge \frac{\delta^s(x)}{\sqrt{(t-\s)}}\Big)^q\le \Big( 1\wedge \frac{\delta^s(x)}{\sqrt{(t-\s)}}\Big)^{q\theta}\mbox{  and   }
\Big( 1\wedge \frac{\delta^s(y)}{\sqrt{(t-\s)}}\Big)^q \le \Big( 1\wedge \frac{\delta^s(y)}{\sqrt{(t-\s)}}\Big)^{q\theta}.
$$
Choosing $\theta=\dfrac{1}{q}$, we deduce that
\begin{eqnarray*}
P^q_{\Omega} (x,y, t-\s) & \le & C(\O_T) \frac{\delta^s(x)}{\sqrt{(t-\s)}}\times \frac{\delta^s(y)}{\sqrt{(t-\s)}}
\frac{(t-\s)^q}{((t-\s)^{\frac{1}{2s}}+|x-y|)^{q(N+2s)}}\\
 &\le & C(\O_T)\delta^s(x)\delta^s(y)\frac{(t-\s)^{q-1}}{((t-\s)^{\frac{1}{2s}}+|x-y|)^{q(N+2s)}}\\
 &\le & \frac{C(\O_T)\delta^s(x)\delta^s(y)}{|x-y|^{N}} \frac{(t-\s)^{q-1}}{((t-\s)^{\frac{1}{2s}}+|x-y|)^{q(N+2s)-N}}\\
&\le & \frac{C(\O_T)\delta^s(x)\delta^s(y)}{|x-y|^{N}}(t-\s)^{-\frac{N}{2s}(q-1)-1}.\\
\end{eqnarray*}
Thus, for $\eta\in (0,1)$ small enough,
\begin{eqnarray*}
\frac{P^q_{\Omega} (x,y, t-\s)}{\d^q(x)} &\le &   \frac{C(\O_T)\delta^s(x)\delta^s(y)}{|x-y|^{N}} \frac{1}{(\d(x))^{q +2s-q-2s\eta}}(\d(x))^{2s-q-2s \eta}(t-\s)^{-\frac{N}{2s}(q-1)-1}\\
& \le &  \frac{C(\O_T)\delta^s(x)\delta^s(y)}{|x-y|^{N}} \frac{1}{(\d(x))^{2s(1-\eta)}}
(t-\s)^{\frac{2s-q-2s\eta}{2s}}(t-\s)^{-\frac{N}{2s}(q-1)-1}\\
& \le & \frac{C(\O_T)\delta^s(x)\delta^s(y)}{|x-y|^{N}} \frac{1}{(\d(x))^{2s(1-\eta)}}(t-\s)^{\frac{N}{2s}-\frac{q(N+1)}{2s}-\eta}.
\end{eqnarray*}
From \eqref{green00}, we deduce that, in this case,
 $$
\mathcal{G}_s(x,y)\simeq \frac{\d^s(x)\d^s(y)}{|x-y|^{N}}.
$$
Hence
$$
\frac{P^q_{\Omega} (x,y, t-\s)}{\d^q(x)}\le C(\O_T)\frac{\mathcal{G}_s(x,y)}{(\d(x))^{2s(1-\eta)}}(t-\s)^{\hat{\g}-\eta}.
$$
Therefore
\begin{eqnarray*}
\io I_{22}(x,t)dx & \le & C(\O_T)||f||^{q-1}_{L^1(\O_t)}\int_0^t\io |f(y,\s)|(t-\s)^{\hat{\g}-\eta}\bigg(\io \frac{\mathcal{G}_s(x,y)}{(\d(x))^{2s(1-\eta)}}dx\bigg)\,dy\, d\sigma\\
& \le & C(\O_T)||f||^{q-1}_{L^1(\O_t)}\int_0^t\io |f(y,\s)|(t-\s)^{\hat{\g}-\eta}\varphi(y) dy\,d\sigma,
\end{eqnarray*}
where $\varphi(y):=\dyle\io \frac{\mathcal{G}_s(x,y)}{(\d(x))^{2s(1-\eta)}}dx$. It is clear that $\varphi$ solves the problem
\begin{equation*}
\left\{
\begin{array}{rcll}
(-\D)^s \varphi &=& \dfrac{1}{(\d(x))^{2s(1-\eta)}} & \text{ in } \O, \\
\varphi &=&0 & \text{ in }(\ren\setminus\O). \\
\end{array}%
\right.
\end{equation*}
Since $2s(1-\eta)<2s$, then $\varphi\in L^\infty(\O)$. Thus, for all $\eta>0$,
\begin{equation}\label{ii22}
\io I_{22}(x,t)dx \le C(\O_T)||f||^{q-1}_{L^1(\O_t)}\int_0^t\io |f(y,\s)|(t-\s)^{\hat{\g}-\eta}dy\,d\sigma.
\end{equation}

Combing estimates \eqref{ii1},\eqref{ii21},\eqref{ii22}, going back to \eqref{vvv} and integrating in $\O$, we conclude that for all $\eta>0$,
$$
\dyle \io |\n u(x,t)|^qdx\le C(\O_T) ||f||^{q-1}_{L^1(\O_t)} \int_0^t\io |f(y,\sigma)|\bigg((t-\s)^{\hat{\g}-\eta}+ (t-\s)^{\hat{\g}}\bigg)dy\,d\sigma.
$$
Hence
$$
\dyle \io |\n u(x,t)|^qdx\le C(\O_T) ||f||^{q-1}_{L^1(\O_t)} \int_0^t\io |f(y,\sigma)|(t-\s)^{\hat{\g}-\eta} dy\,d\sigma.
$$
Now using the fact that
$$
||f||_{L^1(\O_t)} \le C(\O_T)\int_0^t\io |f(y,\sigma)|(t-\s)^{\hat{\g}-\eta}dy\,d\sigma,
$$
we conclude that
$$
\Big(\dyle \io |\n u(x,t)|^qdx\Big)^{\frac{1}{q}}\le C(\O_T) \int_0^t\io |f(y,\sigma)|(t-\s)^{\hat{\g}-\eta}dy\,d\sigma.
$$
\end{proof}

If $m>1$, we can prove a regularity result in a suitable weighted Sobolev space whose weight is a power of the distance to the boundary.

This result will be a consequence of the next two Theorems.
\begin{Theorem}\label{hardy0}
Assume that $u_0\equiv0$,   $f\in L^m(\O_T)$ for some $m>1$ and let $u$ be the unique weak solution to problem \eqref{eq:def}, then $\dfrac{u}{\d^s}\in L^\theta(\O_T)$ for
all $\theta>1$ such that $\dfrac{1}{\theta}>\dfrac{1}{m}-\dfrac{s}{N+2s}$. Moreover,
\begin{equation}\label{hardyeq}
\Big|\Big|\frac{u}{\d^s}\Big|\Big|_{L^\theta(\O_T)}\le C ||f||_{L^m(\O_T)},
\end{equation}
with
$$
\left\{
\begin{array}{lll}
\theta &<& \infty \mbox{  if  }m\ge \dfrac{N+2s}{s},\\
&\,&\\
\theta &< & \dfrac{m (N+2s)}{N+2s-ms}\mbox{  if  } m<\dfrac{N+2s}{s}.
\end{array}
\right.
$$
\end{Theorem}
\begin{proof}
Without loss of generality, we can assume that $f\gvertneqq 0$, hence $u\gvertneqq 0$ in $\ren\times (0,T)$. By the representation formula we have that
$$
u(x,t)=\dint_{0}^{t} \dint_{\Omega} f(y,\sigma) P_{\Omega} (x,y, t-\sigma)\,dy\,d\sigma,
$$
then  using the properties of $P_\O$, and \eqref{upper-kernel} it holds that
$$
\begin{array}{rcl}
\dfrac{u(x,t)}{\d^s(x)} &\leq& C \dint_{0}^{t}\dint_{\Omega}f(y,\s) \dfrac{(t-\s)^{{\frac 12}}}{((t-\s)^{\frac{1}{2s}}+|x-y|)^{N+2s}}\,dy\,d\sigma.\\
\end{array}
$$
As above, consider $\phi\in \mathcal{C}^\infty_0(\O_T)$, then
$$
\begin{array}{rcl}
\Big|\Big|\dfrac{u}{\d^s}\Big|\Big|_{L^\theta(\O_T)} &=& \dyle \sup_{\{||\phi||_{L^{\theta'}(\O_T)}\le 1\}} \iint_{\O_T} \phi(x,t) \frac{u(x,t)}{\delta^s(x)}dxdt\\ &\le & \dyle
\sup_{\{||\phi||_{L^{\theta'}(\O_T)}\le 1\}} \iint_{\O_T}|\phi(x,t)|\dint_{0}^{t}\dint_{\Omega}f(y,\s) \dfrac{(t-\s)^{{\frac 12}}}{((t-\s)^{\frac{1}{2s}}+|x-y|)^{N+2s}}\,dy\,d\sigma dxdt\\ &\le & \dyle \sup_{\{||\phi||_{L^{\theta'}(\O_T)}\le 1\}} \int_0^T\int_0^t \io \io |\phi(x,t)|
H(x-y,t-\s)f(y,\s)dydx d\s dt,
\end{array}
$$
where
$$
H(|x-y|,\s)=\dfrac{(t-\s)^{{\frac 12}}}{((t-\s)^{\frac{1}{2s}}+|x-y|)^{N+2s}}.
$$
Using Young inequality, it holds that
\begin{equation}\label{tt}
\begin{array}{rcl}
\Big|\Big|\dfrac{u}{\d^s}\Big|\Big|_{L^\theta(\O_T)} &\le & C\dyle \sup_{\{||\phi||_{L^{\theta'}(\O_T)}\le 1\}} \int_0^T ||\phi(.,t)||_{L^{\theta'}(\O)} \int_0^t
||f(.,\s)||_{L^m(\O)}||H(., t-\s)||_{L^a(\O)}d\s dt,
\end{array}
\end{equation}
with $\dfrac{1}{\theta'}+\dfrac{1}{m}+\dfrac{1}{a}=2$.

As in the computation of the term defined in \eqref{upper-kernel} and by \eqref{L-a-H}, we get
$$
||H(., t-\s)||_{L^a(\O)}\le C(t-\s)^{-{{\frac 12}}+\frac{N}{2s a}-\frac{N}{2s}}.
$$
Substituting in \eqref{tt},
$$
\begin{array}{rcl}
\Big|\Big|\dfrac{u}{\d^s}\Big|\Big|_{L^\theta(\O_T)} &\le & C\dyle \sup_{\{||\phi||_{L^{\theta'}(\O_T)}\le 1\}} \int_0^T ||\phi(.,t)||_{L^{\theta'}(\O)} \int_0^t
||f(.,\s)||_{L^m(\O)} (t-\s)^{-{{\frac 12}}+\frac{N}{2s a}-\frac{N}{2s}} d\s dt
\\[5mm]
&\le & { C\dyle \sup_{\{||\phi||_{L^{\theta'}(\O_T)}\le 1\}} \int_0^T \int_0^T  ||\phi(.,t)||_{L^{\theta'}(\O)} ||f(.,\s)||_{L^m(\O)} |t-\s|^{-{{\frac 12}}+\frac{N}{2s a}-\frac{N}{2s}} d\s dt.}
\end{array}
$$
Setting
$$
\hat{H}(|t-\s|)=|t-\s|^{-{{\frac 12}}+\frac{N}{2s a}-\frac{N}{2s}},
$$
then using again Young inequality, we get
$$
\Big|\Big|\dfrac{u}{\d^s}\Big|\Big|_{L^\theta(\O_T)}\le C\dyle \sup_{\{||\phi||_{L^{\theta'}(\O_T)}\le 1\}} ||\phi(.,t)||_{L^{\theta'}(\O_T)} ||f||_{L^m(\O_T)}\bigg(\int_0^T \hat{H}^\g (t)dt\bigg)^{\frac{1}{\g}},
$$
where $\g\ge 1$ and $\dfrac{1}{\theta'}+\dfrac{1}{m}+\dfrac{1}{\g}=2$. It is clear that $a=\g$.

Hence
$$
\int_0^T \hat{H}^\g (t)=\int_0^T t^{\gamma( -{{\frac 12}}+\frac{N}{2s a}-\frac{N}{2s})}dt.
$$
The previous integral is finite if and only if
$$
\gamma( -{{\frac 12}}+\frac{N}{2s \g}-\frac{N}{2s})>-1.
$$
Hence $\g<\frac{N+2s}{N+s}$. Using the fact that
$\dfrac{1}{\theta'}+\dfrac{1}{m}+\dfrac{1}{ \g}=2$, it holds that $\dfrac{1}{\theta}>\dfrac{1}{m}-\dfrac{s}{N+2s}$ and hence $\theta<\dfrac{m (N+2s)}{(N+2s-sm)_+}$. Then
we conclude.
\end{proof}
\begin{Theorem}\label{regu-g}
Assume that the conditions of Theorem \ref{hardy0} hold. Let $u$ be the unique weak solution to problem \eqref{eq:def}, then $|\n u|\d^{1-s}\in L^p(\O_T)$ for all
$p\ge 1$ such that $\dfrac{1}{p}>\dfrac{1}{m}-\dfrac{2s-1}{N+2s}$. Moreover
\begin{equation}\label{gradeq1}
\big\| |\n u| \d^{1-s} \big\|_{L^p(\O_T)}\le C ||f||_{L^m(\O_T)},
\end{equation}
with
$$
\left\{
\begin{array}{rcll}
p &<& \infty  &\mbox{  if  }  m\ge \dfrac{N+2s}{2s-1},\\
p &< & \dfrac{m (N+2s)}{N+2s-m(2s-1)}  &\mbox{  if  }  m<\dfrac{N+2s}{2s-1}.
\end{array}
\right.
$$
\end{Theorem}
\begin{proof}
We follow the same technique as in the proof of Theorem \ref{hardy0}. By the representation formula and  by using \eqref{green2} and setting $\O_t=\O\times (0,t)$, we have
\begin{equation}\label{madrid20}
\begin{array}{rcll}
& & |\nabla u(x,t)|\leq C \dyle \iint_{\O_t} f(y,\sigma) |\nabla_x P_{\Omega} (x,y, t-\sigma)|\,dy\,d\sigma\\ \\
 &\leq &  C \dyle\iint_{\O_t} f(y,\sigma) \frac{|\nabla_x P_{\Omega} (x,y, t-\sigma)|}{P_{\Omega} (x,y, t-\sigma)}P_{\Omega} (x,y, t-\sigma)\,dy\,d\sigma\\ \\ &\leq & \dfrac{C}{\d(x)}\dyle \iint_{\{\O_t\cap \{\d(x)\le (t-\s)^{\frac{1}{2s}}\}\}}\,f(y,\s) P_{\Omega} (x,y, t-\sigma) \,dy\,d\sigma \\ \\
& + & C\dyle \iint_{\{\O_t\cap \{\d(x)\ge (t-\s)^{\frac{1}{2s}}\}\}}f(y,\s)
\dfrac{(t-\s)^{\frac{2s-1}{2s}}\,dy\,d\sigma }{((t-\s)^{\frac{1}{2s}}+|x-y|)^{N+2s}}\\ \\ &\le &   C\dfrac{u(x,t)}{\d(x)}+C
\dyle\iint_{\O_t}f(y,\s)
\dfrac{(t-\s)^{\frac{2s-1}{2s}}}{((t-\s)^{\frac{1}{2s}}+|x-y|)^{N+2s}}\,dy\,d\sigma.
\end{array}
\end{equation}
Hence
\begin{equation}\label{rrt}
\begin{array}{rcl}
|\nabla u(x,t)| \d^{1-s}(x)&\leq& C\dfrac{u(x,t)}{\d^s(x)}+C\dyle \iint_{\O_t}f(y,\s)
\dfrac{(t-\s)^{\frac{2s-1}{2s}}}{((t-\s)^{\frac{1}{2s}}+|x-y|)^{N+2s}}\,dy\,d\sigma\\ &:= & J_1(x,t)+J_2(x,t).
\end{array}
\end{equation}
From Theorem \ref{hardy0}, it holds that $J_1\in L^\theta(\O_T)$ for all $\theta<\dfrac{m (N+2s)}{(N+2s-ms)_+}$ and
\begin{equation}\label{h000}
||J_1||_{L^\theta(\O_T)}\le C||f||_{L^m(\O_T)}.
\end{equation}
We deal with $J_2$. As above, we will use a duality argument. Let $\phi\in \mathcal{C}^\infty_0(\O_T)$, and define
$$\bar{H}(x,t):=\dfrac{t^{\frac{2s-1}{2s}}}{(t^{\frac{1}{2s}}+|x|)^{N+2s}},$$
then
$$
\begin{array}{lll}
||J_2||_{L^p(\O_T)} &= & \dyle \sup_{\{||\phi||_{L^{p'}(\O_T)}\le 1\}} \iint_{\O_T}\phi(x,t)J_2(x,t) dxdt\\ &\le & \dyle  \sup_{\{||\phi||_{L^{p'}(\O_T)}\le 1\}}
\iint_{\O_T}|\phi(x,t)| \dint_{0}^{t}\dint_{\Omega}f(y,\s) \dfrac{(t-\s)^{\frac{2s-1}{2s}}}{((t-\s)^{\frac{1}{2s}}+|x-y|)^{N+2s}}\,dy\,d\sigma dxdt\\ &\le & \dyle
\sup_{\{||\phi||_{L^{p'}(\O_T)}\le 1\}} \int_0^T\int_0^t \io \io |\phi(x,t)| \bar{H}(x-y,t-\s)f(y,\s)dydx d\s dt.
\end{array}
$$
Hence using Young inequality, we obtain that
\begin{equation}\label{rrr-1}
||J_2||_{L^p(\O_T)}=\dyle \sup_{\{||\phi||_{L^{p'}(\O_T)}\le 1\}} \int_0^T ||\phi(.,t)||_{L^{p'}(\O)} \int_0^t ||\bar{H}(.,t-\s)||_{L^a(\O)} ||f(.,\s)||_{L^m(\O)}
d\s dt,
\end{equation}
with $\dfrac{1}{p'}+\dfrac{1}{m}+\dfrac{1}{a}=2$.
 By direct computations, we have
$$
||\bar{H}(.,t-\s)||_{L^a(\O)}\le C (t-\s)^{\frac{2s-1}{2s} + \frac{N}{2s a}-\frac{(N+2s)}{2s}}={ C (t-\s)^{\frac{-1}{2s} + \frac{N}{2s a}-\frac{N}{2s}}.}
$$
Going back to \eqref{rrr-1}, we conclude that
$$
\begin{array}{lll}
||J_2||_{L^p(\O_T)} &\le & C\dyle \sup_{\{||\phi||_{L^{p'}(\O_T)}\le 1\}} \int_0^T ||\phi(.,t)||_{L^{p'}(\O)} \int_0^t ||f(.,\s)||_{L^m(\O)} (t-\s)^{\frac{-1}{2s} +
\frac{N}{2sa}-\frac{N}{2s}} d\s dt\\
&\le & C\dyle \sup_{\{||\phi||_{L^{p'}(\O_T)}\le 1\}} \int_0^T \int_0^T||\phi(.,t)||_{L^{p'}(\O)} \int_0^t ||f(.,\s)||_{L^m(\O)} |t-\s|^{\frac{-1}{2s} +\frac{N}{2sa}-\frac{N}{2s}} d\s dt\\
\end{array}
$$
Thus, using again Young inequality, we get
$$
\begin{array}{lll}
||J_2||_{L^p(\O_T)} &\le & C ||f||_{L^m(\O_T)} \dyle \sup_{\{||\phi||_{L^{p'}(\O_T)}\le 1\}} ||\phi||_{L^{p'}(\O_T)} \bigg(\int_0^T
|t|^{\g(\frac{-1}{2s} +\frac{N}{2sa}-\frac{N}{2s})}
dt\bigg)^{\frac{1}{\g}},
\end{array}
$$
where $\dfrac{1}{p'}+\dfrac{1}{m}+\dfrac{1}{\g}=2$. Hence $\g=a$.
It is clear that the last integral is finite if and only if
$\g<\frac{N+2s}{N+1}$. Since
$\dfrac{1}{p'}+\dfrac{1}{m}+\dfrac{1}{\g}=2$, then $\dfrac{1}{p}>\dfrac{1}{m}-\dfrac{2s-1}{N+2s}$. Hence
\begin{equation}\label{k000-1}
||J_2||_{L^p(\O_T)}\le C||f||_{L^m(\O_T)}.
\end{equation}
If  $m\ge \dfrac{N+2s}{2s-1}$, then the above condition holds for all $p>1$.
Since $s\in(\dfrac 12, 1)$, $\dfrac{N+2s}{2s-1}>\dfrac{N+2s}{s}$, and then
combining \eqref{h000} and \eqref{k000-1},
we conclude that, for all $p<\infty$,
\begin{equation}\label{k0001}
|||\n u| \d^{1-s}||_{L^p(\O_T)}\le C||f||_{L^m(\O_T)}.
\end{equation}
If  $m<\dfrac{N+2s}{2s-1}$, then $p<\dfrac{m (N+2s)}{N+2s-m(2s-1)}$.
Since $\dfrac{m (N+2s)}{(N+2s-m(2s-1))}<\dfrac{m (N+2s)}{(N+2s-ms)_+}$, using
\eqref{h000} and \eqref{k000-1}, we obtain that
\begin{equation}\label{k000}
||\n u| \d^{1-s}||_{L^p(\O_T)}\le C||f||_{L^m(\O_T)}
\end{equation}
for all $p$ which  satisfies $\dfrac{1}{p}>\dfrac{1}{m}-\dfrac{2s-1}{N+2s}$.
\end{proof}

 {\begin{remark} Let $u(x,t)=t w$ where $w(x)=(1-|x|^2)_+^s$, solves the problem
$$
(-\D)^s w=1 \mbox{  in  }B_1(0) \mbox{  and  }w=0 \mbox{  in  }\ren\backslash B_1(0).
$$
Then
$$
u_t+(-\D)^s u=w+t:=f(x,t) \mbox{  in  }B_1(0)\times (0,T).
$$
Notice that $f\in \mathcal{C}(\overline{\O_T})$, however $|\n u(x)|=2st|x|(1-|x|^2)_+^{s-1}$ in $B_1(0)\times (0,T)$, then $|\n u|\d^\a\in L^\infty(B_1(0)\times (0,T))$ if and only if $\a\ge 1-s$ which show in some way the optimality of the regularity result obtained in Theorem \ref{regu-g}.
\end{remark}
}

\begin{Corollary}\label{cor11}

\

\begin{enumerate}

\item By the result of Theorems \ref{hardy0}
and \ref{regu-g} and since $|\nabla \d|=1$   a.e. in  $\Omega$, it holds that if $u$ is the unique weak solution to problem \eqref{eq:def}, then $(u \d^{1-s})\in L^p(0,T; W^{1,p}_0(\O))$ for all $p$ such that  $\dfrac{1}{p}>\dfrac{1}{m}-\dfrac{2s-1}{N+2s}$ and
$$
||u \d^{1-s} ||_{L^p(0,T; W^{1,p}_0(\O))}\le C ||f||_{L^m(\O)}.
$$
Moveover, if $m<\dfrac{N+2s}{2s-1}$, then the above estimate holds for all $p<\dfrac{m (N+2s)}{N+2s-m(2s-1)}$.
\item  Assume that $f\in L^m(\O_T)$ for some $m>1$ and let $u$ be the unique weak solution to problem \eqref{eq:def}, then
\begin{enumerate}
\item If $m\ge\dfrac{N+2s}{2s-1}$, then $\dyle \iint_{\O_T}|\n u|^a dx<\infty$ for all $a<\dfrac{1}{1-s}$.
\item If $\dfrac{1}{s}\le m<\dfrac{N+2s}{2s-1}$,
    then $\dyle \int_0^T\io |\n u|^a dx<\infty$,

    for all $a<\check{P}:=\dfrac{m(N+2s)}{(N+2s)(m(1-s)+1)-m(2s-1)}$.
\item If $1<m<\dfrac{1}{s}$, then $\dyle
    \iint_{\O_T}|\n u|^a dx<\infty$ for all $a<\dfrac{N+2s}{N+1}$.
\end{enumerate}
\end{enumerate}
\end{Corollary}
\begin{proof}
{ Let us begin with the first statement. Define $v(x,t)=u(x,t)\d^{1-s}(x)$, then
$$
\n v(x,t)=\d^{1-s}(x)\n u +(1-s)\frac{u(x,t)}{\d^s(x)}\n \d(x).
$$
Using the fact that $|\n \d(x)|=1$ a.e. in $\O$, it holds that
$$
|\n v(x,t)|\le \d^{1-s}(x)|\n u(x,t)| +(1-s)\frac{|u(x,t)|}{\d^s(x)}.
$$
Hence the desired estimate follows combining the two estimates obtained in Theorems \ref{hardy0} and
\ref{regu-g}. }

\

We prove the second point that provides a global regularity for the gradient term
without using any weight. It is clear that $u\in L^a(0,T;W^{1,a}_0(\O))$ for all
$a<\dfrac{N+2s}{N+1}$. Now, using Theorem \ref{regu-g}, we reach that
$|\n u|\d^{1-s}\in L^p(\O_T)$ with $p\ge 1$ which satisfies
$\dfrac{1}{p}>\dfrac{1}{m}-\dfrac{2s-1}{N+2s}$.
Hence using H\"older inequality, we get
$$
\begin{array}{lll}
\dyle \iint_{\O_T}|\n u|^a dx dt&= &\dyle \iint_{\O_T}\bigg(|\n u|\d^{1-s}\bigg)^a \d^{-a(1-s)}dx dt\\ &\le & \dyle \le \bigg(\iint_{\O_T}(|\n u|\d^{1-s})^p dx dt
\bigg)^{\frac{a}{p}}\bigg(\iint_{\O_T}\d^{-\frac{pa}{p-a}(1-s)}dx dt \bigg)^{\frac{p-a}{p}},
\end{array}
$$
where $p>a$ to be chosen later. The last integral is finite if and
only if $\dfrac{ap}{p-a}(1-s)<1$, that is, if  $a<\dfrac{p}{(1-s)p+1}$.
Notice that in particular, $a<\dfrac{1}{1-s}$.

\noindent If $m\ge \dfrac{N+2s}{2s-1}$, then by Theorem \ref{regu-g}, we know that
$|\n u|\d^{1-s}\in L^p(\O_T)$ for all $p<\infty$. Hence the condition
$\dfrac{ap}{p-a}(1-s)<1$ holds if $a(1-s)<1$ and then we conclude.

\noindent Assume that $m<\dfrac{N+2s}{2s-1}$, since $p<\dfrac{m (N+2s)}{N+2s-ms}$,
we get $$a<\breve{P}:=\dfrac{m(N+2s)}{(N+2s)(m(1-s)+1)-m(2s-1)}.$$
It is clear that $\breve{P}\ge \dfrac{N+2s}{N+1}$ if $m\ge \dfrac{1}{s}$.
Thus we conclude.
\end{proof}

{In the next result, under suitable hypotheses on $s,m$ and $p$, we improve the integrability of the gradient of the solution without the degenerate weight or with a less degenerate weight at the boundary. }

{
\begin{Theorem}\label{regu-glast}
Assume that the conditions of Theorem \ref{hardy0} hold. Let $u$ be the unique weak solution to problem \eqref{eq:def}, then:
\begin{enumerate}
\item Let $m_1=\min\{\dfrac{s}{1-s}, m\}$. Then $u\in L^p(0,T;W^{1,p}_0(\O))$ for $p<\dfrac{m_1(N+2s)}{N+s+(1-s)m_1}$, and
$$
||\n u||_{L^p(\O_T)}\le C(\O_T,s,N,p)||f||_{L^m(\O_T)}
$$
\item If  $\dfrac{m(N+2s)}{N+s+(1-s)m}\le p <\min\{\dfrac{m(N+2s)}{N+s}, \dfrac{m(N+2s)}{N+2s-m(2s-1)}\}$, then for all $\a\in (0,1)$ such that
$$
\a>\frac{1}{1-s}\bigg((\frac{1}{m}-\frac{1}{p})(N+2s)+(1-s)-\frac{s}{m}\bigg),
$$
we have $|\n u|\d^{\a(1-s)}\in L^p(\O_T)$ and
$$
|| |\n u|\, \d^{\a(1-s)}||_{L^p(\O_T)}\le C(\O_T,s,N,p)||f||_{L^m(\O_T)}.
$$
\end{enumerate}
\end{Theorem}
}
\begin{proof}  As in the proof of Theorem \ref{regu-g}, from estimate \eqref{madrid20}, we have
\begin{equation}\label{madrid21}
\begin{array}{rcll}
|\nabla u(x,t)|&\leq &
\dyle \dfrac{C}{\d(x)}\dyle \iint\limits_{\{\O_t\cap \{\d(x)\le (t-\s)^{\frac{1}{2s}}\}\}}\,f(y,\s) P_{\Omega} (x,y, t-\sigma) \,dy\,d\sigma \\ \\ &+& C\dyle \iint\limits_{\O_t}f(y,\s) \dfrac{(t-\s)^{\frac{2s-1}{2s}}\,dy\,d\sigma }{((t-\s)^{\frac{1}{2s}}+|x-y|)^{N+2s}}\\ \\
& :=&\hat{J}_1(x,t)+ \hat{J}_2(x,t).
\end{array}
\end{equation}
The estimate of the term $\hat{J}_2$ is similar to the estimate of the term
 $J_2$ in the proof of Theorem \ref{regu-g}. Precisely,
it holds that $\hat{J}_2\in L^p(\O_T)$ for $p\ge 1$ satisfying
$\dfrac{1}{p}>\dfrac{1}{m}-\dfrac{2s-1}{N+2s}$ and
\begin{equation}\label{k000-1last}
||\hat{J}_2||_{L^p(\O_T)}\le C||f||_{L^m(\O_T)}.
\end{equation}
We deal with $\hat{J}_1$, in this case we  follow the argument as in the proof of Theorem \ref{gradiente}.

Since $f\in L^m(\O_T)$ with $m>1$, then using H\"older inequality we deduce that
\begin{equation}\label{last0}
\begin{array}{lll}
\hat{J}_1(x,t)& \le & \dfrac{C}{\d(x)}\dyle\bigg(\iint\limits_{\{\O_t\cap \{\d(x)\le (t-\s)^{\frac{1}{2s}}\}\}}\,f^m(y,\sigma){P_{\Omega} (x,y, t-\s)}\,dy\,d\s\bigg)^{\frac{1}{m}}\\
   & \times & \dyle \bigg(\iint\limits_{\{\O_t\cap \{\d(x)\le (t-\s)^{\frac{1}{2s}}\}\}}\,{P_{\Omega} (x,y, t-\s)}\,dy\,d\s\bigg)^{\frac{1}{m'}}\\
& \le & \dfrac{C}{(\d(x))^{1-\frac{s}{m'}}}\dyle\bigg(\iint\limits_{\{\O_t\cap \{\d(x)\le (t-\s)^{\frac{1}{2s}}\}\}}\,f^m(y,\sigma){P_{\Omega} (x,y, t-\s)}\,dy\,d\s\bigg)^{\frac{1}{m}},
\end{array}
\end{equation}
where we have used the fact that
$$
\iint\limits_{\{\O_t\cap \{\d(x)\le (t-\s)^{\frac{1}{2s}}\}\}}\, {P_{\Omega} (x,y, t-\s)}\,dy\,d\s\le C(\O,T)\d^s(x) \mbox{  for  }(x,t)\in \O_T.
$$
Fix $p>m$ to be chosen later, then

\begin{equation}\label{last1}
\begin{array}{lll}
\hat{J}^p_1(x,t)& \le & \dfrac{C}{(\d(x))^{p(1-\frac{s}{m'})}}\dyle\bigg(\iint\limits_{\{\O_t\cap \{\d(x)\le (t-\s)^{\frac{1}{2s}}\}\}}\,f^m(y,\sigma){P_{\Omega} (x,y, t-\s)}\,dy\,d\s\bigg)^{\frac{p}{m}}\\  \\
& \le & \dfrac{C}{(\d(x))^{p(1-\frac{s}{m'})}}\dyle\bigg(\iint\limits_{\{\O_t\cap \{\d(x)\le (t-\s)^{\frac{1}{2s}}\}\}}\,f^m(y,\sigma){P^{\frac{p}{m}}_{\Omega} (x,y, t-\s)}\,dy\,d\s\bigg)
\bigg(\iint\limits_{\O_t}\,f^m(y,\sigma)\,dy\,d\s\bigg)^{\frac{p-m}{m}}\\  \\
& \le & \dfrac{C ||f||^{p-m}_{L^m(\O_T)}}{(\d(x))^{p(1-\frac{s}{m'})}}\dyle \iint\limits_{\{\O_t\cap \{\d(x)\le (t-\s)^{\frac{1}{2s}}\}\}}\,f^m(y,\sigma){P^{\frac{p}{m}}_{\Omega} (x,y, t-\s)}\,dy\,d\s.
\end{array}
\end{equation}
Recall that $t^{\frac{N}{2s}}P_{\Omega} (x,y, t)\le C:=C(\O,s,N,T)$, thus $\bigg(t^{\frac{N}{2s}}P_{\Omega} (x,y, t)\bigg)^{\frac{p}{m}-1}\le C$. Hence, if $\d(x)\le (t-\s)^{\frac{1}{2s}}$, it follows that
$$
P^{\frac{p}{m}}_{\Omega} (x,y, t-\s)\le C(t-\s)^{-(\frac{p}{m}-1)\frac{N}{2s}}P_{\Omega} (x,y, t-\s)\le \frac{C}{(\d(x))^{(\frac{p}{m}-1)N}}P_{\Omega} (x,y, t-\s).
$$
Therefore we obtain that
\begin{equation}\label{last12}
\begin{array}{lll}
\hat{J}^p_1(x,t)& \le & \dfrac{C ||f||^{p-m}_{L^m(\O_T)}}{(\d(x))^{p(1-\frac{s}{m'})+N(\frac{p}{m}-1)}}\dyle\iint_{\O_t}f^m(y,\sigma){P}_{\Omega} (x,y, t-\s)\,dy\,d\s.
\end{array}
\end{equation}
Let $\beta=p(1-\dfrac{s}{m'})+N(\dfrac{p}{m}-1)$, then according to the value of $\beta$, we will consider two cases:

\

{\bf The first case $\beta<2s$:} Notice that in this case if $m_1<\dfrac{s}{1-s}$ then  $m_1<p<\dfrac{m_1(N+2s)}{N+s+(1-s)m_1}$. Moreover,  $\dfrac{m_1(N+2s)}{N+s+(1-s)m_1}<\dfrac{m_1(N+2s)}{N+2s-m_1(2s-1)}$ defined in Theorem \ref{regu-g}.

Integrating in $\O_T$, there results that
$$
\iint_{\O_T}\hat{J}^p_{1}(x,t)dxdt\le C||f||^{p-m_1}_{L^{m_1}(\O_T)}\iint_{\O_T}f^{m_1}(y,\s)\Big(\io \frac{1}{\d^\beta(x)}\int_{\s}^T P_{\Omega} (x,y,t-\s)dt\,dx\Big)dy\,d\sigma.
$$
Recall that
$$
\dint_{\s}^{T} P_{\Omega} (x,y,
t-\s)\,dt\le \dint_{0}^{T-\s} P_{\Omega} (x,y,
\eta)\,d\eta\le \mathcal{G}_s(x,y).
$$
Hence
$$
\iint_{\O_T}\hat{J}^p_{1}(x,t)dxdt\le C||f||^{p-m_1}_{L^{m_1}(\O_T)}
\iint_{\O_T}f^{m_1}(y,\s)\varphi(y) dy\,d\sigma.
$$
where $\varphi(y)=\dyle\io \frac{\mathcal{G}_s(x,y)}{(\d(x))^{\beta_0}}dx$.
Since $\beta<2s$, then $\varphi\in L^{\infty}(\O)$ and then
$$
\iint_{\O_T}\hat{J}^p_{1}(x,t)dxdt\le C||f||^{p}_{L^{m_1}(\O_T)}.
$$
If $m>\dfrac{s}{1-s}$ the final estimate follows by the H\"{o}lder inequality.
Notice that in this case, $p<\dfrac{s}{1-s}$.

 {\bf The second case $\beta\ge 2s$}: Consider
 $\dfrac{m(N+2s)}{N+s+(1-s)m}\le p <\min\{\dfrac{m(N+2s)}{N+s}, \dfrac{m(N+2s)}{N+2s-m(2s-1)}\}$.
Since $\beta\ge 2s$,
 $$
 (\dfrac{1}{m}-\dfrac{1}{p})(N+2s)<\dfrac{s}{m}.
 $$
Let
$$
{\Upsilon}:= (\frac{1}{m}-\frac{1}{p})(N+2s)+(1-s)-\frac{s}{m},
$$
then $0<\Upsilon<1-s$. Fix $0<\a<1$ such that $\frac{\Upsilon}{1-s}<\a<1$, thus we reach that $\beta-p\a(1-s)<2s$.

Going back to \eqref{last12}, it holds that

\begin{equation}\label{last1211}
\begin{array}{lll}
\hat{J}^p_1(x,t) (\d(x))^{p\a(1-s)}& \le & \dfrac{C ||f||^{p-m}_{L^m(\O_T)}}{(\d(x))^{p(1-\frac{s}{m'})+N(\frac{p}{m}-1)-p\a(1-s)}}\dyle\iint_{\O_t}f^m(y,\sigma){P}_{\Omega} (x,y, t-\s)\,dy\,d\s.
\end{array}
\end{equation}
Setting $\hat{\beta}=p(1-\dfrac{s}{m'})+N(\dfrac{p}{m}-1)-p\a(1-s)$, then
$\hat{\beta}=\dfrac{p}{m}(s+(1-s)m+N-\a m(1-s))-N$.
By the above condition on $p$ and $\a$, we deduce that $\hat{\beta}<2s$. Repeating the argument used in the first case, it holds that
$$
\iint_{\O_T}\hat{J}^p_{1}(x,t)(\d(x))^{p\a(1-s)}dxdt\le C||f||^{p}_{L^m(\O_T)},
$$
and then we conclude.
\end{proof}
\begin{remark}
Notice that $u\in L^{2s}(0,T;W^{1,2s}_0(\O))$ if $m>\dfrac{2s(N+2s)}{N+2s^2}$.
\end{remark}
Now, if $f\in L^1(\O_T)\cap L^m(K\times (0,T))$, where $m>1$ and $K\subset\subset \O$ is any compact set of $\O$, then the regularity result of Theorem \ref{regu-g} holds locally in $\O\times (0,T)$. More precisely we have
\begin{Proposition}\label{key2-locc}
{ Suppose that $m>1$ and assume that $f\in L^1(\O_T)\cap L^m(K\times (0,T))$ for every compact set $K\subset\subset \O$.} Let $u$ be the unique weak solution to
problem \eqref{eq:def} and consider $\O_1\subset\subset \O$ with $\text{dist}(\O_1,\p\O)>0$. Let $K_1\subset\subset \O$ be a compact set of $\O$ such that
$\O_1\subset\subset K$. Then $u\in L^\theta(\O_1\times (0,T))$ for all $\theta<\dfrac{m (N+2s)}{(N+2s-ms)_+}$ and $|\n u|\in L^p(\O_1\times (0,T))$ for all
$p<\dfrac{m (N+2s)}{(N+2s-m(2s-1))_+}$. Moveover,
\begin{equation}\label{eqq1}
||u||_{L^\theta(\O_1\times (0,T))}\le C (||f||_{L^m(K_1\times (0,T))}+||f||_{L^1(\O_T)}),
\end{equation}
and
\begin{equation}\label{eqq2}
\|\n u\|_{L^p(\O_1\times (0,T))}\le C (||f||_{L^m(K_1\times (0,T))}+||f||_{L^1(\O_T)}),
\end{equation}
where $C:=C(K_1,\O_1, \O, T,N,m)$.
\end{Proposition}
\begin{proof}
Since $f\in L^1(\O_T)$, then $|\n u|\in L^{q}(\O_T)$ for all $q<\dfrac{N+2s}{N+1}$. Now we closely follow the proofs of Theorem \ref{hardy0} and Theorem
\ref{regu-g}. We have
$$
\begin{array}{rcl}
\dfrac{u(x,t)}{\d^s(x)} &\leq& C \dint_{0}^{t}\dint_{\Omega}f(y,\s) \dfrac{(t-\s)}{((t-\s)^{\frac{1}{2s}}+|x-y|)^{N+2s}}\,dy\,d\sigma.\\
\end{array}
$$
Fix $\O_1\subset\subset \O$ with $\text{dist}(\O_1,\p\O)=c_0>0$ and let $K_1$ be a compact set of $\O$ such that $\O_1\subset\subset K_1\subset\subset \O$. Let
$x\in \O_1$, then
$$
\begin{array}{rcl}
u(x,t) &\le & C(\O_1, c_0, C)\dint_{0}^{t}\dint_{\Omega}f(y,\s) \dfrac{(t-\s)}{((t-\s)^{\frac{1}{2s}}+|x-y|)^{N+2s}}\,dy\,d\sigma\\ &\le & C(\O_1, c_0,
C)\bigg\{\dint_{0}^{t}\dint_{K_1}f(y,\s) \dfrac{(t-\s)}{((t-\s)^{\frac{1}{2s}}+|x-y|)^{N+2s}}\,dy\,d\sigma\\ & + &
\dint_{0}^{t}\dint_{(\O\backslash K_1)}f(y,\s) \dfrac{(t-\s)}{((t-\s)^{\frac{1}{2s}}+|x-y|)^{N+2s}}\,dy\,d\sigma\bigg\}.
\end{array}
$$
Since $x\in \O_1\subset\subset K_1$, then for all $y\in \O\backslash K_1, |x-y|>\hat{c}>0$. Thus
$$
\begin{array}{rcl}
& &\dint_{0}^{t}\dint_{(\O\backslash K_1)}f(y,\s) \dfrac{(t-\s)}{((t-\s)^{\frac{1}{2s}}+|x-y|)^{N+2s}}\,dy\,d\sigma\\ & \le &
\dint_{0}^{t}\dint_{(\O\backslash K_1)}f(y,\s) \dfrac{(t-\s)}{((t-\s)^{\frac{1}{2s}}+\hat{c})^{N+2s}}\,dy\,d\sigma\\ & \le & C
||f||_{L^1(\O_T)}.
\end{array}
$$
Therefore we conclude that
$$
\begin{array}{rcl}
u(x,t) &\le & C(\O_1, c_0, C)\bigg(\dint_{0}^{t}\dint_{K_1}f(y,\s) \dfrac{(t-\s)}{((t-\s)^{\frac{1}{2s}}+|x-y|)^{N+2s}}\,dy\,d\sigma
+||f||_{L^1(\O_T)}\bigg).
\end{array}
$$
Since $f\in L^m(K_1\times (0,T))$, then the rest of the proof follows exactly from the same duality argument as in the proof of Theorem \ref{hardy0}. Hence, estimate
\eqref{eqq1} holds. In a similar way, we prove estimate \eqref{eqq2}.
\end{proof}

\begin{remarks}\label{rm001}

\

\begin{enumerate}

\item For $a\ge 1$, we define the space $L^{a}_{loc}(\Omega_T)$ as the set of measurable functions $u$ such that $u\eta\in L^{a}(\Omega_T)$, for any $\eta\in \mathcal{C}^\infty_0(\O)$. Then the result of Proposition \ref{key2-locc} affirms  that if $u$ is the unique solution to problem \eqref{eq:def}, then
    $u\in L^\theta_{loc}(\O_T)$ for all $\theta<\dfrac{m (N+2s)}{(N+2s-ms)_+}$ and $|\n u|\in L_{loc}^p(\O_T)$ for all $p<\dfrac{m (N+2s)}{(N+2s-m(2s-1))_+}$.

\item The result of Proposition \ref{key2-locc} will be useful in order to get $\mathcal{C}^1$ regularity using a bootstrap argument if, in addition, we have global
    bounds in $L^1$ and a local family of bounds in a suitable $L^m_{loc}$ space.
\end{enumerate}
\end{remarks}
\noindent We deal now with the case $f\equiv 0$ and $u_0\in L^\theta(\O)$  with $\theta\ge 1$. Following the same computations as above, we get the next results.
\begin{Theorem}\label{u0}
Suppose that $f\equiv 0$ and $u_0\in L^\rho(\O)$  with $\rho\ge 1$. If $u$ is the unique weak solution to problem \eqref{eq:def}, then $\dfrac{u}{\d^s}\in
L^\theta(\O_T)$ for all ${\theta<\dfrac{\rho(N+2s)}{N+s\rho}}$ and $|\n u|\d^{1-s}\in L^{p}(\O)$ for all $p<\dfrac{\rho(N+2s)}{N+\rho}$. Moveover,
$$
\Big|\Big|\dfrac{u}{\d^s}\Big|\Big|_{L^\theta(\O_T)}+|||\n u|\d^{1-s}||_{L^p(\O_T)}\le C(\O_T, p,\theta) ||u_0||_{L^\rho(\O)}.
$$
\end{Theorem}
\begin{proof}
For the reader convenience, we include here some details for the estimate of the term $\dfrac{u}{\d^s}$.
$$
\begin{array}{rcl}
\Big|\Big|\dfrac{u}{\d^s}\Big|\Big|_{L^\theta(\O_T)} &=& \dyle \sup_{\{||\phi||_{L^{\theta'}}(\O_T)\le 1\}} \iint_{\O_T}\phi(x,t) \frac{u(x,t)}{\d^s(x)}dxdt\\ &\le & \dyle
\sup_{\{||\phi||_{L^{\theta'}}(\O_T)\le 1\}} \iint_{\O_T}|\phi(x,t)|\dint_{\Omega}u_0(y) \dfrac{t^{{\frac 12}}}{(t^{\frac{1}{2s}}+|x -y|)^{N+2s}}\,dy\,dxdt\\
& \le &  \dyle \sup_{\{||\phi||_{L^{\theta'}(\O_T)}\le 1\}} \int_0^T\io \io |\phi(x,t)|
H(x-y,t)u_0(y)dydx dt
\end{array}
$$
with $\theta'=\dfrac{\theta}{\theta-1}$. Similarly to the proofs above, by Young's inequality, we have that
\begin{equation*}
\Big|\Big|\dfrac{u}{\d^s}\Big|\Big|_{L^\theta(\O_T)}\le C\dyle \sup_{\{||\phi||_{L^{\theta'}(\O_T)}\le 1\}} \int_0^T ||\phi(.,t)||_{L^{\theta'}(\O)} ||u_0||_{L^\rho(\O)}||H(., t)||_{L^a(\O)}d\s dt,
\end{equation*}
where
$$
H(|x-y|,t)= \dfrac{t^{{\frac 12}}}{(t^{\frac{1}{2s}}+|x-y|)^{N+2s}}
$$
and $\dfrac{1}{\theta'}+\dfrac{1}{\rho}+\dfrac{1}{a}=2$.
Notice that
$$
||H(., t)||_{L^a(\O)}\le C t^{-\frac{1}{2}+\frac{N}{2s a}-\frac{N}{2s}}.
$$
Therefore,
$$
\begin{array}{rcl}
\Big|\Big|\dfrac{u}{\d^s}\Big|\Big|_{L^\theta(\O_T)} &\le & C ||u_0||_{L^\rho(\O)} \dyle \sup_{\{||\phi||_{L^{\theta'}(\O_T)}\le 1\}} \int_0^T ||\phi(.,t)||_{L^{\theta'}(\O)}
t^{-\frac 12+\frac{N}{2s a}-\frac{N}{2s}}dt
\end{array}
$$
and by using the H\"older inequality, we get
$$
\begin{array}{lll}
\Big|\Big|\dfrac{u}{\d^s}\Big|\Big|_{L^\theta(\O_T)} &\le & C ||u_0||_{L^\rho(\O)} \dyle \sup_{\{||\phi||_{L^{\theta'}(\O_T)}\le 1\}} ||\phi||_{L^{\theta'}(\O_T)}
\dyle\left(\int_0^T t^{\theta(-\frac 12+\frac{N}{2s a}-\frac{N}{2s})}dt\dyle\right)^{\frac{1}{\theta}}.
\end{array}
$$
The last integral is finite if and only if $\theta(-\dfrac 12+\dfrac{N}{2s a}-\dfrac{N}{2s})>-1$, thus
$$ \dfrac{N}{2s}(1-\frac{1}{a})<\frac{1}{\theta}-\frac 12.
$$
Since  $\dfrac{1}{\rho}+\dfrac{1}{a}=1+\dfrac{1}{\theta}$
then $1-\dfrac{1}{a}=\dfrac{1}{\rho}-\dfrac{1}{\theta}$.
Substituting in the previous inequality, we conclude that
${{\theta<\dfrac{\rho (N+2s)}{N+\rho s}}}$.

To estimate the gradient term we consider that, by the representation formula, we have
$$
\begin{array}{rcl}
|\nabla u(x,t)|&\leq&
 C(\O_T) \dint_{\Omega} u_0(y)|\nabla_x P_{\Omega} (x,y, t)|\,dy\leq C(\O_T)\dint_{\Omega} u_0(y)\frac{|\nabla_x P_{\Omega} (x,y, t)|}{P_{\Omega} (x,y, t)}\,P_{\Omega} (x,y, t) dy\,\\
&\leq& C(\O_T)\dint_{\Omega} u_0(y) h(x,y,t){P_{\Omega} (x,y, t)}\,dy
\end{array}
$$
with $h(x,y,t)=\dfrac{|\nabla_x P_{\Omega} (x,y, t)|}{P_{\Omega} (x,y, t)}\leq C \Big( \dfrac{1}{\delta(x) \wedge t^{\frac{1}{2s}}}\Big)$.

Hence
\begin{eqnarray*}
|\n u(x,t)| & \le & \frac{C}{t^{\frac{1}{2s}}}}{ \chi_{\{\d(x)>t^{\frac{1}{2s}}\}}\io u_0(y){P_{\Omega} (x,y, t)}\,dy  +
\frac{C}{\d(x)}\chi_{\{\d(x)\le t^{\frac{1}{2s}}\}}\io u_0(y){P_{\Omega} (x,y, t)}\,dy\\
& = & J_{11}(x,t)+ J_{12}(x,t).
\end{eqnarray*}
We again use the duality argument, starting by the estimation of the term $J_{11}$.
By estimate \eqref{green1} and H\"older inequality, we obtain that
$$
\begin{array}{rcl}
||J_{11}||_{L^p(\O_T)} &=& \dyle \sup_{\{||\phi||_{L^{p'}}(\O_T)\le 1\}} \iint_{\O_T}\phi(x,t) J_{11}(x,t)dxdt\\ &\le & \dyle
\sup_{\{||\phi||_{L^{p'}}(\O_T)\le 1\}} \iint_{\O_T}|\phi(x,t)|\dint_{\Omega}u_0(y) \dfrac{t^{1-\frac 1{2s}}}{(t^{\frac{1}{2s}}+|x-y|)^{N+2s}}\,dy\,dxdt\\
& \le &  \dyle \sup_{\{||\phi||_{L^{p'}(\O_T)}\le 1\}} \int_0^T\io \io |\phi(x,t)|
H(x-y,t)u_0(y)dydx dt
\end{array}
$$
with $H(x,t)=\dfrac{t^{1-\frac 1{2s}}}{(t^{\frac{1}{2s}}+|x|)^{N+2s}}$.
Using Young inequality, it holds that
\begin{equation*}
||J_{11}||_{L^p(\O_T)}\le C\dyle \sup_{\{||\phi||_{L^{p'}(\O_T)}\le 1\}} \int_0^T ||\phi(.,t)||_{L^{p'}(\O)} ||u_0||_{L^\rho(\O)}||H(., t)||_{L^a(\O)}d\s dt,
\end{equation*}
with $\dfrac{1}{\theta'}+\dfrac{1}{\rho}+\dfrac{1}{a}=2$.
Notice that by a direct calculation
$$
||H(., t)||_{L^a(\O)}\le C t^{(1-\frac 1{2s})+\frac{N}{2s a}-\frac{(N+2s)}{2s}}.
$$
Therefore we reach that
$$
\begin{array}{rcl}
||J_{11}||_{L^\theta(\O_T)} &\le & C ||u_0||_{L^\rho(\O)} \dyle \sup_{\{||\phi||_{L^{p'}(\O_T)}\le 1\}} \int_0^T ||\phi(.,t)||_{L^{p'}(\O)}
t^{(1-\frac 1{2s})+\frac{N}{2s a}-\frac{(N+2s)}{2s}}dt.
\end{array}
$$
Using H\"older inequality, we get
$$
\begin{array}{lll}
||J_{11}||_{L^p(\O_T)} &\le & C ||u_0||_{L^\rho(\O)} \dyle \sup_{\{||\phi||_{L^{p'}(\O_T)}\le 1\}} ||\phi||_{L^{p'}(\O_T)}
\dyle\left(\int_0^T t^{p((1-\frac 1{2s})+\frac{N}{2s a}-\frac{(N+2s)}{2s})}dt\dyle\right)^{\frac{1}{p}}.
\end{array}
$$
The last integral is finite if and only if $p((1-\dfrac 1{2s})+\dfrac{N}{2s a}-\dfrac{(N+2s)}{2s})>-1$.
Thus
$$\dfrac 1{2s}+\dfrac{N}{2s}(1-\frac{1}{a})<\frac{1}{p}.
$$
Since $\dfrac{1}{p'}+\dfrac{1}{\rho}+\dfrac{1}{a}=2$  we have
 $\dfrac{1}{\rho}+\dfrac{1}{a}=1+\dfrac{1}{p}$.
Then $1-\dfrac{1}{a}=\dfrac{1}{\rho}-\dfrac{1}{p}$.
Going back to the previous inequality, we conclude that
$$
\dfrac{1}{2s}+\dfrac{1}{\rho}  \dfrac{N}{2s}<\dfrac{1}{p}(1+\dfrac{N}{2s}).
$$
Hence $J_{11}\in L^p(\O_T)$ for all $p<\dfrac{\rho (N+2s)}{N+\rho}$.

We deal now with $J_{12}$.
We have
$$
\begin{array}{lll}
J_{12}(x,t) &= & \dfrac{C}{\d(x)}\dint\limits_{\{\Omega\cap \{\d(x)\le t^{\frac{1}{2s}}\}\}} u_0(y){P_{\Omega} (x,y, t)}\,dy\\
& \le & C \frac{C}{\d^{1-s}(x)}\chi_{\{\d(x)\le t^{\frac{1}{2s}}\}}\io u_0(y)\dfrac{t^{1-{ \frac 1{2}}}}{(t^{\frac{1}{2s}}+|x-y|)^{N+2s}}\,dy.\\
\end{array}
$$
Hence
$$
J_{12}(x,t)\d^{1-s}(x)\le C\io u_0(y)\dfrac{t^{{{\frac 12}}}}{(t^{\frac{1}{2s}}+|x-y|)^{N+2s}}\,dy,
$$
that is the same term estimated above. Hence we conclude that $J_{12}\d^{1-s}\in L^p(\O_T)$ for all $p<\dfrac{\rho (N+2s)}{N+s\rho}$.

Since $\dfrac{\rho (N+2s)}{N+s\rho}>\dfrac{\rho (N+2s)}{N+\rho}$, then using the fact that $\O_T$ is bounded, it follows that $|\n u|\d^{1-s}(x)
\in L^p(\O_T)$ for all $p<\dfrac{\rho (N+2s)}{N+\rho}$.
\end{proof}

 As in Corollary \ref{cor11}, we have the next regularity for the gradient.
\begin{Corollary}\label{cor1100}
Assume that $f\equiv 0$ and $u_0\in L^\rho(\O)$  with $\rho\ge 1$. If $u$ is the unique weak solution to problem \eqref{eq:def}, then $|\n u|\in L^{a}(\O)$ for all $a<\check{U}:=\dfrac{\rho(N+2s)}{(1-s)\rho(N+2s)+N+\rho}$. Moveover $\check{U}>\dfrac{N+2s}{N+1}$ if $\dfrac{N+2s}{N}<\frac{1}{1-s}$ and  $\rho>\dfrac{N}{sN-2s(1-s)}$.
\end{Corollary}

{ As a conclusion, from Proposition \ref{first11} and using the same approach as in the proof of Theorem \ref{u0}, keeping  the power in the time variable, we get the next results. }

\begin{Proposition}\label{pro:lp2}
Suppose that $f\equiv 0$ and $u_0\in L^\rho(\O)$. If $u$ is the unique weak solution to problem \eqref{eq:def}, then for all $r\ge 1$ and for all $t>0$, we
have
\begin{equation}\label{sem1}
||u(\cdot,t)||_{L^r(\O)}\le Ct^{-\frac{N}{2s}(\frac{1}{\rho}-\frac{1}{r})}||u_0||_{L^\rho(\O)},
\end{equation}
\begin{equation}\label{sem001}
\Big\|\frac{u(\cdot,t)}{\d^s}\Big\|_{L^r(\O)}\le Ct^{-\frac{N}{2s}(\frac{1}{\rho}-\frac{1}{r})-\frac{1}{2s}}||u_0||_{L^\rho(\O)},
\end{equation}
and
\begin{equation}\label{sem111}
||\n u(\cdot,t)\d^{1-s}||_{L^r(\O)}\le Ct^{-\frac{N}{2s}(\frac{1}{\rho}-\frac{1}{r})-\frac{1}{2s}}||u_0||_{L^\rho(\O)}.
\end{equation}
Moreover, $u\in L^{\s}(\O_T)$ for all $\s<\frac{\rho(N+2s)}{N}$ and $\dfrac{u}{\d^s}, |\n u|\d^{1-s}\in L^{\gamma}(\O_T)$ for all $\gamma<\frac{\rho(N+2s)}{N+\rho}$.
\end{Proposition}

{ Within the same framework of the Theorem \ref{regu-glast}, in order to get regularity for the gradient term without any degenerate weight we have the next result.

 \begin{Theorem}\label{lastu0}
Suppose that $f\equiv 0$ and $u_0\in L^\rho(\O)$ and let $\rho_1$ be defined by $\rho_1=\min\{\rho, 2s\}$.  If $u$ is the unique weak solution to problem \eqref{eq:def}, then $|\n u|\in L^{q}(\O)$ for all $q<\dfrac{\rho_1(N+2s)}{N+\rho_1}$. Moveover,
$$
|||\n u|||_{L^q(\O_T)}\le C(\O_T, p,\theta) ||u_0||_{L^\rho(\O)}.
$$
\end{Theorem}
}
\begin{proof}
To estimate the gradient term we consider that, by the representation formula, we have
\begin{eqnarray*}
& & |\nabla u(x,t)|\leq
 C(\O_T) \dint_{\Omega} u_0(y)|\nabla_x P_{\Omega} (x,y, t)|\,dy\leq C(\O_T)\dint_{\Omega} u_0(y)\frac{|\nabla_x P_{\Omega} (x,y, t)|}{P_{\Omega} (x,y, t)}\,P_{\Omega} (x,y, t) dy\,\\
& & \le \frac{C}{t^{\frac{1}{2s}}}\chi_{\{\d(x)>t^{\frac{1}{2s}}\}}\io  u_0(y){P_{\Omega} (x,y, t)}\,dy  +
\frac{C}{\d(x)}\chi_{\{\d(x)\le t^{\frac{1}{2s}}\}}\io u_0(y){P_{\Omega} (x,y, t)}\,dy\\
& &= \breve{J}_{1}(x,t)+ \breve{J}_{2}(x,t).
\end{eqnarray*}
As in the proof of Theorem \ref{u0}, using the duality argument we obtain that, for all $q<\dfrac{\rho(N+2s)}{N+\rho}$,
$$
||\breve{J}_{1}||_{L^q(\O_T)}\le C ||u_0||_{L^\rho(\O)}.
$$
We treat now $\breve{J}_{2}$. We have
{
$$
\begin{array}{lll}
\breve{J}_{2}(x,t) &= & \dyle \dfrac{C}{\d(x)}\chi_{\{\d(x)\le t^{\frac{1}{2s}}\}} \io u_0(y){P_{\Omega} (x,y, t)}\,dy\\
 &\le & \dyle \dfrac{C}{\d(x)}\chi_{\{\d(x)\le t^{\frac{1}{2s}}\}} \bigg(\io u^\rho_0(y){P_{\Omega} (x,y, t)}\,dy\bigg)^{\frac{1}{\rho}}
\bigg(\io{P_{\Omega} (x,y, t)}\,dy\bigg)^{\frac{1}{\rho'}}\\
& \le & \dyle \dfrac{C}{\d(x)} \chi_{\{\d(x)\le t^{\frac{1}{2s}}\}} \bigg(\io u^\rho_0(y){P_{\Omega} (x,y, t)}\,dy\bigg)^{\frac{1}{\rho}},
\end{array}
$$
where we have used the fact that
$$
\io{P_{\Omega} (x,y, t)}\,dy\le C\frac{\d^s(x)}{\sqrt{t}}\le C \mbox{  in the set  } \{\d(x)\le t^{\frac{1}{2s}}\}.
$$
}
Hence, as in the estimate of the term $\hat{J}_1$ in Theorem \ref{regu-glast}, we deduce that
$$
\begin{array}{lll}
\breve{J}^q_{2}(x,t) &\le & \dyle \dfrac{C}{\d^q(x)}\chi_{\{\d(x)\le t^{\frac{1}{2s}}\}} \bigg(\io u^\rho_0(y){P_{\Omega} (x,y, t)}\,dy\bigg)^{\frac{q}{\rho}}\\
&\le & \dyle\dfrac{C}{\d^q(x)} \chi_{\{\d(x)\le t^{\frac{1}{2s}}\}} \bigg(\io u^\rho_0(y){P^{\frac{q}{\rho}}_{\Omega} (x,y, t)}\,dy\bigg)\bigg(\io u^\rho_0(y) dy \bigg)^{\frac{q-\rho}{\rho}}\\
&\le & \dyle \dfrac{C||u_0||^{\frac{q-\rho}{\rho}}_{L^\rho(\O)}}{\d^q(x)} \chi_{\{\d(x)\le t^{\frac{1}{2s}}\}} \bigg(\io  u^\rho_0(y){P^{\frac{q}{\rho}}_{\Omega} (x,y, t)}\,dy\bigg)\\
&\le & \dyle \dfrac{C||u_0||^{\frac{q-\rho}{\rho}}_{L^\rho(\O)}}{(\d(x))^{q+(\frac{q}{\rho}-1)N}}\chi_{\{\d(x)\le t^{\frac{1}{2s}}\}} \io  u^\rho_0(y){P}_{\Omega} (x,y, t)\,dy.\\
\end{array}
$$
Setting $\beta_0=q+(\frac{q}{\rho}-1)N$, then $\beta<2s$ if and only if $q<\frac{\rho(N+2s)}{N+\rho}$. It is clear that $\frac{\rho(N+2s)}{N+\rho}>\rho$ if and only if $\rho<2s$. Thus assuming the above assumptions and integrating in $\O_T$, we obtain that
$$
\iint_{\O_T}\breve{J}^q_{2}(x,t) dxdt \le C||u_0||^{\frac{q-\rho}{\rho}}_{L^\rho(\O)} \io u^\rho_0(y)\phi (y) dy,
$$
where, as above,  $\phi(y)=\dyle\io \frac{\mathcal{G}_s(x,y)}{(\d(x))^{\beta_0}}dx\le C$.
Hence
$$
\iint_{\O_T}\breve{J}^q_{2}(x,t) dxdt \le C||u_0||^{\frac{q}{\rho}}_{L^\rho(\O)}.
$$
Thus $u\in L^q((0,T); W^{1,q}_0(\O))$ and
$$
|| |\n u|||_{L^q(\O_T)}\le C(\O_T,N,s,p) ||u_0||_{L^\rho(\O)}.
$$
\end{proof}

Notice that, from \cite{BKW}, see also \cite{W}, working in the whole space $\ren$, the above regularity result on the gradient holds  globally. However, when working in a bounded domain, the  term $\delta^{1-s}$  appears in a natural way describing  the gradient's behavior.

\

 Under the local summability condition $f\d^\beta\in L^1(\O_T)$ for some $\beta<2s-1$, it is possible to show the existence of weak solution with the same range of regularity. This will be the key in order to analyze problem \eqref{grad} for large value of $\a$.

More precisely, we have the following existence result.
\begin{Theorem} \label{key}
Assume that $f, u_0$ are measurable functions such that $f\d^\beta \in L^1(\O_T)$ for some $\beta<(2s-1)$ and $u_0\in L^1(\O)$.

Then problem \eqref{eq:def} has a
unique weak solution $u$ such that for all $q<\dfrac{N+2s}{N+\beta+1}$,
$$
||u||_{\mathcal{C}([0,T],L^1(\O, \d^\beta dx))}+ ||\nabla u||_{L^{q}(\Omega_T)}\le C(q,\beta,\Omega_T)\bigg(||f\d^\beta
||_{L^{1}(\Omega_T)}+||u_0||_{L^1(\O)}\bigg).
$$
Moreover, for $q<\dfrac{N+2s}{N+\beta+1}$ fixed, setting $\hat{K}: L^{1}(\Omega_T, \d^\beta(x)dxdt)\times L^1(\O)\to L^q(0,T; W_{0}^{1,q}(\Omega))$, $\hat{K}(f,
u_0)=u$, the unique solution to problem \eqref{eq:def}, then $\hat{K}$ is a compact operator.
\end{Theorem}
\begin{proof}
Without loss of generality, we can assume that $u_0=0$ and that $f\gneqq 0$ in $\O_T$. We begin by proving the existence part. Let $f_n=T_n(f)$ and define $u_n$
to be the unique energy solution to the approximating problem
\begin{equation}\label{apprd}
\left\{
\begin{array}{rcll}
u_{nt}+(-\D)^s u_n&=& \dyle f_n  & \text{ in } \O_{T}, \\ u_n&=&0 & \text{ in }(\ren\setminus\O) \times (0,T), \\ u(x,0)&=& 0& \mbox{
in  }\O.
\end{array}%
\right.
\end{equation}
It is clear that $\{u_n\}_n$ is an increasing sequence in $n$. Let $\psi$ be the solution to the problem
\begin{equation}\label{adim}
\left\{
\begin{array}{rcll}
(-\Delta)^s \psi&=& \dfrac{1}{\d^{2s-\beta}} & \text{ in } \O, \\ \psi &=&0 & \text{ in }\ren\setminus\O,
\end{array}%
\right.
\end{equation}
whose existence is a consequence of \cite{CV2}. By the results in \cite{abatan} and  \cite{Adm}  extending the results in  \cite{quim}, we find that $\psi\backsimeq C \d^\beta$.  Hence, by an approximation argument, we can use $\psi$ as a test function in \eqref{apprd} to obtain that
$$
\sup_{t\in [0,T]}\io u_n(x,t) \d^\beta(x) dx +\int_0^T\io \frac{u_n(x,t)}{\d^{2s-\beta}(x)}dxdt \le C(\O,\beta,s) \int_0^T\io f\d^\beta(x)dxdt.
$$
Hence there exists a measurable function $u\in L^\infty(0,T;L^1(\O, \d^\beta dx))\cap L^1(\O_T)$, $\dfrac{u}{\d^{2s-\beta}}\in L^1(\O_T)$,  such that
$\dfrac{u_n}{\d^{2s-\beta}}\to \dfrac{u}{\d^{2s-\beta}}$ strongly in $L^1(\O_T)$.

We claim that the sequence $\{u_n\}_n$ is bounded in $L^\theta(\O_T)$, for all $\theta<\frac{N+2s}{N+\beta}$.

To prove the claim consider the representation formula,
$$
u_n(x,t)=\dint_{0}^{t} \dint_{\Omega} f_n(y,\sigma) P_{\Omega} (x,y, t-\sigma)\,dy\,d\sigma,
$$
then  using the properties of $P_\O$, it holds that
$$
\begin{array}{rcl}
u_n(x,t) &\leq& C \dint_{0}^{t}\dint_{\Omega}\Big( 1\wedge \frac{\delta^s(y)}{\sqrt{t-\s}}\Big) f(y,\s)
\dfrac{(t-\s)}{((t-\s)^{\frac{1}{2s}}+|x-y|)^{N+2s}}\,dy\,d\sigma\\ &\le & C \dint_{0}^{t}\dint_{\Omega}\Big( 1\wedge
\frac{\delta^s(y)}{\sqrt{t-\s}}\Big)\frac{1}{\d^\beta(y)} \, f(y,\s)\d^\beta(y) \dfrac{(t-\s)}{((t-\s)^{\frac{1}{2s}}+|x-y|)^{N+2s}}\,dy\,d\sigma.\\
\end{array}
$$
Since
$$
\Big( 1\wedge \frac{\delta^s(y)}{\sqrt{t-\s}}\Big)\frac{1}{\d^\beta(y)}\le C(\O_T,s,\beta)(t-\s)^{-\frac{\beta}{2s}} \mbox{  in  }\O_T,
$$
we conclude that
$$
u_n(x,t)\leq C \dint_{0}^{t}\dint_{\Omega}\, f(y,\s)\d^\beta(y) \dfrac{(t-\s)^{\frac{2s-\beta}{2s}}}{((t-\s)^{\frac{1}{2s}}+|x-y|)^{N+2s}}\,dy\,d\sigma.
$$
Setting $g(y,\s)=f(y,\s)\d^\beta(y)$, then $g\in L^1(\O_T)$. We use the duality argument as in the proof of Theorem \ref{hardy0}. For the reader convenience we
include here some details.

Let $\phi\in \mathcal{C}^\infty_0(\O_T)$, then
$$
\begin{array}{rcl}
||u_n||_{L^\theta(\O_T)} &=& \dyle \sup_{\{||\phi||_{L^{\theta'}(\O_T)}\le 1\}} \iint_{\O_T}\phi(x,t) u_n(x,t)dxdt\\ &\le & \dyle
\sup_{\{||\phi||_{L^{\theta'}(\O_T)}\le 1\}} \iint_{\O_T}|\phi(x,t)| \dint_{0}^{t}\dint_{\Omega}\, g(y,\s)\d^\beta(y)
\dfrac{(t-\s)^{\frac{2s-\beta}{2s}}}{((t-\s)^{\frac{1}{2s}}+|x-y|)^{N+2s}}\,dy\,d\sigma\\ &\le & \dyle \sup_{\{||\phi||_{L^{\theta'}(\O_T)}\le 1\}}
\int_0^T\int_0^t \io \io |\phi(x,t)| H(x-y,t-\s)g(y,\s)dydx d\s dt,
\end{array}
$$
where $H(x,\s)=\dfrac{\s^{\frac{2s-\beta}{2s}}}{(\s^{\frac{1}{2s}}+|x|)^{N+2s}}$.

Using Young inequality, it holds that
\begin{equation}\label{ttbeta}
\begin{array}{rcl}
||u_n||_{L^\theta(\O_T)} &\le & C\dyle \sup_{\{||\phi||_{L^{\theta'}(\O_T)}\le 1\}} \int_0^T ||\phi(.,t)||_{L^{\theta'}(\O)} \int_0^t ||g(.,\s)||_{L^1(\O)}||H(.,
t-\s)||_{L^a(\O)}d\s dt,
\end{array}
\end{equation}
with $\dfrac{1}{\theta'}+\dfrac{1}{a}=1$  { and then $a=\theta$.}

By a direct computation we deduce that $||H(.,t-\s)||_{L^a(\O)} \le C(t-\s)^{\frac{2s-\beta}{2s}+\frac{N}{2s a}-\frac{(N+2s)}{2s}}$. Thus
$$
\begin{array}{rcl}
||u_n||_{L^\theta(\O_T)} &\le & C\dyle \sup_{\{||\phi||_{L^{\theta'}(\O_T)}\le 1\}} \int_0^T ||\phi(.,t)||_{L^{\theta'}(\O)} \int_0^t ||g(.,\s)||_{L^1(\O)}
(t-\s)^{\frac{2s-\beta}{2s}+\frac{N}{2s a}-\frac{(N+2s)}{2s}} d\s dt\\ &\le & C\dyle \sup_{\{||\phi||_{L^{\theta'}(\O_T)}\le 1\}} \int_0^T
||g(.,\s)||_{L^1(\O)}\int_{\s}^T||\phi(.,t)||_{L^{\theta'}(\O)} (t-\s)^{\frac{2s-\beta}{2s}+\frac{N}{2s a}-\frac{(N+2s)}{2s}} d\s dt.
\end{array}
$$
Using H\"older inequality, we get
$$
\begin{array}{lll}
& ||u_n||_{L^\theta(\O_T)}\le\\ & C\dyle \sup_{\{||\phi||_{L^{\theta'}(\O_T)}\le 1\}} ||\phi||_{L^{\theta'}(\O_T)} \int_0^T ||g(.,\s)||_{L^1(\O)}\bigg(\int_{\s}^T
(t-\s)^{\theta[\frac{2s-\beta}{2s}+\frac{N}{2s a}-\frac{(N+2s)}{2s}]} d\s\bigg)^\frac{1}{\theta} dt.
\end{array}
$$
The last integral is finite if and only if $\theta[\dfrac{2s-\beta}{2s}+\dfrac{N}{2s a}-\dfrac{(N+2s)}{2s}]>-1$. Since $\dfrac{1}{\theta'}+\dfrac{1}{a}=1$, then the
above condition holds if $\theta<\dfrac{N+2s}{N+\beta}$ and then the claim follows. Thus
\begin{equation}\label{eq:beta1}
||u_n||_{L^\theta(\O_T)}\le C(\O_T,s,\beta)||g||_{L^1(\O_T)}=C||f\d^\beta||_{L^1(\O_T)}.
\end{equation}
In the same way we can prove that the sequence $\{\dfrac{u_n}{\d^s}\}_n$ is bounded in $L^\theta(\O_T)$ for all $\theta<\dfrac{N+2s}{N+s+\beta}$.

 To show  that the sequence $\{|\n u_n|\}_n$ is bounded in $L^p(\O_T)$ for all $1\le p<\dfrac{N+2s}{N+\beta+1}$ we use  the same arguments as in the proof of Theorem \ref{regu-g}. We have
$$
\begin{array}{rcl}
& & |\nabla u_n(x,t)|\leq C \dint_{0}^{t} \dint_{\Omega} f_n(y,\sigma) |\nabla_x P_{\Omega} (x,y, t-\sigma)|\,dy\,d\sigma\\ \\ &\leq& C \dyle \dyle\iint_{\{\O\times (0,t)\}}f_n(y,\sigma) \frac{|\nabla_x P_{\Omega} (x,y, t-\sigma)|}{P_{\Omega} (x,y, t-\sigma)}P_{\Omega} (x,y, t-\sigma)\,dy\,d\sigma\\ \\ &\leq & C\Big(\dyle\iint_{\{\O\times (0,t) \cap\{\d(x)>t^{\frac{1}{2s}}\}\}\}} f_n(y,\s)\d^\beta(y)
\frac{P_{\Omega} (x,y, t-\sigma)}{(t-\s)^{\frac{\b+1}{2s}}}\,dy\,d\sigma \\
& + &  \dyle \dfrac{1}{\d(x)}\dyle\iint_{\{\O\times (0,t) \cap\{\d(x)\le t^{\frac{1}{2s}}\}\}\}} f_n(y,\s) P_{\Omega} (x,y, t-\sigma) \,dy\,d\sigma \bigg)\\ \\ &=& I_{1n}(x,t)+I_{2n}(x,t).
\end{array}
$$
Let us begin by estimating $I_{1n}$. Using estimate \eqref{green1} and by H\"older inequality, we get
\begin{eqnarray*}
I^q_{1n}(x,t) &=& \Big(\iint_{\{\O\times (0,t) \cap\{\d(x)>t^{\frac{1}{2s}}\}\}\}}f_n(y,\sigma)\d^\beta(y) \frac{(t-\s)^{1-\frac{1+\beta}{2s}}}{((t-\s)^{\frac{1}{2s}}+|x-y|)^{N+2s}} \,dy\,d\sigma\Big)^q\\
& \le  & C||f_n\d^\beta||^{q-1}_{L^1(\O_T)}\iint_{\{\O\times (0,t)\}\}}f_n(y,\sigma)\d^\beta(y) \frac{(t-\s)^{q\frac{2s-\beta-1}{2s}}}{((t-\s)^{\frac{1}{2s}}+|x-y|)^{q(N+2s)}} \,dy\,d\sigma.
\end{eqnarray*}
Integrating in $\O_T$, as in  the proof of Theorem \ref{regu-g}, we deduce that
$$
\dyle \iint_{\O_T}I^q_{1n}(x,t) dxdt\le CT^{\g_1+1}||f_n\d^\beta||^{q}_{L^1(\O_T)},
$$
with $\g_1=\dfrac{N}{2s}-q\dfrac{N+\beta+1}{2s}>-1$.

Now respect to $I_{2n}$, using estimate \eqref{one0} and by H\"older inequality, it follows that
\begin{eqnarray*}
I^q_{2n}(x,t)& =& \dfrac{C}{\d^q(x)}\dyle\bigg(\iint_{\{\O\times (0,t) \cap\{\d(x)\le t^{\frac{1}{2s}}\}\}\}} f_n(y,\s) P_{\Omega} (x,y, t-\sigma) \,dy\,d\sigma \bigg)^q\\ & = & \dfrac{C}{\d^q(x)}\dyle\bigg(\iint_{\{\O\times (0,t) \cap\{\d(x)\le t^{\frac{1}{2s}}\}\}\}} f_n(y,\s)\d^\beta(y) \frac{P_{\Omega} (x,y, t-\sigma)}{\d^\beta(y)}\,dy\,d\sigma \bigg)^q\\
&\le & \frac{C||f_n\d^\beta||^{q-1}_{L^1(\O_T)}}{\d^q(x)}\iint_{\{\O\times (0,t)\}\}}f_n(y,\sigma)\d^\beta(y)\frac{P^q_{\Omega} (x,y, t-\s)}{\d^{q\beta(y)}}\,dy,\,d\s\\
& \le & C||f_n\d^\beta||^{q-1}_{L^1(\O_T)}
\iint_{\{\O\times (0,t)\}\}}\frac{f_n(y,\sigma)\d^\beta(y)}{(\d(x))^{(q-1)N+q}}\frac{P_{\Omega} (x,y, t-\s)}{\d^{q\beta(y)}}dy\,d\s.
\end{eqnarray*}
Integrating in $\O_T$,
$$
\iint\limits_{\O_T}I^q_{2n}(x,t)dxdt\le C||f_n\d^\beta||^{q-1}_{L^1(\O_T)}\iint\limits_{\O_T}\frac{f_n(y,\s)\d^\beta(y)}{\d^{q\beta}(y)}\Big(\io \frac{1}{(\d(x))^{q(N+1)-N}}\int_{\s}^T P_{\Omega} (x,y, t-\s)dt\,dx\Big)dy\,d\sigma.
$$
Recalling that
$$
\dint_{\s}^{T} P_{\Omega} (x,y,
t-\s)\,dt\le \dint_{0}^{T-\s} P_{\Omega} (x,y,
\eta)\,d\eta\le \mathcal{G}_s(x,y),
$$
we find that
$$
\iint_{\O_T}I^q_{2n}(x,t)dxdt\le C||f_n\d^\beta||^{q-1}_{L^1(\O_T)}\iint_{\O_T}f_n(y,\s)\d^\beta(y)\,\frac{\varphi(y)}{\d^{q\beta}(y)}dy,
$$
where $\varphi$ is the unique solution to problem \eqref{varphi}. Since $s<q(N+1)-N<2s$ then from \cite{abatan} and \cite{Adm}, it follows that $\varphi\in L^{\infty}(\O)$ and $\varphi(y)\simeq (\d(y))^{2s-(q(N+1)-N)}$. Thus $\dfrac{\varphi(y)}{\d^{q\beta}(y)}\simeq (\d(y))^{2s-(q(N+1)-N)-q\beta}$.
It is clear that $2s-(q(N+1)-N)-q\beta>0$ if and only if $q<\dfrac{N+2s}{N+1+\beta}$, which is the  hypothesis. Thus
$$
\iint_{\O_T}I^q_{2n}(x,t)(x,t)dxdt\le C||f_n\d^\beta||^{q}_{L^1(\O_T)}.
$$
As a conclusion, we have proved that for all $q<\frac{N+2s}{N+1+\beta}$,
$$
||\n u_n||_{L^q(\O_T)}\le C(\O_T)||f_n\d^\beta||_{L^1(\O_T)}.
$$
Hence  there exists a solution $u$ to problem \eqref{eq:def}  in the
sense of distributions such that
$$u\in  L^q(0,T; W_{0}^{1,q}(\Omega))\cap\mathcal{C}([0,T], L^1(\O, \d^\beta dx)),
\hbox{ for all } q<\frac{N+2s}{N+\beta+1}.$$
It is clear that if $u_1,u_2\in L^q(0,T; W_{0}^{1,q}(\Omega))\cap\mathcal{C}([0,T], L^1(\O,\d^\beta dx))$ are solutions
to \eqref{eq:def}, then the difference $v=u_1-u_2$, satisfies $v\in L^q(0,T;
W_{0}^{1,q}(\Omega))\cap\mathcal{C}([0,T], L^1(\O, \d^\beta dx))$ and
\begin{equation}\label{vvv}
\left\{
\begin{array}{rcll}
v_{t}+(-\D)^s v&=& \dyle 0 & \text{ in } \O_{T}, \\ v&=&0 & \text{ in }(\ren\setminus\O) \times (0,T), \\ v(x,0)&=&0& \mbox{  in  }\O.
\end{array}%
\right.
\end{equation}
Using Kato inequality as in Theorem \ref{Conv}, we reach that $v_+\in L^q(0,T; W_{0}^{1,q}(\Omega))\cap\mathcal{C}([0,T], L^1(\O, \d^\beta dx))$ satisfies
\begin{equation}\label{v++}
(v_{+})_t+(-\D)^s v_+\le 0 \text{ in } \O_{T}.
\end{equation}
Hence using $\phi_1$, the positive first eigenfunction of the fractional Laplacian in $\O$, as a test function in \eqref{v++}, we reach that $v_+=0$. In the same
way and since $-v$ is also a solution to \eqref{vvv}, we conclude that $v_-=0$. Thus $v=0$.

Now setting $$\hat{\mathcal{K}}: L^{1}(\Omega_T, \d^\beta(x)dxdt)\times L^1(\O)\to L^q(0,T; W_{0}^{1,q}(\Omega)), q<\dfrac{N+2s}{N+\beta+1},$$
where  $\hat{\mathcal{K}}(f, u_0)=u$ is
the unique solution to problem \eqref{eq:def}, then as in the proof
 of Theorem \ref{gradiente}, taking advantage of the linearity of the operator,
 we conclude that $\hat{\mathcal{K}}$ is a compact operator.
\end{proof}
As in proposition \ref{key2-locc}, if
$f\in L^{1}(\Omega_T, \d^\beta(x)dxdt)\cap L^m(K\times (0,T))$, with $m>1$
and $K\subset\subset \Omega$ is any compact set of $\Omega$, then we have the next  general regularity result.
\begin{Proposition}\label{key2-beta}
{ Let $m>1$. Assume that $f\in L^{1}(\Omega_T, \d^\beta(x)dxdt)\cap L^m(K\times (0,T))$ for any compact set $K$ of $\O$.} Define  $u$ to be the
unique weak solution to problem \eqref{eq:def} and let $\O_1\subset\subset \O$ with $\text{dist}(\O_1,\p\O)>0$. Consider  $K_1\subset\subset \O$, a compact set of
$\O$ such that $\O_1\subset\subset K$.

Then $u\in L^\theta(\O_1\times (0,T))$ for all $\theta<\dfrac{m (N+2s)}{(N+2s-m(2s-\beta))_+}$ and $|\n u|\in
L^p(\O_1\times (0,T))$ for all $p<\dfrac{m (N+2s)}{(N+2s-m(2s-1-\beta))_+}$. Moveover
\begin{equation}\label{eqq1beta}
||u||_{L^\theta(\O_1\times (0,T))}+ \|\n u\|_{L^p(\O_1\times (0,T))}\le C (||f||_{L^m(K_1\times (0,T))}+||f\d^{\beta}||_{L^1(\O_T)}),
\end{equation}
where $C:=C(K_1,\O_1, \O, T,N,m)$.
\end{Proposition}
\begin{proof}
Since $f\d^\beta\in L^1(\O_T)$, then $|\n u|\in L^{q}(\O_T)$ for all $q<\frac{N+2s}{N+\beta+1}$. As in the proof of Proposition \ref{key2-locc}, we have
$$
u_n(x,t)\leq C \dint_{0}^{t}\dint_{\Omega}\, f(y,\s)\d^\beta(y) \dfrac{(t-\s)^{\frac{2s-\beta}{2s}}}{((t-\s)^{\frac{1}{2s}}+|x-y|)^{N+2s}}\,dy\,d\sigma.
$$
Fix $\O_1\subset\subset \O$ with $\text{dist}(\O_1,\p\O)=c_0>0$ and let $K_1$ be a compact set of $\O$ such that $\O_1\subset\subset K_1\subset\subset \O$. Let
$x\in \O_1$, then
$$
\begin{array}{rcl}
u(x,t) &\le & C(\O_1, c_0, C)\bigg\{\dint_{0}^{t}\dint_{K_1}f(y,\s)\d^\beta(y) \dfrac{(t-\s)^{\frac{2s-\beta}{2s}}
}{((t-\s)^{\frac{1}{2s}}+|x-y|)^{N+2s}}\,dy\,d\sigma\\ & + & \dint_{0}^{t}\dint_{\O\backslash K_1}f(y,\s)\d^\beta(y)
\dfrac{(t-\s)^{\frac{2s-\beta}{2s }}}{((t-\s)^{\frac{1}{2s}}+|x-y|)^{N+2s}}\,dy\,d\sigma\bigg\}.
\end{array}
$$
If $(x,y)\in (\O_1\subset\subset K_1)\times (\O\backslash K_1)$, then $|x-y|>\hat{c}>0$. Hence
$$
\begin{array}{rcl}
\dint_{0}^{t}\dint_{\O\backslash K_1}f(y,\s)\d^\beta(y) \dfrac{(t-\s)^{\frac{2s-\beta}{2s}}}{((t-\s)^{\frac{1}{2s}}+|x-y|)^{N+2s}}\,dy\,d\sigma\le C
||f\d^\beta||_{L^1(\O_T)}.
\end{array}
$$
Thus
$$
\begin{array}{rcl}
u(x,t) &\le & C(\O_1, c_0, C)\bigg(\dint_{0}^{t}\dint_{K_1}f(y,\s)\d^\beta(y) \dfrac{(t-\s)^{\frac{2s-\beta}{2s}}
\,dy\,d\sigma}{((t-\s)^{\frac{1}{2s}}+|x-y|)^{N+2s}} +||f\d^\beta ||_{L^1(\O_T)}\bigg).
\end{array}
$$
Since $f\d^\beta\in L^m(K_1\times (0,T))$, then we exactly follow  the same duality argument as in the proof of Theorem \ref{hardy0}. Hence estimate
\eqref{eqq1beta} follows. In a similar way we prove estimate  in the gradient.
\end{proof}
\begin{remarks}\label{rm00}
To obtain the above regularity result we have used the fact that if
$$
g(x,t):=\dint_{0}^{t}\dint_{\Omega}f(y,\s)\dfrac{(t-\s)^a}{((t-\s)^{\frac{1}{2s}}+|x-y|)^{N+2s}}\,dy\,d\sigma
$$
with $f\in L^m(\O_T), m\ge 1$ and $a>0$, then $g\in L^\gamma(\O_T)$ where $\g$ satisfies
$$
\frac{1}{\g}>\frac{1}{m}-\frac{2s a}{N+2s}.
$$
This result will be used systematically in what follows.
\end{remarks}

\section{Non existence result.}
In the local case, $s=1$, existence of solution holds for all $\a>0$ at least for $f\in L^\infty(\O_T)$ and $u_0\in L^\infty(\O)$, see for instance \cite{BS}, \cite{GGK} and \cite{QS}. However in the nonlocal case, $s<1$,  and by the lack of regularity  near of the boundary  $\partial\Omega$,  a threshold on $\alpha $ appears for the existence.
This behavior  represents  a deep difference with the local case, though it is stable when $s\to 1$.
Let us recall and prove the non existence result stated in the introduction.
\begin{Theorem}\label{non1}
Assume that $\a\ge \dfrac{1}{1-s}$, then for all nonnegative data $(f,u_0)\in L^\infty(\O_T)\times L^\infty(\O)$ with $(f,u_0)\neq (0,0)$, the problem
\begin{equation}\label{grad-nn}
\left\{
\begin{array}{rcll}
u_t+(-\Delta )^s u&=&|\nabla u|^{\alpha}+ f & \inn \Omega_T,\\ u(x,t)&=&0 & \inn(\mathbb{R}^N\setminus\Omega)\times [0,T),\\
u(x,0)&=&u_{0}(x) & \inn\Omega,\\
\end{array}\right.
\end{equation}
has no weak solution $u$ in the sense of Definition \ref{veryweak} with $u\in L^\a(0,T;W^{1,\alpha}_0(\O))$.
\end{Theorem}
Before starting with the proof of Theorem \ref{non1}, we need the following
result that extends to the fractional framework the one proved in \cite{Mart}
for the heat equation. The result is proved in \cite{BM} using apriori estimates. We  give here a different proof using the properties of the Dirichlet heat kernel.
\begin{Proposition}\label{diss}
Assume that the condition on $(f,u_0)$ of the above Theorem holds. Let $w$ be the unique solution to the problem
\begin{equation}\label{grad-nn11}
\left\{
\begin{array}{rcll}
w_t+(-\Delta )^s w&=& f & \inn \Omega_T,\\ w(x,t)&=&0 & \inn(\mathbb{R}^N\setminus\Omega)\times [0,T),\\ w(x,0)&=&u_{0}(x) &
\inn\Omega,\\
\end{array}\right.
\end{equation}
with either $u_0\gneqq 0$ or $f\gneqq 0$ in $\O\times (0,t_0)$, $t_0<T$ being fixed. Then there exists $C:=C(t_0,\O,u_0,f)>0$ such that
\begin{equation}\label{contr}
w(x,t_0)\ge C \d^s(x) \mbox{  in   }\O.
\end{equation}
\end{Proposition}
\begin{proof}
Notice that
\begin{eqnarray*}
w(x,t) &= & \dint_{\Omega}u_0(y) P_{\Omega} (x,y, t)\,dy\,+ \dint_{0}^{t} \dint_{\Omega} f(y,\sigma) P_{\Omega} (x,y, t-\sigma)\,dy\,d\sigma\\ &\ge &
\dint_{\Omega}u_0(y) P_{\Omega} (x,y, t)\,dy.
\end{eqnarray*}
{ Suppose in the first case that $u_0\gneqq 0$, then
$$
w(x,t) \ge \dint_{\Omega}u_0(y) P_{\Omega} (x,y, t)\,dy.
$$
}
Now fix $t_0\in (0,T)$, using the estimate on $P_\O$ given in \eqref{green1}, we deduce that
$$
w(x,t_0)\ge C\Big( 1\wedge \frac{\delta^s(x)}{\sqrt{t_0}}\Big)\ \io \Big( 1\wedge \frac{\delta^s(y)}{\sqrt{t_0}}\Big)\times \Big(t_0^{-\frac{N}{2s}}\wedge
\frac{t_0}{|x-y|^{N+2s}}\Big) u_0(y) dy.
$$
Notice that for all $x, y\in \O$, we have
$$
\Big(t_0^{-\frac{N}{2s}}\wedge \frac{t_0}{|x-y|^{N+2s}}\Big)\ge C\dfrac{t_0}{(t_0^{\frac{1}{2s}}+|x-y|)^{N+2s}}\ge
C\dfrac{t_0}{(t_0^{\frac{1}{2s}}+\text{diam}(\O))^{N+2s}},
$$
and
$$
\Big( 1\wedge \frac{\delta^s(x)}{\sqrt{t_0}}\Big)\ge \d^s(x)\min\{\frac{1}{\sqrt{t_0}}, \frac{1}{\delta^s(x)}\}\ge \d^s(x)\min\{\frac{1}{\sqrt{t_0}},
\frac{1}{\text{diam}^s(\O)}\}={{\d^s(x)}}C(t_0,\O).
$$
Hence
$$
w(x,t_0)\ge R(t_0)\delta^s(x)\io u_0(y)\d^s(y)dy,
$$
with
$$
R(t)=\bigg(\dfrac{t_0}{(t_0^{\frac{1}{2s}}+\text{diam}(\O))^{N+2s}} (\min\{\frac{1}{\sqrt{t_0}},
\frac{1}{\text{diam}^s(\O)}\})^2\bigg).
$$
Notice that
$$
R(t)\ge
\left\{
\begin{array}{lll}
\dfrac{C(\O)}{t^{\frac{N+2s}{2s}}} & \mbox{  if  } & t\ge (\text{diam}(\O))^{2s}\\ \\
C(\O)t & \mbox{  if  } & t\le (\text{diam}(\O))^{2s}.
\end{array}
\right.
$$
Thus $w(x,t_0)\ge C(t_0,u_0,\O)\d^s(x)$ which is the desire estimate.\\

{ Consider now the case where $f\gneqq 0$ in $\O\times (0,t_0)$. As above we deduce that
$$
w(x,t) \ge \dint_{0}^{t} \dint_{\Omega} f(y,\sigma) P_{\Omega} (x,y, t-\sigma)\,dy\,d\sigma.
$$
}
For $t_0>0$ fixed, using the properties of the kernel $P_\O$, given in \eqref{green1}, it holds that
$$
w(x,t_0)\ge C \dint_{0}^{t_0} \dint_{\Omega} \Big( 1\wedge \frac{\delta^s(x)}{\sqrt{t_0-\s}}\Big)\, \Big( 1\wedge \frac{\delta^s(y)}{\sqrt{t_0-\s}}\Big)\times \Big((t_0-\s)^{-\frac{N}{2s}}\wedge
\frac{t_0-\s}{|x-y|^{N+2s}}\Big) f(y,\s)dy\,d\s.
$$
As in the first case, it follows that
$$
w(x,t_0)\ge C \d^s(x)\dint_{0}^{t_0} \dint_{\Omega} R(t_0-\s)f(y,\s)dy\,d\s.
$$
Since $R_0(t_0-\s)\ge C(T,\O)(t_0-\s)$, then using the fact that $f\gneqq 0$ in  $\O\times (0,t_0$, it follows that
$$
\dint_{0}^{t_0} \dint_{\Omega} R(t_0-\s)f(y,\s)dy\,d\s=C(t_0,f,\O)>0.
$$
Thus we conclude.
\end{proof}
\begin{remark}
It is clear that if $t\in (t_1,t_2)\subset \subset (0,T)$, by
Proposition \ref{diss} we have that, for all $(x,t)\in \O\times (t_1,t_2)$,
$$
w(x,t)\ge C(t_0,t_1,\O) \delta^s(x)\io u_0(y)\d^s(y)dy.
$$
\end{remark}

We are now in position to prove the non existence result in Theorem \ref{non1}.

\noindent{\bf Proof of Theorem \ref{non1}}. We argue by contradiction. Assume that $\a\ge \dfrac{1}{1-s}$ and suppose that problem \eqref{grad-nn} has a solution $u\in
L^\a(0,T; W^{1,\a}_0(\O))$. Fix $(t_1,t_2)\subset \subset (0,T)$, then by Proposition \ref{diss}, we reach that
$$
u(x,t)\ge C(t_0,t_1\O) \delta^s(x)\io u_0(y)\d^s(y)dy \mbox{  for all  }(x,t)\in \O\times (t_1,t_2).
$$
Since $u\in L^\a(0,T; W^{1,\a}_0(\O))$, using Hardy inequality it holds that $\dfrac{u}{\d}\in L^\a(\O\times (t_1,t_2))$. Thus
$$
C(t_0,t_1\O) \int_{t_1}^{t_2}\io\frac{\delta^{\a s}(x)}{\d^\a(x)}dx\le C\int_{t_1}^{t_2} \io \frac{u^\a(x,t)}{\d^\a(x)}dx<\infty.
$$
Since $\a\ge \dfrac{1}{1-s}$, then $\a(1-s)\ge 1$ and then we reach a
contradiction. Hence we conclude. $\Box$

If we are dealing with weak solution in the sense of Definition \ref{veryweak},
we can prove a non existence result for a suitable range
of $\a$. Before starting with the non existence result,
we recall the weighted Hardy inequality proved in \cite{Ne} (see also \cite{LE}).
\begin{Proposition}\label{hardygg}
Assume that $\alpha>1$ and $0<\s<\a-1$, then for all $\phi\in W^{1,\a}_0(\O)$, we have
\begin{equation}\label{eq:super-hardy}
\io \dfrac{|\phi(x)|^\a}{\d^{\a-\s}(x)}dx\le C\io |\n \phi(x)|^\a \d^\s(x)\, dx,
\end{equation}
where $C:=C(\O,p,N)$.
\end{Proposition}
Define the space
$$
\widehat{W}_{\a,\s}(\O):=\{\varphi\in W^{1,1}_0(\O) \mbox{  with  }\io |\n \varphi|^\a \d^\s dx<\infty\}.
$$
If $\s+1<\a$, then using H\"older inequality, the space $\widehat{W}_{\a,\s}(\O)$ can be endowed with the norm
$$
||\varphi||_{\widehat{W}_{\a,\s}(\O)}=\bigg(\io |\n \varphi|^\a \d^\s dx\bigg)^{\frac{1}{\a}}.
$$
Let $\widehat{H}_{\a,\s}(\O)$ be the completion of $\mathcal{C}^\infty_0(\O)$ with respect to the norm of $\widehat{W}_{\a,\s}(\O)$.

Notice  that $\d^\s$ belongs to the Muckenhoupt class $A_\a$ if { $0<(\s+1)<\a$, see for instance Theorem 3.1 in \cite{Dur}}.
Hence by the results of \cite{ZZ}, we reach that
$\widehat{H}_{\a,\s}(\O)=\widehat{W}_{\a,\s}(\O)$. Thus,
for all $\varphi\in \widehat{W}_{\a,\s}(\O)$ and by Proposition \ref{hardygg}, for all
$\varphi\in \widehat{W}_{\a,\s}(\O)$, we have
\begin{equation}\label{super-hardy}
C(\O,N) \io \dfrac{|\varphi(x)|^\a}{\d^{\a-\s}(x)}dx \le\io |\n \varphi|^\a \d^\s dx.
\end{equation}
As a consequence, we have the next general non existence result.
\begin{Theorem}\label{non100}
Assume that $\a\ge \dfrac{1+\beta}{1-s}$ for some $\beta>0$. Then for all nonnegative data $(f,u_0)\in L^\infty(\O_T)\times L^\infty(\O)$ with  $(f,u_0)\neq (0,0)$, the problem \eqref{grad-nn} has no weak solution $u$ in the sense of
Definition \ref{veryweak} such that $u\in L^1(0,T;W^{1,1}_0(\O))$ with $|\n u|^\a \d^\b\in L^1(\O_T)$.
\end{Theorem}
\begin{proof} We argue by contradiction. Assume that the above conditions hold and that problem \eqref{grad} has a weak solution $u$ with $u\in L^1(0,T;
W^{1,1}_0(\O))$ and $|\n u|^\a \d^{\b}\in L^1(\O_T)$. Since $\beta+1<\a$, then $u\in \widehat{W}_{\a,\beta}(\O)$.
Fix $(t_1,t_2)\subset \subset (0,T)$, then by Proposition \ref{diss},
we reach that
\begin{equation}\label{rrrp}
u(x,t)\ge C(t_0,t_1,\O) \delta^s(x)\mbox{  for all  }(x,t)\in \O\times (t_1,t_2).
\end{equation}
Since $u\in L^\a(0,T; \widehat{W}_{\a,\beta}(\O))$, using the weighted Hardy
inequality \eqref{super-hardy}, it holds that
$$
C(\O,N) \io \dfrac{u^\a(x,t)}{\d^{\a-\b}(x)}dx\le \io |\n u(x,t)|^\a \d^{\b} dx \mbox{  a.e.  in } (0,T).
$$
Integrating in the time,
$$
C(\O,N) \int_{t_0}^{t_1}\io \dfrac{u^\a(x,t)}{\d^{\a-\b}(x)}dx \le \int_{t_0}^{t_1}\io |\n u(x,t)|^\a \d^{\b} dx<\infty.
$$
Hence by estimate \eqref{rrrp} we reach that
$$
C(t_0,t_1,\O) \io \frac{1}{(\delta(x))^{\a(1-s)-\b}} dx\le \int_{t_0}^{t_1}\io |\n u(x,t)|^\a \d^{1-s} dx<\infty.
$$
Using the fact that $\a(1-s)-\b\ge 1$, we reach a contradiction.
\end{proof}

\begin{Corollary}

\

\begin{enumerate}
\item Assume that $\a\ge \dfrac{1+s}{1-s}$, then problem \eqref{grad-nn}
has no weak solution $u$ in the sense of Definition \ref{veryweak} with $u\in
    L^1(0,T;W^{1,1}_0(\O))$ and $|\n u|^\a \d^s\in L^1(\O_T)$.
\item Assume that $\a\ge \dfrac{2s}{1-s}$, then problem \eqref{grad-nn} has non
weak solution $u$ in the sense of Definition \ref{veryweak} with
$u\in L^1(0,T;W^{1,1}_0(\O))$ and  $|\n u|^\a \d^\b\in L^1(\O_T)$ for some
$\beta<2s-1$.
\end{enumerate}
\end{Corollary}

\begin{remarks}
The above non existence results make a significative difference with respect
to the local case $s=1$, where an existence result holds for all $\a>1$ under
suitable condition of the data. See for example \cite{BS} and the reference
therein.
\end{remarks}
\section{The existence results.}\label{sec4}
The  goal of this section is to prove an existence result for problem
\eqref{grad} under suitable condition on $\alpha$ and the data $f$ as was established in the Introduction.
\subsection{Existence result for $L^1$ data and $\a<\dfrac{N+2s}{N+1}$.}

In this subsection we will prove  Theorem \ref{maria}. We assume that $(f,u_0)\in L^m(\O_{\overline{T}})\times L^1(\O)$ where
$1\le m<\dfrac{1}{s}$. By the regularity result in Theorem \ref{gradiente} and
the second point in Corollary \ref{cor11}, it follows that the condition
$\a<\dfrac{N+2s}{N+1}$ is natural in order to get the existence of a
solution to problem \eqref{grad}.

\noindent{\bf Proof of Theorem \ref{maria}.} {Let $T<\overline{T}$ to be chosen later} and define
$E_q(\O_T)=L^q(0,T; W_{0}^{1,q}(\Omega))$ for $q\ge 1$. Assume that
$1\le m<\dfrac{1}{s}$ and fix $\a<\frac{N+2s}{N+1}$. Let $l>0$ to be chosen later. Define the set
$$
E(\O_T)=\{v\in E_1(\O_T)\mbox{  such that  }v\in E_r(\O_T) \mbox{  with  }\alpha<r<\frac{N+2s}{N+1} \mbox{  and  }\|v\|_{E_r(\O_T)}\leq l^{\frac{1}{\alpha}}\}.
$$
It is easy to check that $E(\O_T)$ is a closed convex set of $E_1(\O_T)$. For $(f,u_0)\in L^m(\O_T)\times L^1(\O)$ fixed, we consider the operator
$$
\begin{array}{rcl}
K:E(\O_T) &\rightarrow&  E_1(\O_T)\\
   v&\rightarrow&T(v)=u
\end{array}
$$
where $u$ is the unique solution to problem
\begin{equation}\label{gradff}
\left\{
\begin{array}{rcll}
u_t+(-\Delta )^s u&=&|\nabla v|^{\alpha}+ f & \inn \Omega_T\equiv\Omega\times (0,T),\\ u(x,t)&=&0 & \inn(\mathbb{R}^N\setminus\Omega)\times [0,T),\\ u(x,0)&=&
u_0(x) & \inn\Omega.\\
\end{array}\right.
\end{equation}
Since $\alpha<r$, then $|\nabla v|^{\alpha}+f \in L^1(\O_T)$, thus the existence of $u$ is a consequence of \cite{LPPS}  and moreover $u\in L^q(0,T;
W_{0}^{1,q}(\Omega))\cap \mathcal{C}([0,T];L^1(\O))$ for all $q<\dfrac{N+2s}{N+1}.$ Hence $K$ is well defined.

We claim that:
\begin{enumerate}
\item There exists $l>0$ such that  $K(E(\O_T))\subset E(\O_T)$,
\item $K$ is a continuous and  compact operator on $E(\O_T)$.
\end{enumerate}

\textit{ Proof of the claim.}
First, we prove that we can find $l>0$ such that $K(E(\O_T))\subset E(\O_T)$. In fact, thanks to Theorem
\ref{gradiente}, we have that $u\in E_q(\O_T)$ for all $q<\dfrac{N+2s}{N+1}$. In
particular, $u\in E_r(\O_T)$. Now,  fixed $q<\dfrac{N+2s}{N+1}$, it follows that
\begin{eqnarray*}
\|u\|_{E_q(\O_T)} &\leq & C(T,\O) \Big(|| f+|\nabla v|^\alpha||_{L^{1}(\Omega_T)}+ ||u_0||_{L^1(\O)}\Big)\\ &\le & C(T,\O)\bigg(||f||_{L^{1}(\Omega_T)}+||\nabla
v||^\alpha_{L^{\alpha}(\Omega_T)}+||u_0||_{L^1(\O)}\bigg)\\ &\le & C(T,\O)\bigg( ||f||_{L^{1}(\Omega_T)}+||\nabla
v||^{\a}_{L^{r}(\Omega_T)}+||u_0||_{L^1(\O)}\bigg)\\ &\le & C(T,\O)\bigg( ||f||_{L^{1}(\Omega_T)}+||v||^\alpha_{E_r}+||u_0||_{L^1(\O)}\bigg)\\ &\le
& C(T,\O)\bigg( ||f||_{L^{1}(\Omega_T)}+||u_0||_{L^1(\O)}+ l\bigg).
\end{eqnarray*}
Recall that, by Remark \ref{mainrr}, we know that $C(\O,T)\to 0$ as $T\to 0$,
then choosing $T$ small, there exists $l>0$ such that $C(T,\O)\bigg(
||f||_{L^{1}(\Omega_T)}+||u_0||_{L^1(\O)}+l\bigg)\le l^{\frac{1}{\alpha}}$. Hence
$$
\|u\|_{E_r(\O_T)}\leq l^{\frac{1}{\alpha}},
$$
and then $u\in E(\O_T)$. Thus $K(E(\O_T))\subset E(\O_T)$. {Hence, from now, we fix $T<\overline{T}$ such that the above conclusions hold.}

 To prove the continuity of $K$ with respect to the topology of $E_1(\O_T)$,
 we consider $\{v_n\}\subset E(\O_T)$ such that $v_n\rightarrow v$ strongly in
 $E_1(\O_T)$. Define $u_n=K(v_n)$, $u=K(v)$ and $w_n=u_n-u$. We have to show that  $u_n\rightarrow u$ strongly in $E_1(\O_T)$.
Since
$$
w_{nt}+(-\Delta )^s w_n=|\nabla v_n|^{\alpha}-|\nabla v|^\alpha,
$$
to show that $w_n\to 0$ strongly in $E_1(\O_T)$, we have to prove that $||\n v_n-\n v||_{_{L^{\alpha}(\Omega_T)}}\to 0$ as $n\to \infty.$

Recall that $\{v_n\}_n\subset E(\O_T)$ and $||v_n-v||_{E_1(\O_T)}\to 0$ as $n\to \infty$, then $\n v_n\to \n v$ strongly in $(L^1(\Omega_T))^N$ and $||\n
v_n||_{L^{r}(\Omega_T)}\le C$.

Since $1<\alpha<r$, then using Hölder inequality, we reach that
\begin{eqnarray*}
||\n v_n-\n v||_{L^{\alpha}(\Omega_T)} & \le & ||\n v_n-\n v||^{\frac{r-\alpha}{1+r}}_{L^{1}(\Omega_T)} ||\n v_n-\n v||^{\frac{\alpha}{1+r}}_{L^{r}(\Omega_T)}\\
&\le & C||\n v_n-\n v||^{\frac{r-\alpha}{1+r}}_{L^{1}(\Omega)}\to 0\mbox{  as  }n\to \infty.
\end{eqnarray*}
Now, by using the definition of $u_n$ and $u$, there results that $u_n\to u$ strongly in $E_1(\O_T)$. Thus $K$ is continuous.

To finish we have just to show that $K$ is a compact operator with respect to the topology of $E_1(\O_T)$.
Let $\{v_n\}_n\subset E(\O_T)$ be such that $||v_n||_{E_1(\O_T)}\le C$. Since $\{v_n\}_n\subset E_r(\O_T)$, then $||v_n||_{E_r}\le C$ and therefore up to a
subsequence, $v_{n}\rightharpoonup v$ weakly in $E_r(\O_T)$.

Since $\a<r$, then there exists $\d>0$ such that the sequence $\{|\nabla v_n|^{\alpha}\}_n$ is bounded in $L^{1+\delta}(\Omega_T)$. Thus, up to a subsequence,
$$
|\nabla v_n|^{\alpha}\rightharpoonup g\mbox{  weakly  in   }L^{1+\delta}(\Omega_T).
$$
Let $u$ to be the unique solution to the problem
\begin{equation}\label{grad011}
\left\{
\begin{array}{rcll}
u_t+(-\Delta )^s u&=& g+f &\inn \Omega_T,\\ u(x,t)&=&0 &\inn(\mathbb{R}^N\setminus\Omega)\times [0,T),\\ u(x,0)&=& u_0(x) & \inn\Omega,\\
\end{array}\right.
\end{equation}
using the compactness result of Theorem \ref{gradiente}, we conclude that, up to a subsequence, $u_{n}\to u$ strongly in $E_q(\O_T)$  for all
$q<\frac{N+2s}{N+1}$. In particular $u_{n}\to u$ strongly in $E_r(\O_T)$, hence the claim follows.

As a conclusion and using the Schauder Fixed Point Theorem, there exists $u\in E(\O_T)$ such that $K(u)=u$, then $u\in L^q(0,T; W_{0}^{1,q}(\Omega))$ and $u$
solves \eqref{grad}. It is not difficult to show that
$T_k(u)\in L^{2}(0,T;H^s_0(\Omega))$ for all $k>0$.\hskip 10.5cm {$\square$}

\begin{remark}

\

\begin{enumerate}
\item Using a suitable approximation argument we can prove that the
existence result of Theorem \ref{maria} holds if $f$ is a bounded Radon measure.
\item The existence result of Theorem \ref{maria} is optimal in the sense that if $\a>\dfrac{N+2s}{N+1}$, then we can find $f\in L^1(\O_T)$ or $u_0\in L^1(\O)$ such that
    problem \eqref{grad} has no solution in the space $L^\a(0,T; W^{1,\a}_0(\O))$. To see  the optimality condition, we will use Remark \ref{optimal}.\\
    Fix $0\le f\in L^1(\O_T)$ or $0\le u_0\in L^1(\O)$ such that problem
    \eqref{eq:def} has a solution $v$ with $v^m\notin L^1(\O_T)$, being
    $m=\dfrac{N+2s}{N}$. Suppose now that problem \eqref{maria} has a solution
    $u$ with $\a\ge \dfrac{N+2s}{N+1}$. By the comparison principle we easily reach
    that $u\ge v$. Since $u\in L^\a(0,T; W^{1,\a}_0(\O))\cap \mathcal{C}([0,T]; L^1(\O))$, then we obtain that $u\in L^\s(\O_T)$ where $\s=\a\dfrac{N+1}{N}$.
    Since $v\le u$, then $v\in L^\s(\O_T)$. It is clear that $\s>\dfrac{N+2s}{N}$, hence we reach a contradiction with the  hypotheses on $v$.
\item For the uniqueness and the global existence in the time, we refer to Theorem \ref{uniqq}.
\end{enumerate}
\end{remark}

\subsection{Existence results for $\dfrac{N+2s}{N+1}\le \a$.
Proofs of Theorem \ref{hhh} and  Theorem  \ref{fix001}}\label{sub:sec41}
In this subsection we assume that $\dfrac{N+2s}{N+1}\le \a$. According to the regularity of $f$ and $u_0$, we are able to show the existence of a solution that is in a suitable Sobolev space, under suitable condition on $m$ and $s$. For simplicity of presentation we will consider two cases:
\begin{itemize}
\item $f\neq 0$, $u_0=0$
and
\item $f=0$, $u_0\neq 0$.
\end{itemize}

{\bf Proof of Theorem \ref{hhh}.}

Recall that $\dfrac{2s-1}{1-s}>\dfrac{(N+2s)^2}{N+1}$ and $\dfrac{N+2s}{N+1}\le \a<\dfrac{2s-1}{(1-s)(N+2s)}$.

Assume that $f\in L^m(\O_T)$ with $m\ge \dfrac{1}{s}$.
We present a proof in the case $m<\dfrac{N+2s}{2s-1}$. The other case follows in a similar way.

Since $\frac{N+2s}{\a'}\frac{1}{(2s-1)-(1-s)(N+2s)}<m$, then $\a<\dfrac{(N+2s)}{(N+2s)(m(1-s)+1)-m(2s-1)}$. It is clear that $m>\dfrac{2s-1}{2s-1-(1-s)(N+2s)}$. Hence using the fact that
$$\dfrac{N+2s}{N+1}\le \a<\dfrac{(N+2s)}{(N+2s)(m(1-s)+1)-m(2s-1)},$$ we get the existence of $r$ such that  $m\a<r<\dfrac{m(N+2s)}{(N+2s)(m(1-s)+1)-m(2s-1)}=\check{P}$ defined in Corollary \ref{cor11}.

Recall that $E_\s(\O_T)\equiv L^\s(0,T; W_{0}^{1,\s}(\Omega))$. Define the set
\begin{equation}\label{set}
E(\O_T)=\{v\in E_1(\O_T)\mbox{  such that  }v\in E_{r}(\O_T) \mbox{  with  } \|v\|_{E_{r}(\O_T)}\leq l^{\frac{1}{\alpha}}\}.
\end{equation}
Then $E(\O_T)$ is a closed convex set of $E_1(\O_T)$. Setting
$$
\begin{array}{rcl}
K:E(\O_T) &\rightarrow&  E_1(\O_T)\\
   v&\rightarrow&K(v)=u
\end{array}
$$
where $u$ is the unique solution to problem
\begin{equation}\label{gradss}
\left\{
\begin{array}{rcll}
u_t+(-\Delta )^s u&=&|\nabla v|^{\alpha}+ f &\inn \Omega_T,\\ u(x,t)&=&0 &\inn(\mathbb{R}^N\setminus\Omega)\times [0,T),\\ u(x,0)&=& 0 & \inn\Omega.
\end{array}\right.
\end{equation}
Since $m\a<r< \dfrac{m(N+2s)}{(N+2s)(m(1-s)+1)-ms}$, using the regularity result in Corollary \ref{cor11} and as in the proof of Theorem \ref{maria}, there exists $T>0$ such that $K(E(\O_T))\subset E(\O_T)$ and that $K$ is a continuous, compact operator on $E(\O_T)$. Using the Schauder Fixed Point Theorem, we
get the existence of $u\in E(\O_T)$ such that $K(u)=u$, then $u\in L^r(0,T; W_{0}^{1,r}(\Omega))$ and $u$ solves \eqref{grad}. \cqd

\begin{remarks}
Notice that

\begin{enumerate}
\item If $m>\dfrac{N+2s}{2s-1}$, then the critical growth range $\a=2s$ is covered by the existence result of Theorem \ref{hhh} if
$
\dfrac{2s-1}{(1-s)}>2s(N+2s), \mbox{  which holds at least for $s$ close to $1$}.
$
\item If $\dfrac{1}{s}<m\le \dfrac{N+2s}{2s-1}$, then the critical growth range $\a=2s$ is covered by the existence result of Theorem \ref{hhh} if
\begin{equation}\label{tmm}
m>\dfrac{N+2s}{2s}\dfrac{2s-1}{(2s-1)-(1-s)(N+2s)}.
\end{equation}
Since in this case $m\le \dfrac{N+2s}{2s-1}$, then we deduce that from \eqref{tmm}, we have $\dfrac{2s-1}{(1-s)}>2s(N+2s)$.
\end{enumerate}
\end{remarks}
We deal now with  the complete range of the parameter $\a$, $\dfrac{N+2s}{N+1}\le \a<\dfrac{s}{1-s}$. In this case, under suitable hypothesis on $m$, we will show the existence of a distributional solution that is in a suitable weighted Sobolev space. This result shows the influence of the singularity of the kernel on the boundary, which is a relevant difference with the heat equation.

{\bf Proof of Theorem \ref{fix001}.} Without loss of generality we can assume that $N\ge 2$.

Fix $\a$ such that $\dfrac{N+2s}{N+1}\le \a<\dfrac{s}{1-s}$ and suppose that $f\in L^m(\O_T)$ with
$$m>\max\left\{\dfrac{N+2s}{s(2s-1)}, \dfrac{N+2s}{s-\a(1-s)}\right\}.$$

Recall that if $v$ solves the problem
$$
\left\{
\begin{array}{rcll}
v_t+(-\Delta )^s v &=& f &\inn \Omega_T,\\ v(x,t)&=&0 & \inn(\mathbb{R}^N\setminus\Omega)\times [0,T),\\ v(x,0)&=& 0 & \inn\Omega,\\
\end{array}
\right.
$$
then by Theorem \ref{regu-g}, for all $p<\infty$, we have
\begin{equation}\label{fix2}
|||\n v|\,\d^{1-s}||_{L^p(\Omega_T)}\le C_0||f||_{L^m(\O_T)},
\end{equation}
and then
\begin{equation}\label{fix25}
||\n (v\,\d^{1-s})||_{L^p(\Omega_T)}\le \hat{C}_0||f||_{L^m(\O_T)}.
\end{equation}
Since $s>\dfrac 12$,  we can chose $T>0$ such that for some universal constant $\bar{C}$ depending only on $\O_T, N,s$, there exists $l>0$ such
that
\begin{equation}\label{llc}
\bar{C}(l+||f||_{L^m(\O_T)})=l^{\frac{1}{2s}}.
\end{equation}
Fix $T, l>0$ as above and define the set
\begin{equation}\label{sett}
E=\{v\in E_1(\O): v\, \d^{1-s}\in L^{m\a}(0,T; W^{1,m\a}_0(\O))\mbox{  and  } \bigg(\iint_{\O_T} |\n (v\, \d^{1-s})|^{m\a} dxdt\bigg)^{\frac{1}{m\a}}\le
l^{\frac{1}{\a}}\}.
\end{equation}
It is clear that $E$ is a closed convex set of $E_1(\O_T)$. Using Hardy inequality we reach that if $v\in E$, then
$$
\bigg(\iint_{\O_T} |\n v|^{m\a}\, \d^{m\a(1-s)}dx\bigg)^{\frac{1}{m\a}}\le \hat{C}_0 l^{\frac{1}{\a}}.
$$
Consider the operator
$$
\begin{array}{rcl}
\mathcal{T}:E &\rightarrow& E_1(\O_T)\\
   v&\rightarrow&T(v)=u
\end{array}
$$
where $u$ is the unique solution to problem
\begin{equation}
\left\{
\begin{array}{rcll}
u_t+(-\Delta )^s u &=&|\nabla v|^{\a}+ f &\inn \Omega_T,\\ u(x,t)&=&0 & \inn(\mathbb{R}^N\setminus\Omega)\times [0,T),\\ u(x,0)&=&u_{0}(x)
& \inn\Omega.\\
\end{array}
\right.  \label{fix1}
\end{equation}
{\bf First step:} We prove that $\mathcal{T}$ is well defined. By Theorem \ref{key}, to get the desired result, we just have to show the existence of $\b<2s-1$ such
that $|\n v|^{\a}\d^\b\in L^1(\O_T)$. Since $v\in E$, then $|\n v|^{\a}\in L^1_{loc}(\O)$. We have
\begin{eqnarray*}
& & \dyle \iint_{\O_T}|\n v|^{\a}\d^\b dxdt=\iint_{\O_T}|\n v|^{\a}\d^{\a(1-s)}\d^{\beta-\a(1-s)}dxdt\\
\\
& \le & \dyle \bigg(\iint_{\O_T}|\n v|^{m\a}\d^{(1-s)m\a}dxdt\bigg)^{\frac{1}{m}} \bigg(\iint_{\O_T}\d^{(\beta-\a(1-s))m'}dxdt\bigg)^{\frac{1}{m'}}.
\end{eqnarray*}
If $\a(1-s)<2s-1$, we can chose $\beta<2s-1$ such that $\a(1-s)<\beta$. Hence $\dint_\Omega (\d(x))^{(\beta-\a(1-s))m'}dx<\infty$.

Assume that $\a(1-s)\ge 2s-1$, then $s\in \Big(\dfrac{1}{2}, \dfrac{\a+1}{\a+2}\Big]$. Since $m>\dfrac{N+2s}{s-\a(1-s)}$ and $\a<\dfrac{s}{1-s}$, then $(\a(1-s)-(2s-1))m'<1$.
Hence we get easily the existence of $\beta<2s-1$ such that $(\a(1-s)-\beta)m'<1$. Thus we conclude.

Therefore, since $|\nabla v|^{\a}\d^\b+f \in L^1(\O_T)$, using Theorem \ref{key}, there exists $u$ that solves problem \eqref{fix1} with $u\in
E_a(\O_T)$ for all $a<\dfrac{N+2s}{N+\beta+1}$. Hence $\mathcal T$ is well defined.

{\bf Second step:}

We claim that:
\begin{enumerate}
\item For $l>0$ as above, $\mathcal{T}(E)\subset E$,
\item $\mathcal{T}$ is a continuous and  compact operator on $E$.
\end{enumerate}
\textsc{Proof of (1)}.

We have
\begin{eqnarray*}
u(x,t) & = & \dint_{0}^{t} \dint_{\Omega} |\n v(y,\s)|^{\a}P_{\Omega} (x,y, t-\sigma)\,dy\,d\sigma+ \dint_{0}^{t} \dint_{\Omega} f(y,\s)P_{\Omega} (x,y,
t-\sigma)\,dy\,d\sigma\\ &= & J_1(x,t)+J_2(x,t).
\end{eqnarray*}
Let us begin by estimating $J_1$. Notice that, for all $\theta\in [0,1]$,
\begin{equation}\label{gh1}
\Big( 1\wedge \frac{\delta^s(y)}{\sqrt{t-\s}}\Big)\le\d^{s\theta}(y)(t-\s)^{-\frac{\theta}{2}} \mbox{  in  }\O_T,
\end{equation}
and
\begin{equation}\label{gh2}
\Big( 1\wedge \frac{\delta^s(x)}{\sqrt{t-\s}}\Big)\le \d^{s(1-\theta)}(x)(t-\s)^{-\frac{(1-\theta)}{2}} \mbox{  in  }\O_T.
\end{equation}
Choosing $\theta=\dfrac{\a(1-s)}{s}<1$ and using the properties of $P_\O$, it holds that
$$
\begin{array}{rcl}
J_1(x,t)&\leq& C(\d(x))^{s(\a+1)-\a}\dint_{0}^{t}\dint_{\Omega}|\n v(y,\s)|^{\a} \d^{\a(1-s)}(y) \dfrac{(t-\s)^{\frac
12}}{((t-\s)^{\frac{1}{2s}}+|x-y|)^{N+2s}}\,dy\,d\sigma.
\end{array}
$$
Since $|\n v|^{\a} \d^{\a(1-s)}\in L^m(\O_T)$, then by Remark \ref{rm00} and since $m>\dfrac{N+2s}{s}$, then $\dfrac{J_1}{\d^{s-\a(1-s)}}\in L^\gamma(\O_T)$ for
all $\gamma<\infty$ and
$$
\Big|\Big| \dfrac{J_1}{\d^{s-\a(1-s)}}\Big|\Big|_{L^\gamma(\O_T)}\le C || |\n v|^{\a} \d^{\a(1-s)}||_{L^m(\O_T)}.
$$
As in the proof of Theorem \ref{hardy0} and since $f\in L^m(\O_T)$, we obtain that
$$
\Big|\Big| \dfrac{J_2}{\d^{s-\a(1-s)}}\Big|\Big|_{L^\gamma(\O_T)}\le C || f||_{L^m(\O_T)}.
$$
Hence we conclude that
\begin{equation}\label{tlm1}
\Big|\Big| \dfrac{u}{\d^{s-\a(1-s)}}\Big|\Big|_{L^\gamma(\O_T)}\le C \bigg(|| |\n v|^{\a} \d^{\a(1-s)}||_{L^m(\O_T)} +||f||_{L^m(\O_T)}\bigg).
\end{equation}

We deal now with the gradient term.

We have that $$
\begin{array}{rcl}
& |\nabla u(x,t)|\\ & \dyle \leq C \bigg\{  \dint_{0}^{t} \dint_{\Omega} |\n v(y,\sigma)|^\a |\nabla_x P_{\Omega} (x,y, t-\sigma)|\,dy\,d\sigma + \dint_{0}^{t}
\dint_{\Omega} f(y,\sigma) |\nabla_x P_{\Omega} (x,y, t-\sigma)|\,dy\,d\sigma\bigg\}\\ \\ & =I_1(x,t)+I_2(x,t).
\end{array}
$$
Notice that following the same computation as in the proof of Theorem \ref{regu-g} and using the fact that $m>\dfrac{N+2s}{s(2s-1)}$, we find that
\begin{equation}\label{II2}
\bigg\| I_2 \d^{1-s} \bigg\|_{L^\g(\O_T)}\le C ||f||_{L^m(\O_T)},
\end{equation}
for all $\g<\infty$.  Hence we have just to estimate $I_1$. Notice that
$$
I_1(x,t)\leq \dyle \dint_{0}^{t} \dint_{\Omega} |\n v(y,\sigma)|^\a \Big( \dfrac{1}{\delta(x) \wedge (t-\s)^{\frac{1}{2s}}}\Big)\, P_{\Omega}(x,y,t-\s)dyd\s.
$$
Thus
$$
\begin{array}{lll}
& \dyle I_1(x,t)\d^{1-s}(x)\leq \dyle \d^{1-s}(x)\iint_{\Omega_T} |\n v(y,\sigma)|^\a \Big( \dfrac{1}{\delta(x) \wedge (t-\s)^{\frac{1}{2s}}}\Big)\,
P_{\Omega}(x,y,t-\s)dyd\s\\ &\le \dyle C\bigg(\frac{1}{\d^s(x)}\iint_{\O_T\cap \{\d(x)<(t-\s)^{\frac{1}{2s}}\}}|\n v(y,\sigma)|^\a P_{\Omega}(x,y,t-\s)dyd\s \\ \\
&+ \dyle \d^{1-s}(x)\iint_{\O_T\cap \d(x)\ge (t-\s)^{\frac{1}{2s}}\}}|\n v(y,\sigma)|^\a\frac{ P_{\Omega}(x,y,t-\s)}{(t-\s)^{\frac{1}{2s}}}dyd\s\bigg)\\ \\ & =
L_1(x,t)+L_2(x,t).\\
\end{array}
$$
By the estimates \eqref{gh1} and \eqref{gh2}, we deduce that
$$
\begin{array}{lll}
L_1(x,t)\le \dyle \frac{1}{\d^{s(1-\theta)}(x)} \iint\limits_{\O_T\cap \{\d(x)<(t-\s)^{\frac{1}{2s}}\}}|\n
v(y,\sigma)|^{\a}\d^{\a(1-s)}(y)\dfrac{(t-\s)^{\frac{2s-\a(1-s)}{2s}-\frac \theta 2}}{((t-\s)^{\frac{1}{2s}}+|x-y|)^{N+2s}} dyd\s.
\end{array}
$$
Setting
$$
g(x,t):=\iint_{\O_T\cap \{\d(x)<(t-\s)^{\frac{1}{2s}}\}}|\n v(y,\sigma)|^{\a}\d^{\a(1-s)}(y)\dfrac{(t-\s)^{\frac{2s-\a(1-s)}{2s}-\frac \theta
2}}{((t-\s)^{\frac{1}{2s}}+|x-y|)^{N+2s}} dyd\s,
$$
then by Remark \ref{rm00} and since  $|\n v(y,\sigma)|^{\a}\d^{\a(1-s)}(y)\in L^{m}(\O_T)$, it follows that $g\in L^{\g_\theta}(\O_T)$ with $\g_\theta$ that satisfies
\begin{equation}\label{gama2}
\dfrac{1}{\g_\theta}>\dfrac{1}{m}-\dfrac{2s(\dfrac{2s-\a(1-s)}{2s}-\dfrac{\theta}{2})}{N+2s}=\frac{1}{m}-\frac{2s-\a(1-s)-s\theta}{N+2s}.
\end{equation}
It is clear that for all $\e>0$, we can chose $\theta$ closed to $1$ such that
$L_1(x,t)\in L^{\g_\theta-\e}(\O_T)$. Hence to show that $L_1(x,t)\in L^{m\a}(\O_T)$, we have
just to show that $\g_\theta>m\a$ for $\theta$ close to 1.
Since $m>\dfrac{N+2s}{s-\a(1-s)}$, then $\dfrac{1}{m\a}>\dfrac{1}{m}-\dfrac{2s-\a(1-s)-s}{N+2s}$. Hence,
for any $\theta<1$, we have $\dfrac{1}{m\a}>\dfrac{1}{m}-\dfrac{2s-\a(1-s)-\theta s}{N+2s}$ and then we conclude. As a consequence we deduce that $L_1\in
L^{m\a+\rho}(\O_T)$ for some $\rho>0$ and
\begin{equation}\label{l1}
||L_1||_{L^{m\a+\rho}(\O_T)}\le C(\O_T)  || |\n
v|^\a\d^{\a(1-s)}||_{L^m(\O_T)}.
\end{equation}

We deal now with $L_2(x,t)$.
\begin{eqnarray*}
L_2(x,t) &=& \dyle \d^{1-s}(x)\iint_{\O_T\cap \d(x)\ge (t-\s)^{\frac{1}{2s}}\}}|\n v(y,\sigma)|^\a\frac{ P_{\Omega}(x,y,t-\s)}{(t-\s)^{\frac{1}{2s}}}dyd\s\\ \\ &
= & \dyle \d^{1-s}(x)\iint_{\O_T\cap \d(x)\ge (t-\s)^{\frac{1}{2s}}\cap \{|x-y|\le \frac12 \d(x)\}} + \d^{1-s}(x)\iint_{\O_T\cap \d(x)\ge
(t-\s)^{\frac{1}{2s}}\cap \{\frac12 \d(x)<|x-y|\le \d(x)\}}\\ & + & \dyle  \d^{1-s}(x)\iint_{\O_T\cap \d(x)\ge (t-\s)^{\frac{1}{2s}}\cap \{\d(x)<|x-y|\}} \\  \\
&= & L_{21}(x,t) + L_{22}(x,t) + L_{23}(x,t).
\end{eqnarray*}
Let us begin by estimating $L_{21}$. Since $\d$ is a Lipschitz function, it holds that, if $(x,y)\in \{|x-y|<\frac 12 \d(x)\}$, then $|\d(y)-\d(x)|\le |x-y|\le \frac
12 \d(x).$ Thus $ \frac 12 \d(x)\le \d(y)\le \frac 32 \d(x). $ Hence we obtain that
\begin{eqnarray*}
L_{21}(x,t) &\le & C\iint\limits_{\O_T\cap \d(x)\ge (t-\s)^{\frac{1}{2s}}\cap |x-y|\le \frac12 \d(x)\}}|\n v(y,\sigma)|^\a\d^{(1-s)+\theta
s}(y)\dfrac{(t-\s)^{1-\frac{1}{2s}-\frac \theta 2}}{((t-\s)^{\frac{1}{2s}}+|x-y|)^{N+2s}}dyd\s.
\end{eqnarray*}
Choosing $\theta=\dfrac{(\a-1)(1-s)}{s}<1$, we get
\begin{eqnarray*}
L_{21}(x,t) &\le & C\iint\limits_{\O_T\cap \d(x)\ge (t-\s)^{\frac{1}{2s}}\cap |x-y|\le \frac12 \d(x)\}}|\n
v(y,\sigma)|^\a\d^{\a(1-s)}(y)\dfrac{(t-\s)^{\frac{s-\a(1-s)}{2s}}}{((t-\s)^{\frac{1}{2s}}+|x-y|)^{N+2s}}dyd\s.
\end{eqnarray*}
Using Remark \ref{rm00}, we conclude that $L_{21}\in L^\g(\O)$ for all $\g$ such that $\dfrac{1}{\g}>\dfrac 1m-\dfrac{s-\a(1-s)}{N+2s}$ and
\begin{equation}\label{l21}
||L_{21}||_{L^\g(\O_T)}\le C(\O_T) || |\n
v|^\a\d^{\a(1-s)}||_{L^m(\O_T)}.
\end{equation}
Since $m>\dfrac{N+2s}{s-\a(1-s)}$, then the above estimate holds for all $\g$ and then we conclude.

We analyze now the term $L_{22}$. We set
$$
A:=\O_T\cap \{\d(x)\ge (t-\s)^{\frac{1}{2s}}\}\cap \{\frac12 \d(x)<|x-y|\le \d(x)\},
$$
then as above choosing $\theta=\dfrac{\a(1-s)}{s}<1$, we obtain that
$$
L_{22}(x,t)\le \dyle \d^{1-s}(x)\iint\limits_{A} |\n
v(y,\sigma)|^{\a}\d^{\a(1-s)}(y)\dfrac{(t-\s)^{\frac{(2s-1)-\a(1-s)}{2s}}}{((t-\s)^{\frac{1}{2s}}+|x-y|)^{N+2s}} dyd\s.
$$
It is clear that $2s-1-\a(1-s)>0$ if and only if $s>\dfrac{\a+1}{\a+2}$.  Hence, in any case, using H\"older inequality, we get
\begin{eqnarray*}
& L_{22}(x,t)\le \\ & \dyle \d^{1-s}(x)\bigg(\iint\limits_{A} |\n v(y,\sigma)|^{m\a}(\d(y))^{m\a(1-s)}dyd\s\bigg)^{\frac{1}{m}}  \bigg(\iint\limits_{A}
\dfrac{(t-\s)^{\frac{m'}{2s}((2s-1)-\a(1-s))}}{((t-\s)^{\frac{1}{2s}}+|x-y|)^{m'(N+2s)}}dyd\s \bigg)^{\frac{1}{m'}}.
\end{eqnarray*}
By the definition of the set $A$, we obtain
\begin{eqnarray*}
& \dyle \iint\limits_{A} \dfrac{(t-\s)^{\frac{m'}{2s}((2s-1)-\a(1-s))}}{((t-\s)^{\frac{1}{2s}}+|x-y|)^{m'(N+2s)}}dyd\s\\ & \dyle =\int_{\{(t-\s)\le \d^{2s}(x)\}}
(t-\s)^{\frac{m'}{2s}((2s-1)-\a(1-s))}\int_{\{\frac12 \d(x)<|x-y|\le \d(x)\}}\dfrac{dy\, d\s}{((t-\s)^{\frac{1}{2s}}+|x-y|)^{m'(N+2s)}}\\ \\ &=\dyle
C(N)\int_{\{(t-\s)\le \d^{2s}(x)\}}(t-\s)^{\frac{m'}{2s}((2s-1)-\a(1-s))+\frac{N}{2s}-\frac{m'(N+2s)}{2s}}\int_{\frac12\frac{
\d(x)}{(t-\s)^{\frac{1}{2s}}}}^{\frac{ \d(x)}{(t-\s)^{\frac{1}{2s}}}}  \dfrac{r^{N-1}dr \, d\s}{(1+r)^{m'(N+2s)}} \\ \\ &\le \dyle C(N,s) \int_{\{(t-\s)\le
\d^{2s}(x)\}}(t-\s)^{\frac{m'}{2s}((2s-1)-\a(1-s))+\frac{N}{2s}-\frac{m'(N+2s)}{2s}} \times \frac{(t-\s)^{\frac{1}{2s}(m'(N+2s)-N)}}{\d^{m'(N+2s)-N}(x)} d\s\\ \\
&\dyle \le \frac{C(N,s)}{(\d(x))^{m'(N+2s)-N}} \int_{\{(t-\s)\le
\d^{2s}(x)\}}(t-\s)^{\frac{m'}{2s}((2s-1)-\a(1-s))+\frac{N}{2s}-\frac{m'(N+2s)}{2s}+\frac{1}{2s}(m'(N+2s)-N)}d\s\\ \\ &\dyle \le
\frac{C(N,s)}{(\d(x))^{m'(N+2s)-N}} \int_0^{\d^{2s}(x)}\rho^{\frac{m'}{2s}((2s-1)-\a(1-s))}d\rho.
\end{eqnarray*}
If $s\ge \dfrac{\a+1}{\a+2}$, then $\dfrac{m'}{2s}((2s-1)-\a(1-s))\ge 0$. If $\dfrac 12<s<\dfrac{\a+1}{\a+2}$, since $m'<\dfrac{N+2s}{N+s+\a(1-s)}$, then
$\dfrac{m'}{2s}((2s-1)-\a(1-s))>-1$. Thus
$$
\dyle \frac{C(N,s)}{(\d(x))^{m'(N+2s)-N}} \int_0^{\d^{2s}(x)}\rho^{\frac{m'}{2s}((2s-1)-\a(1-s))}d\rho\le (\d(x))^{N+2s-m'(N+1+\a(1-s))}.
$$
Going back to the estimate of $L_{22}$ it holds that
$$
L_{22}(x,t)\le C (\d(x))^{s-\a(1-s)-\frac{N+2s}{m}} \bigg(\iint\limits_{A} |\n v(y,\sigma)|^{m\a}(\d(y))^{m\a(1-s)}dyd\s\bigg)^{\frac{1}{m}}.
$$
Since $m>\dfrac{N+2s}{s-\a(1-s)}$, we reach that $s-\a(1-s)-\dfrac{N+2s}{m}>0$. Thus $L_{22}\in L^\infty(\O_T)$ and for all $\g>1$,
\begin{equation}\label{l22}
||L_{22}||_{L^\g(\O_T)}\le C(\O_T)\bigg(\iint\limits_{\O_T} |\n v(y,\sigma)|^{2sm}\d^{2sm(1-s)}(y)dyd\s\bigg)^{\frac{1}{m}}.
\end{equation}
Respect to $L_{23}$, setting
$$
\bar{A}:=\O_T\cap \{\d(x)\ge (t-\s)^{\frac{1}{2s}}\}\cap \{|x-y|\ge \d(x)\}\}
$$
and following the same computations as above, it holds that
\begin{eqnarray*}
& L_{23}(x,t)\le\\ & \dyle \d^{1-s}(x)\bigg(\iint\limits_{\bar{A}} |\n v(y,\sigma)|^{m\a}\d^{\a m(1-s)}(y)dyd\s\bigg)^{\frac{1}{m}}  \bigg(\iint\limits_{\bar{A}}
\dfrac{(t-\s)^{\frac{m'}{2s}((2s-1)-\a(1-s))}}{((t-\s)^{\frac{1}{2s}}+|x-y|)^{m'(N+2s)}}dyd\s \bigg)^{\frac{1}{m'}}\\ \\ &\le \dyle
\d^{1-s}(x)\bigg(\iint\limits_{\bar{A}} |\n v(y,\sigma)|^{m\a}\d^{m\a(1-s)}(y)dyd\s\bigg)^{\frac{1}{m}}\\  \\ &\times \dyle \bigg(  \int_{\{(t-\s)\le
\d^{2s}(x)\}}(t-\s)^{\frac{m'}{2s}((2s-1)-\a(1-s))+\frac{N}{2s}-\frac{m'(N+2s)}{2s}} \int_{\frac{ \d(x)}{(t-\s)^{\frac{1}{2s}}}}^{\infty}
\dfrac{r^{N-1}}{(1+r)^{m'(N+2s)}}dr \, d\s\bigg)^{\frac{1}{m'}}.
\end{eqnarray*}
Thus, as above, we conclude that
$$
L_{23}(x,t)\le C (\d(x))^{s-\a(1-s)-\frac{N+2s}{m}}\bigg(\iint\limits_{\bar{A}} |\n v(y,\sigma)|^{m\a}\d^{m\a(1-s)}(y)dyd\s\bigg)^{\frac{1}{m}}.
$$
Then as in the estimate of $L_{22}$, we reach that $L_{23}\in L^\infty(\O_T)$ and for all $\g>1$,
\begin{equation}\label{l23}
||L_{23}||_{L^\g(\O_T)}\le C\bigg(\iint\limits_{\O_T} |\n v(y,\sigma)|^{2sm}\d^{2sm(1-s)}(y)dyd\s\bigg)^{\frac{1}{m}}.
\end{equation}

As a conclusion and combing estimates \eqref{l21}, \eqref{l22} and \eqref{l23} we reach that if $u$ solves \eqref{fix1} with $v\in E$, then
\begin{eqnarray*}
& \dyle \bigg(\iint_{\O_T} |\n u|^{\g}\, \d^{\g(1-s)}) dxdt\bigg)^{\frac{1}{\g}}\\ & \le \dyle C(\O_T,N,s)\bigg(\bigg(\iint\limits_{\O_T} |\n
v(y,\sigma)|^{m\a}\d^{m\a(1-s)}(y)dyd\s\bigg)^{\frac{1}{m}} +||f||_{L^m(\O_T)}\bigg)
\end{eqnarray*}
and
$$
\Big|\Big| \dfrac{u}{\d^{s-\a(1-s)}}\Big|\Big|_{L^\gamma(\O_T)}\le C \bigg(|| |\n v|^{\a} \d^{\a(1-s)}||_{L^m(\O_T)} +||f||_{L^m(\O_T)}\bigg),
$$
for some $\g>m\a$. Recall that $T$ is fixed such that \eqref{llc} holds. Thus $u\in E$ and then $\mathcal{T}(E)\subset E$.

\

Now we proof (2). Let begin by proving the continuity of $\mathcal{T}$ respect to the topology of $L^{1}(0,T; W^{1,1}_0(\O))$. Consider $\{v_n\}_n\subset E$ such that
$v_n\to v$ strongly in $L^{1}(0,T; W^{1,1}_0(\O))$. Define $u_n=\mathcal{T}(v_n), u=\mathcal{T}(v)$.

We have to show that $u_n\to u$ strongly in $L^{1}(0,T; W^{1,1}_0(\O))$. Using Theorem \ref{key}, to show the desired result, we have just to prove that
$$||\n v_n -\n v||_{_{L^{\a}(\d^\beta dx, \O_T)}}\to 0 \hbox{   as } n\to \infty,$$ for some $\beta<2s-1$. As in the proof of the fist step, we get the existence
of $\beta<2s-1$ such that
$\dint_{\O_T}|\n v_n|^{\a}\d^\b dx\le C$ for all $n$.

Since $m\a>1$, then setting $a=\dfrac{\a(m-1)}{m\a-1}<1$, it follows that $\dfrac{\a-a}{1-a}=m\a$. Hence by Hölder inequality, we conclude that
\begin{eqnarray*}
||\n v_n -\n v||_{_{L^{\a}(\d^\beta dx, \O_T)}} & \le & ||\n v_n-\n v||^{\frac{a}{\a}}_{L^{1}(\d^{\beta} dx, \O_T)} ||\n v_n-\n
v||^{\frac{\a-a}{\a}}_{L^{\frac{\a-a}{1-a}}(\d^{1-s} dx, \O_T)}\\ &\le & C||\n v_n-\n v||^{\frac{a}{\a}}_{L^{1}(\O_T)}\to 0\mbox{  as  }n\to \infty.
\end{eqnarray*}
Now, using the definition of $u_n$ and $u$, there results that $u_n\to u$ strongly in $L^{\s}(0,T; W^{1,\s}_0(\O))$ for some $\s>1$ and then we conclude.

We prove now that $\mathcal{T}$ is compact. Let $\{v_n\}_n\subset E$ be such that $||v_n||_{L^{1}(0,T; W^{1,1}_0(\O))}\le C$.

Since $\{v_n\}_n\subset E$, then $||\n (v_n \d^{1-s})||_{L^{m\a}(\O_T)}\le C$. Therefore,  up to a subsequence,
$$v_{n_k}\d^{1-s}\rightharpoonup v\d^{1-s} \hbox{ weakly in  } L^{m\a}(0,T; W^{1,m\a}_0(\O)).$$
It is clear that $v_{n_k}\rightharpoonup v$ weakly in $L^{m\a}(0,T; W^{1,m\a}_{loc}(\O))$.

Let
$$
F_n=|\nabla v_n|^{\a}+f, F=|\nabla v|^{\a}+f,
$$
then, as in the first step, we can prove that $F_n \d^\beta$ is bounded in $L^{1+a}(\O)$ for some $a>0$ and then $F_n \d^\beta\rightharpoonup F \d^\beta$ weakly in
$L^{1+a}(\O)$. Using the compactness result in Theorem \ref{key}, we conclude that, up to a subsequence, $u_{n_k}\to u$ strongly in $L^{q}(0,T; W^{1,q}_0(\O))$
for all $q<\dfrac{N+2s}{N+\beta+1}$ and then the result follows.

\noindent Hence we are in position to use the Schauder Fixed Point Theorem and then there exists $u\in E$ such that $\mathcal{T}(u)=u$. Thus $u\d^{1-s}\in L^{m\a}(0,T;
W^{1,m\a}_0(\O))$ and $u$ solves \eqref{grad} in the sense of distribution. \cqd

\begin{remarks}
The above arguments can be used to treat the problem
\begin{equation}\label{grada}
\left\{
\begin{array}{rcll}
u_t+(-\Delta)^s u +|\n u|^\a &= & f & \text{ in }\Omega_T ,
\\ u(x,t) &=& 0 &\hbox{  in } \mathbb{R}^N\setminus\Omega\times (0,T),\\
u(x,0)&=&0 &\hbox{ in }\Omega,
\end{array}%
\right.
\end{equation}
where $\a<\dfrac{s}{1-s}$ and $f\in L^m(\O_T)$ with $m>\max\{\dfrac{N+2s}{s(2s-1)}, \dfrac{N+2s}{s-\a(1-s)}\}$.
Then for $T$ small, problem \eqref{grada} has a distributional solution $u\in
L^{\a}(0,T;W^{1,\a}_{loc}(\O))\cap L^1(0,T;W^{1,1}_0(\O))$ with $u\d^{1-s}\in L^{m\a}(0,T;W^{1,m\a}_0(\O))$.
\end{remarks}

In the case where $\dfrac{s}{1-s}\le \a$, then we consider the modified problem
$$
\left\{
\begin{array}{rcll}
u_t+(-\Delta)^s u &= & \d^{\a(1-s)-s}(x)|\nabla u|^{\a}+f & \text{ in }\Omega_T ,
\\ u(x,t) &=& 0 &\hbox{  in } \mathbb{R}^N\setminus\Omega\times (0,T),\\
u(x,0)&=&0 &\hbox{ in }\Omega,
\end{array}%
\right.
$$
where $m>\max\{\dfrac{N+2s}{\a'(2s-1)}, \dfrac{N+2s}{\a(1-s)-s}\}$, then as in the proof of the previous Theorem, existence of solution holds using the Schauder fixed point Theorem in the set
$$
\hat{E}(\O_T)=\{v\in E_1(\O_T): v\, \d^{1-s}\in E_{m}(\O_T) \mbox{  with  } \|v \d^{1-s}\|_{E_{m}(\O_T)}\leq l^{\frac{1}{\alpha}}\}.
$$

\

\

If $f=0$ and $u_0\neq 0$, we have the result in Theorem \ref{hhh2} whose proof follows.

{\bf Proof of Theorem \ref{hhh2}. }

 Recall that  $\dfrac{2s-1}{1-s}>\dfrac{(N+2s)^2}{N}$ and that $\frac{N+2s}{N+1}\le \a<\frac{2s}{(1-s)(N+2s)+1}$.
Define $\psi$ to be the unique solution to problem
\begin{equation}\label{inter1}
\left\{
\begin{array}{rcll}
\psi_t+(-\Delta )^s \psi&=& 0 &\inn \Omega_T,\\ \psi(x,t)&=&0 &\inn(\mathbb{R}^N\setminus\Omega)\times [0,T),\\ \psi(x,0)&=& u_0(x) & \inn\Omega.
\end{array}\right.
\end{equation}
Recall that $u_0\in L^\s(\O)$ where $\s>\dfrac{(\a-1)N}{(2s-\a)-\a(1-s)(N+2s)}$, then by the regularity results in Proposition \ref{pro:lp2} and Corollary \ref{cor1100},  we obtain that $|\n \psi|\in L^\theta(\O_T)$ for all $\theta<m_1=\dfrac{\s(N+2s)}{(1-s)\s(N+2s)+N+\s}$.

Notice that if $v$ solves the problem
\begin{equation}\label{gradss-n}
\left\{
\begin{array}{rcll}
v_t+(-\Delta )^s v&=&|\nabla (v+\psi)|^{\alpha} &\inn \Omega_T,\\ v(x,t)&=&0 &\inn(\mathbb{R}^N\setminus\Omega)\times [0,T),\\ v(x,0)&=& 0 & \inn\Omega,
\end{array}\right.
\end{equation}
then $u\equiv v+\psi$ is a solution to \eqref{gradu0}. Hence we have just to prove the existence of $v$.  Notice that
$$
|\nabla (v+\psi)|^{\alpha}\le C_1|\n v|^\a+C_2|\n \psi|^\a.
$$
Define $f\equiv C_2|\n \psi|^\a$, then $f\in L^m(\O_T)$ for any $m<\dfrac{m_1}{\a}=\dfrac{\s(N+2s)}{\a[(1-s)\s(N+2s)+N+\s]}$. It is clear that $m<\dfrac{\s(N+2s)}{\a \s}<\dfrac{N+2s}{2s-1}$.

Hence, using the fact that $\frac{N+2s}{N+1}\le \a<\frac{2s}{(1-s)(N+2s)+1}$ and $\s>\frac{(\a-1)N}{(2s-\a)-\a(1-s)(N+2s)}$, then we get the existence of $\frac{1}{s}\le m<\frac{N+2s}{2s-1}$ such that $m\a<\dfrac{m(N+2s)}{(N+2s)(m(1-s)+1)-m(2s-1)}=\check{P}$ defined in Corollary \ref{cor11}.

Hence we can fix $r>1$ such that $m\a<r<\dfrac{m(N+2s)}{(N+2s)(m(1-s)+1)-m(2s-1)}$. Now following the same argument as in the proof of Theorem \ref{hhh} we get the desired existence result. \cqd
\begin{remarks}

\

\begin{enumerate}
\item { Let us consider now the case where  $f\gneqq 0$ and $u_0 \gneqq 0$ simultaneously . As in the proof of Theorem \ref{hhh2}, define $\psi$ to be the unique solution to problem \eqref{inter1} and let $\vartheta$ the solution to the problem
\begin{equation}\label{gradss-nnn}
\left\{
\begin{array}{rcll}
\vartheta_t+(-\Delta )^s \vartheta &=&|\nabla (\vartheta+\psi)|^{\alpha}+f &\inn \Omega_T,\\ \vartheta(x,t)&=&0 &\inn(\mathbb{R}^N\setminus\Omega)\times [0,T),\\ \vartheta(x,0)&=& 0 & \inn\Omega,
\end{array}\right.
\end{equation}
then $u\equiv \vartheta+\psi$ solves the problem
\begin{equation*}
\left\{
\begin{array}{rcll}
u_t+(-\Delta )^s u &=&|\nabla u|^{\alpha}+f &\inn \Omega_T,\\ u(x,t)&=&0 &\inn(\mathbb{R}^N\setminus\Omega)\times [0,T),\\ u(x,0)&=& 0 & \inn\Omega,
\end{array}\right.
\end{equation*}
Taking into consideration that
$$
|\nabla (\vartheta+\psi)|^{\alpha}\le C_1|\n \vartheta|^\a+C_2|\n \psi|^\a,
$$
we can reproduce the same approach as in the proof of Theorem \ref{hhh2}, with $\tilde{f}=C_2|\n \psi|^\a,+ f$, to get the existence of a solution $\vartheta$ to problem \ref{gradss-nnn} combining the regularity of $f$ and $u_0$. Hence we conclude.
}
\

\

\item In a forthcoming paper we will treat the problem
\begin{equation}\label{gradss01}
\left\{
\begin{array}{rcll}
u_t+(-\Delta )^s u&=&|\nabla u|^{\alpha}&\inn \Omega_T,\\ u(x,t)&=&0 &\inn(\mathbb{R}^N\setminus\Omega)\times [0,T),\\ u(x,0)&=& u_0(x) & \inn\Omega.
\end{array}\right.
\end{equation}
under the general condition $\a<\dfrac{s}{1-s}$. Existence of solutions will proved in a suitable parabolic weighted Sobolev space. Global existence in time or blow-up in finite time will be also analyzed.

\end{enumerate}
\end{remarks}

\section{Comparison principle and a partial uniqueness result for a problem
with a drift term. Applications to the quasi-linear problem \eqref{grad}}

 We will study in this Section an equation with a \textit{ drift}, that is, we substitute the nonlinear term in the gradient  by a term of the form $B(x,t)\,\cdot\, \n u$.
Then the existence of a solution is obtained under natural conditions of the field $B$ and the data $f, u_0$.  Using the previous  arguments, we can prove the existence for the largest class of the data $f, u_0$, under the natural condition on $B$.

Let us begin by considering the next problem
\begin{equation}\label{grad1}
\left\{
\begin{array}{rcll}
u_t+(-\Delta )^s u&=& B(x,t)\,\cdot\,  \n u+ f &\inn \Omega_T,\\ u(x,t)&=&0 &\inn(\mathbb{R}^N\setminus\Omega)\times [0,T),\\ u(x,0)&=& u_0(x)
&\inn\Omega,\\
\end{array}\right.
\end{equation}
 where $B\in (L^{m}(\O_T))^N$ with $m>\dfrac{N+2s}{2s-1}$. We are able to prove the next existence result that extends the one obtained in \cite{W1}.
 \begin{Theorem}\label{th2}
Assume that $\O\subset \ren$ is a bounded regular domain and $s\in (\dfrac 12, 1)$. Suppose that $B\in (L^{m}(\O_T))^N$ with $m>\dfrac{N+2s}{2s-1}$, then for all $(f,u_0)\in
L^{1}(\Omega_T)\times L^{1}(\Omega)$, the problem \eqref{grad1} has a weak solution $u\in L^{q}(0,T;W_{0}^{1,q}(\Omega)), \, q<\dfrac{N+2s}{N+1}$ and $T_k(u)\in
L^{2}(0,T;H^s_0(\Omega))$ for all $k>0$.
 Furthermore, If $u_0\ge 0$ in $\O$ and $f\ge 0$ in $\O_T$, then $u\ge 0$ in $\O_T$.

Moreover, if $B$ satisfies one of the following assumption:
\begin{enumerate}
\item $B$ does not depend on $t$,
\item $B\in (L^{m}(\O_T))^N$ with $m>\max\{q', \dfrac{2s}{(N+2s)-q(N+1)}\}>\dfrac{N+2s}{2s-1}$ for some $1<q<\dfrac{N+2s}{N+1}$,
\end{enumerate}
then the solution obtained above is unique.
\end{Theorem}
\begin{proof}
As in the proof of Theorem \ref{th1}, we define the operator
$$
\begin{array}{rcl}
K:\widetilde{E}(\O_T)&\rightarrow&  E_1(\O_T)\\
   v&\rightarrow&T(v)=u,
\end{array}
$$
where
$$
\widetilde{E}(\O_T)=\{v\in E_1(\O_T)\mbox{  such that  }v\in E_r(\O_T) \mbox{  with  }\s'<r<\frac{N+2s}{N+1} \mbox{  and  }\|v\|_{E_r(\O_T)}\leq
l^{\frac{1}{\alpha}}\},
$$
and $u$ is the unique solution to problem
\begin{equation}\label{grad00}
\left\{
\begin{array}{rcl}
u_t+(-\Delta )^s u&=& B(x,t)\,\cdot\, \n v+ f\inn \Omega_T\equiv\Omega\times (0,T),\\ u(x,t)&=&0\inn(\mathbb{R}^N\setminus\Omega)\times [0,T),\\ u(x,0)&=&
u_0(x)\inn\Omega.\\
\end{array}\right.
\end{equation}
Let us begin by proving that $K$ is well defined.

Since $B\in (L^{m}(\O_T))^N$ with $m>\dfrac{N+2s}{2s-1}$, then $m'<\dfrac{N+2s}{N+1}$. Fix $r$ such that $m'<r<\dfrac{N+2s}{N+1}$, then for $v\in
\widetilde{E}(\O_T)$, using H\"older inequality, we conclude that
$$
\dint_{0}^{T}\dint_{\Omega} |B(x,t)\,\cdot\, \n v|dxdt\le C(T,\O)\Big(\dint_{0}^{T}\dint_{\Omega}|B(x,t)|^m
dxdt\Big)^{\frac{1}{m}}\Big(\dint_{0}^{T}\dint_{\Omega}|\n v|^r dxdt\Big)^{\frac{1}{r}}.
$$
Hence $|B(x,t)\,\cdot\, \n v|\in L^1(\O_T)$ and then $K$ is well defined. Now, the existence result follows using the same compactness arguments as in the proof
of Theorem \ref{th1}.

Assume now that $f\ge 0$. To get the existence of a nonnegative solution we consider the next variation of the operator $K$. Namely we define $u$ to be the unique
solution to the problem
\begin{equation}\label{grad001}
\left\{
\begin{array}{rcll}
u_t+(-\Delta )^s u&=& B(x,t)\,\cdot\, \n v_+ + f &\inn \Omega_T,\\ u(x,t)&=&0 &\inn(\mathbb{R}^N\setminus\Omega)\times [0,T),\\ u(x,0)&=& u_0(x)
&\inn\Omega.\\
\end{array}\right.
\end{equation}
 It is clear that, with the new definition, if $u$ is a fixed point of $K$, then $u$ solves
$$
\left\{
\begin{array}{rcll}
u_t+(-\Delta )^s u&=& B(x,t)\,\cdot\, \n u_+ + f &\inn \Omega_T,\\ u(x,t)&=&0 & \inn(\mathbb{R}^N\setminus\Omega)\times [0,T),\\ u(x,0)&=& u_0(x) &
\inn\Omega.\\
\end{array}\right.
$$
Using $T_k(u_-)$ as a test function in the previous equation we reach that $T_k(u_-)=0$ for all $k$. Hence $u_-=0$ and then $u\ge 0$ in $\O_T$.
 It is clear that in a symmetric way, if $f\le 0$ in $\O_T$ and $u_0\le 0$ in $\O$, then we get the existence of a $u\le 0$ in
$\O_T$.

We prove now that the solution is unique. Assume that $u_1, u_2$ are solutions to problem \eqref{grad1}, setting $v=u_1-u_2$, then $v$ solves
\begin{equation}\label{grad001-un}
\left\{
\begin{array}{rcll}
v_t+(-\Delta )^s v&=& B(x,t)\,\cdot\, \n v &\inn \Omega_T,\\ v(x,t)&=&0 &\inn(\mathbb{R}^N\setminus\Omega)\times (0,T),\\ v(x,0)&=& 0 &\inn\Omega.\\
\end{array}\right.
\end{equation}
Notice that $v \in L^{a}(0,T;W_{0}^{1,a}(\Omega))$ for all $a<\dfrac{N+2s}{N+1}$  and $T_k(v)\in L^{2}(0,T;H^s_0(\Omega))$ for all $k>0$.

Define $g(x,t)=B(x,t)\,\cdot\, \n v(x,t)$, since $B\in (L^m(\O_T))^N$ with $m>\dfrac{N+2s}{2s-1}$, then we can fix $1<q<\dfrac{N+2s}{N+1}$ such that $q'<m$.
Notice that
$$
v(x,t)=\dint_{0}^{t} \dint_{\Omega} g(y,\sigma) P_{\Omega} (x,y, t-\sigma)\,dy\,d\sigma,
$$
and
$$
|\nabla v(x,t)|\leq \dyle \dint_{0}^{t} \dint_{\Omega}| g(y,\sigma)| |\nabla_x P_{\Omega} (x,y, t-\sigma)|\,dy\,d\sigma.
$$
From Proposition \ref{more-tregu} we deduce that for all $\eta>0$,
$$
\Big(\dyle \io |\n v(x,t)|^qdx\Big)^{\frac{1}{q}}\le C(\O_T) \bigg(\int_0^t\io |g(y,\sigma)|(t-\s)^{\hat{\g}-\eta}dy\,d\sigma\bigg)^{\frac{1}{q}}.
$$
Recall that $|g(x,t)|\le |B(x,t)| |\n v(x,t)|$, thus
$$
\io |g(y,\s)| dy \le \bigg(\io |\n v(y,\s)|^q dy \bigg)^{\frac{1}{q}}  \bigg(\io |B(y,\s)|^{q'} dy \bigg)^{\frac{1}{q'}}.
$$
Let $\a=1+\hat{\g}-\eta\in (0,1)$, then
\begin{eqnarray*}
\Big(\dyle \io |\n v(x,t)|^qdx\Big)^{\frac{1}{q}} & \le & C(\O_T) \int_0^t(t-\s)^{\a-1}\bigg(\io |\n v(y,\s)|^q dy \bigg)^{\frac{1}{q}}  \bigg(\io |B(y,\s)|^{q'} dy \bigg)^{\frac{1}{q'}}\,d\sigma. \\
\end{eqnarray*}
Setting $Y(t)=\bigg(\dyle \io |\n v(x,t)|^qdx\bigg)^{\frac{1}{q}}$, then $Y\in L^q(0,T)$ and
$$
Y(t)\le C(\O_T)\int_0^t (t-\s)^{\a-1}K(\s) Y(\s) d\sigma,
$$
with $K(\s)=\bigg(\dint_\Omega |B(y,\s)|^{q'} dy \bigg)^{\frac{1}{q'}}$.

Assume that $B$ depends only on $x$, then $K\in L^\infty(0,T)$. From the singular Bellman-Gronwall inequality proved in \cite{Hen}, Lemma 7.7.1, we deduce that $Y=0$ in $L^1(0,T)$. Thus
$$\dyle\int_0^T\bigg(\dyle \io |\n v(x,t)|^qdx\bigg)^{\frac{1}{q}}dt=0.$$ Since $v\in L^q((0,T);W^{1,q}_0(\O))$, we obtain that $v=0$ and the result follows.

Now, in second hypothesis, if $B\in (L^{m}(\O_T))^N$ with $m>\max\{q', \dfrac{2s}{(N+2s)-q(N+1)}\}>\dfrac{N+2s}{2s-1}$, then for $\eta$ small enough, $m\a>1$. We set $\hat{K}(t,\s)=(t-\s)^{\a-1}K(\s)$, then
 there exists $\theta\in \Big(1, \dfrac{1}{1-\a}\Big)$ such that
$$
\int_0^T \bigg(\int_0^t (\hat{K}(t,\s))^\theta d\s\bigg)^{\frac{1}{\theta-1}} dt<\infty.
$$
Since
$$
Y(t)\le C(\O_T)\int_0^t \hat{K}(\t,\s)Y(\s) d\sigma,
$$
then from \cite{Kwa}, we obtain that $Y=0$ in $L^1(0,T)$. Hence as above we conclude that $v=0$.
\end{proof}
\begin{remarks}
Under the additional hypothesis on $B$ that ensures the uniqueness of the solution to problem \eqref{grad1}, we can also see $u$ as mild solution to problem \eqref{grad1} in the sense that, if we denote by $\hat{P}_{\Omega}$, the Dirichlet heat kernel associated to the operator
$$
L(v):=\p_t v+(-\Delta )^s v -B(x,t)\,\cdot\, \n v,
$$
then
$$
u(x,t)=\dint_{\Omega}u_0(y) \hat{P}_{\Omega} (x,y, t)\,dy\,+ \dint_{0}^{t} \dint_{\Omega} f(y,\sigma) \hat{P}_{\Omega} (x,y, t-\sigma)\,dy\,d\sigma.
$$
Notice that from \cite{JS1} (see Example 3, page 335), we get that $B\in \mathcal{K}(\eta,Q)$, defined in
\cite{JS1}. Then $\hat{P}_{\Omega}\backsimeq P_{\Omega}$.
\end{remarks}

\begin{remarks}
Related to the linear problem \eqref{grad001-un}, we can show that $v\in \mathcal{C}^{\tau,2s}_{t,x}(\Omega\times (0,T))$ for some $\tau\in
(0,1)$. Effectively, let us begin by proving that $g\in L^{m}_{loc}(\Omega_T)$ (where $L^{m}_{loc}(\Omega_T)$ is defined in Remark \ref{rm001}).
Recall that $m>\dfrac{N+2s}{2s-1}$, then, using H\"older inequality it holds that $g\in L^{l_1}(\Omega_T)$ for $1<l_1<\dfrac{\hat{q}\sigma}{\sigma+\hat{q}}$ where $\hat{q}=\frac{N+2s}{N+1}$ . Fix $l_1$
as above, by Proposition \ref{key2-locc}, we obtain that $|\n v|\in L^{r_1}_{loc}(\Omega_T)$ with $r_1=\dfrac{l_1(N+2s)}{N+2s-l_1(2s-1)}$. Using again H\"{o}lder
inequality, we reach that $g\in L^{l_2}_{loc}(\Omega_T)$ with $l_2=\dfrac{r_1\sigma}{\sigma+r_1}$. Using again Proposition \ref{key2-locc}, it follows that $|\n
w|\in L^{r_2}_{loc}(\Omega_T)$ with $r_2=\dfrac{l_2(N+2s)}{N+2s-l_2(2s-1)}$.

Hence, we define the two sequences $\{l_n\}_n$ and $\{r_n\}_n$ by
$$
\left\{
\begin{array}{rcll}
1< & l_1 & <\hat{q}=\dfrac{N+2s}{N+1},\\ \\ r_i &=& \dfrac{l_i (N+2s)}{N+2s-l_i(2s-1)}, i\ge 1,\\ \\ l_{i+1} & = & \dfrac{r_i \s}{\s+r_i}, i\ge 1, i\ge 1.
\end{array}\right.
$$
Thus
$$
r_{i+1}=\frac{m(N+2s)r_i}{m(N+2s)-r_i((2s-1)m-(N+2s))}.
$$
It is clear that $r_{i+1}>r_i$. Let show that there exists $i_0$ such that $r_{i_0}\ge \dfrac{\sigma (N+2s)}{(2s-1)m-(N+2s)}$. We argue by contradiction.
Assume that $r_i<\dfrac{m (N+2s)}{(2s-1)m-(N+2s)}$ for all $i$. Since $\{r_i\}_i$ is an increasing sequence, there exists a $\bar{r}$ such that $r_i\uparrow \bar{r}\le \dfrac{m (N+2s)}{(2s-1)m-(N+2s)}$.

Thus $\bar{r}=\dfrac{m(N+2s)\bar{r}}{m(N+2s)-\bar{r}((2s-1)m-(N+2s))}$, hence $\bar{r}=0$,
a contradiction with the fact that $\{r_i\}_i$ is an increasing sequence.\\
Therefore, there exists $i_0\in \ene$ such that $r_{i_0}\ge \dfrac{m (N+2s)}{(2s-1)m-(N-2s)}$.Thus, using H\"older inequality, we conclude that $g\in L_{loc}^{\frac{N+2s}{2s-1}}(\Omega_T)\cap L^1(\O_T)$. Hence by the result of Proposition  \ref{key2-locc},
we obtain that $|\n v|\in L^{a}_{loc}(\Omega_T)$ for all $a>1$. Thus $g\in L^{m}_{loc}(\Omega_T)$ and the claim follows.

\noindent Therefore, by the regularity results in \cite{CF}, \cite{Gr} and \cite{JS1}, we obtain that then $v\in \mathcal{C}^{\tau,2s}_{t,x}(\Omega\times (0,T{]})$ for some $\tau\in (0,1)$.
\end{remarks}

Under the hypotheses
\begin{equation}\label{covid}f\d^\beta\in L^1(\O),\,  B\d^\beta \in (L^{\s}(\O_T))^N \hbox{ with  } \s>\frac{N+2s}{2s-\beta-1} \hbox{ for some } 0<\beta<2s-1,
\end{equation}
 then as
in Theorem \ref{th2}, we have the next existence result.
 \begin{Theorem}\label{th2beta}
Assume that the hypotheses \eqref{covid} on $f$ and $B$ hold and $u_0\in L^1(\O)$. Then the problem \eqref{grad1} has a distributional solution $u\in
L^{a}(0,T;W_{0}^{1,a}(\Omega)), \, a<\dfrac{N+2s}{N+\beta+1}$. Moreover, if in addition $B$ satisfies one of the following conditions:
\begin{enumerate}
\item $B$ does not depends on $t$,
\item $B\in (L^{m}(\O_T))^N$ with $m>\max\{q', \dfrac{2s}{(N+2s)-q(N+\beta+1)}\}>\dfrac{N+2s}{2s-\beta-1}$ for some $1<q<\dfrac{N+2s}{N+\beta+1}$,
\end{enumerate}
then the solution is unique.
\end{Theorem}
\begin{proof}
As in the proof of Theorem \ref{th2}, consider the set
$$
E_\beta(\O_T)=\{v\in E_1(\O_T)\mbox{  such that  }v\in E_r(\O_T) \mbox{  with  }\s'<r<\frac{N+2s}{N+\beta+1} \mbox{  and  }\|v\|_{E_r(\O_T)}\leq
l^{\frac{1}{\alpha}}\}.
$$
We define the operator
$$
\begin{array}{rcl}
K:E_\beta(\O_T)&\rightarrow&  E_1(\O_T)\\
   v&\rightarrow&T(v)=u
\end{array}
$$
where $u$ is the unique solution to problem
\begin{equation}\label{grad00beta}
\left\{
\begin{array}{rcl}
u_t+(-\Delta )^s u&=& B(x,t)\,\cdot\,  \n v+ f\inn \Omega_T,\\ u(x,t)&=&0\inn(\mathbb{R}^N\setminus\Omega)\times [0,T),\\ u(x,0)&=&
u_0(x)\inn\Omega.\\
\end{array}\right.
\end{equation}
Since, for $v\in E_\beta(\O_T)$,
$$
\iint_{\O_T}|B(x,t)|  |\n v| \d^\beta dxdt \le C(T,\O)\Big(\dint_{0}^{T}\dint_{\Omega}|B(x,t)|^\s
\d^{\beta\s}dxdt\Big)^{\frac{1}{\s}}\Big(\dint_{0}^{T}\dint_{\Omega}|\n v|^r dxdt\Big)^{\frac{1}{r}},
$$
we obtain that $|\langle B(x,t), \n v\rangle+ f|\d^\beta\in L^1(\O_T)$. Thus using Theorem \ref{key}, we get the existence of a unique $u\in E_a(\O_T)$ for all
$a<\dfrac{N+2s}{N+\beta+1}$ and
$$
\begin{array}{lll}
||u||_{E_a(\O_T)} &\le & \dyle C(\O_T)\bigg(||f\d^\beta||_{L^1(\O_T)}+\Big(\iint_{\O_T}|B(x,t)|^\s
\d^{\beta\s}dxdt\Big)^{\frac{1}{\s}}\Big(\iint_{\O_T}|\n v|^r dxdt\Big)^{\frac{1}{r}}\bigg)\\ &\le &
C(\O_T)\bigg(||f\d^\beta||_{L^1(\O_T)}+l\bigg)\le l^{\frac{1}{r}}.
\end{array}
$$
Hence we conclude that $K$ is well defined and that $K(E_\beta(\O_T))\subset E_\beta(\O_T)$. Now the rest of the proof follows exactly as the proof of Theorem
\ref{th2}.

To prove the uniqueness part under the additional hypotheses on $B$, we follow exactly the same arguments as in the proof of uniqueness part in Theorem \ref{th2}.
\end{proof}
\subsection{Applications}
In this subsection we will obtain some applications of Theorem \ref{th2} in order to prove a comparison principle and, as a consequence, a uniqueness result for some particular cases of the quasi-linear problem. We begin by showing the next comparison principle.
\begin{Theorem}\label{compa0}(Comparison Principle)
Let $w_1, w_2\in L^q(0,T;W^{1,q}_0(\O))$ for all $q<\frac{N+2s}{N+1}$, be such that

$$\begin{array}{ll}
\begin{cases}(w_1)_t+(-\Delta)^s w_1 = H_1(x,t,w_1,\n w_1) \hbox{ in }\Omega_T, \\
w_1 = 0 \inn (\mathbb{R}^{N}\setminus\Omega)\times (0,T),\\ w(x,0) = u_{0}(x) \inn\Omega,
\end{cases}
&
\begin{cases}(w_2)_t+(-\Delta)^s w_2=H_2(x,t,w_2,\n w_2)\hbox{ in }\O_T, \\
w_2=0\inn (\mathbb{R}^{N}\setminus\Omega)\times (0,T),\\ w_2(x,0)=\hat{u}_0(x)\inn\Omega,
\end{cases}
\end{array}$$

\

where
\begin{enumerate}
\item $H_1(x,t,w_1,\n w_1), H_1(x,t,w_2,\n w_2)\in L^1(\O_T)$ and $u_{0}, \hat{u}_{0} \in L^1(\O)$. \item $ H_1(x,t,w_1,\n w_1)-H_1(x,t,w_2,\n w_2)=\langle
    B(x,t,w_1,w_2), \n (w_1-w_2)\rangle + f(x,t,w_1,w_2)  \mbox{  in  }\O_T$ where $B\in (L^{m}(\O_T))^N$ with
    $m>\max\{q', \dfrac{2s}{(N+2s)-q(N+1)}\}>\dfrac{N+2s}{2s-1}$ for some $1<q<\dfrac{N+2s}{N+1}$ and $f\le 0$ in $\O_T$.
    \item $u_{0}\le \hat{u}_{0} \in L^1(\O)$.
\end{enumerate}
Then $w_1\le w_2$ in $\O_T$.
\end{Theorem}
\begin{proof}
Let $v=(u_1-u_2)$, then $v$ solves
\begin{equation}\label{op}
\left\{
\begin{array}{rcll}
v_t+(-\Delta )^s v& = & B(x,t,w_1,w_2)\,\cdot\, \n v + f(x,t,w_1,w_2) & \inn \Omega_T,\\ v(x,t)&=&0 &
\inn(\mathbb{R}^N\setminus\Omega)\times (0,T),\\ v(x,0)&=&v_0=u_{0}-\hat{u}_{0} & \inn\Omega.
\end{array}\right.
\end{equation}
Since $B$ satisfies the hypotheses of Theorem \ref{th2}, it follows that problem \eqref{op} has a unique solution. Now, using the fact that $f\le 0$ in $\O_T$ and $v_0\le 0$, then by Theorem \ref{th2}, we reach that $v\le 0$ in $\O_T$ and then we conclude.
\end{proof}
As a consequence, we get the next comparison result for approximated problems.
\begin{Theorem}\label{uniqapr}
Assume that $a>0$ and $\a>1$. Then for all $(f,u_0)\in L^1(\O_T)\times L^1(\O)$, the problem
\begin{equation}\label{apr0}
\left\{
\begin{array}{rcll}
u_t+(-\Delta )^s u&=&\dfrac{|\nabla u|^{\alpha}}{a+|\nabla u|^{\alpha}}+ f & \inn \Omega_T,\\ u(x,t)&=&0 &
\inn(\mathbb{R}^N\setminus\Omega)\times [0,T),\\ u(x,0)&=&u_{0}(x) & \inn\Omega,\\
\end{array}\right.
\end{equation}
has a unique solution $u_a$ such that $u_a\in L^{q}(0,T;W_{0}^{1,q}(\Omega))$ for all $q<\dfrac{N+2s}{N+1}$ and $T_k(u_a)\in L^{2}(0,T;H^s_0(\Omega))$, for all
$k>0$. Moreover, if $0<a_1<a_2$, then $u_{a_1}\ge u_{a_2}$.
\end{Theorem}
\begin{proof}
 Since the dependance on the gradient term is bounded, then the existence follows using  the arguments as in the proof of Theorem \eqref{gradiente}. Let show the comparison result which gives apriori the uniqueness part. Let $0<a_1\le a_2$ and consider $u_{a_1}, u_{a_2}$ the solutions to  \eqref{apr0} with $a=a_1,
a_2$ respectively. Setting $v=(u_{a_1}-u_{a_2})$, then $v\in L^\tau((0,T),W^{1,\tau}_0(\O))$, for all $\tau<\dfrac{N+2s}{N+1}$,  and $v$ solves
$$
v_t+(-\Delta )^s v=\dfrac{|\nabla u_1|^{\alpha}}{a_1+|\nabla u_1|^{\alpha}}-\dfrac{|\nabla u_2|^{\alpha}}{a_2+|\nabla u_2|^{\alpha}}.
$$
Let $H(\rho)=\dfrac{\rho^\a}{a_1+\rho^\a}, \,\, \rho\ge 0$, then
$$
v_t+(-\Delta )^s v=H(|\n u_1|)-H(|\n u_2|) -\bigg(\dfrac{|\nabla u_2|^{\alpha}}{a_1+|\nabla u_2|^{\alpha}}-\dfrac{|\nabla u_2|^{\alpha}}{a_2+|\nabla u_2|^{\alpha}}\bigg).
$$
It is clear that $h(x,t):=\bigg(\dfrac{|\nabla u_2|^{\alpha}}{a_1+|\nabla u_2|^{\alpha}}-\dfrac{|\nabla u_2|^{\alpha}}{a_2+|\nabla u_2|^{\alpha}}\bigg)\ge 0$ a.e. in $\O_T$. Thus $v$ satisfies

\begin{equation*}
\left\{
\begin{array}{rcll}
v_t+(-\Delta )^s v& = &B(x,t)\,\cdot\, \n v +h(x,t)& \inn \Omega_T,\\ v(x,t)&=&0 & \inn(\mathbb{R}^N\setminus\Omega)\times [0,T),\\
v(x,0)&=&0 & \inn\Omega,
\end{array}\right.
\end{equation*}
where
$$
B(x,t)= \left\{
\begin{array}{rcll}
0 &  & \mbox{  if   } & |\n u_{a_1}-\n u_{a_2}|=0,\\ \\ (H(|\n u_{a_1}|)-H(|\n u_{a_2}|))\dfrac{\n u_{a_1}-\n u_{a_2}}{|\n u_{a_1}-\n u_{a_2}|^2}  &  & \mbox{ if
} & |\n u_{a_1}-\n u_{a_2}|\neq 0.
\end{array}%
\right.
$$
One can easily see that $|B(x,t)|\le C$. Therefore by Theorem \ref{th2} we obtain that $v\le 0$ and then we conclude.
\end{proof}

\begin{Theorem}\label{uniqq}
Consider the problem
\begin{equation}\label{uniq-grad}
\left\{
\begin{array}{rcll}
u_t+(-\Delta )^s u&=&|\nabla u|^{\alpha}+ f & \inn \Omega_T,\\ u(x,t)&=&0 & \inn(\mathbb{R}^N\setminus\Omega)\times [0,T),\\
u(x,0)&=&u_{0}(x) & \inn\Omega,\\
\end{array}\right.
\end{equation}
where $\Omega\subset \ren$ is a bounded regular domain with $N> 2s$ and $\dfrac{1}{2}<s<1$. Assume that $\a<\dfrac{N+2s}{N+1}$ and $(f,u_0)\in L^1(\O_T)\times L^1(\O)$ are non negative functions. Then problem \eqref{uniq-grad} has a minimal solution  $u\in
    L^{q}(0,T;W_{0}^{1,q}(\Omega))$ for all $q<\dfrac{N+2s}{N+1}$. {In addition, if
    $\a<\dfrac{(N+2s)^2+2s(N+1)}{(N+1)(N+4s)}<\dfrac{N+2s}{N+1}$, then the solution is unique.}
\end{Theorem}
\begin{proof}
Using Theorem \ref{maria} we get the existence of $T_0\le T$ such that problem \eqref{uniq-grad} has a solution $u\in
    L^{q}(0,T_0;W_{0}^{1,q}(\Omega))$ for all $q<\dfrac{N+2s}{N+1}$.

To show the existence of a minimal solution we consider $f_n=T_n(f)$ and $u_{0n}=T_n(u_0)$. Define $u_n$ to be the unique solution to problem
\begin{equation}\label{apr00n}
\left\{
\begin{array}{rcll}
u_{nt}+(-\Delta )^s u_n&=&\dfrac{|\nabla u_n|^{\alpha}}{1+{\frac 1n}|\nabla u_n|^{\alpha}}+ f_n & \inn \Omega_{T_0}\equiv\Omega\times (0,T_0),\\ u_n(x,t)&=&0 &
\inn(\mathbb{R}^N\setminus\Omega)\times [0,T_0),\\ u_n(x,0)&=&u_{0n}(x) & \inn\Omega,\\
\end{array}\right.
\end{equation}
It is clear that $\{u_n\}_n$ is an increasing sequence of $n$. If $\hat{u}$ is a nonnegative solution to problem \eqref{uniq-grad} with $\hat{u}\in
    L^{q}(0,T_0;W_{0}^{1,q}(\Omega))$ for all $q<\dfrac{N+2s}{N+1}$,  then by the comparison principle in Theorem \ref{uniqapr} we deduce that $u_n\le \hat{u}$ for all $n$. Hence we get the existence of $u=\limit_{n\to \infty}u_n$ such that $u\le \hat{u}$. Thus $u\in L^r(\O_{T_0})$ for all $r<\dfrac{N+2s}{N}$. To finish we have just to show that $u$ is a solution to problem \eqref{uniq-grad}.

We claim that the sequence $\bigg\{\dfrac{|\nabla u_n|^{\alpha}}{1+{\frac 1n}|\nabla u_n|^{\alpha}}\bigg\}_n$ is bounded in $L^1(\O_{T_0})$.
 For simplicity of tipping we set $g_n(x,t)=\dfrac{|\nabla u_n|^{\alpha}}{1+{\frac 1n}|\nabla u_n|^{\alpha}}$ and we define $w_n$ to be the unique solution to the problem
\begin{equation*}
\left\{
\begin{array}{rcll}
w_{nt}+(-\Delta )^s w_n&=& f_n & \inn \Omega_{T_0}\equiv\Omega\times (0,T_0),\\ w_n(x,t)&=&0 &
\inn(\mathbb{R}^N\setminus\Omega)\times [0,T_0),\\ w_n(x,0)&=&u_{0n}(x) & \inn\Omega.\\
\end{array}\right.
\end{equation*}
By Theorem \ref{th1}, we obtain that the sequence $\{w_n\}$ is bounded in $L^{q}(0,T_0;W_{0}^{1,q}(\Omega))$ for all $q<\dfrac{N+2s}{N+1}$ and that $w_n\uparrow w$, the unique weak solution to the problem
\begin{equation*}
\left\{
\begin{array}{rcll}
w_t+(-\Delta )^s w&=& f & \inn \Omega_{T_0},\\ w(x,t)&=&0 &
\inn(\mathbb{R}^N\setminus\Omega)\times [0,T_0),\\ w(x,0)&=&u_{0}(x) & \inn\Omega.\\
\end{array}\right.
\end{equation*}
Thus
$$
u_n(x,t)=\dint_{0}^{t} \dint_{\Omega} g_n(y,\sigma) P_{\Omega} (x,y, t-\sigma)\,dy\,d\sigma +w_n(x,t)
$$
and
\begin{equation*}
|\nabla u_n(x,t)|\leq \dyle \dint_{0}^{t} \dint_{\Omega}g_n(y,\sigma)|\nabla_x P_{\Omega} (x,y, t-\sigma)|\,dy\,d\sigma +|\n w_n(x,t)|.
\end{equation*}
Fixed $q\in \Big(\a, \dfrac{N+2s}{N+1}\Big)$, then
\begin{equation}\label{nnn}
|\nabla u_n(x,t)|^q\leq \dyle C\bigg(\dint_{0}^{t} \dint_{\Omega}g_n(y,\sigma)|\nabla_x P_{\Omega} (x,y, t-\sigma)|\,dy\,d\sigma\bigg)^q +C |\n w_n(x,t)|^q.
\end{equation}
Setting
$$
D_n(x,t)=\bigg(\dint_{0}^{t} \dint_{\Omega}g_n(y,\sigma)|\nabla_x P_{\Omega} (x,y, t-\sigma)|\,dy\,d\sigma\bigg)^q,
$$
and taking into consideration that $\{|\n w_n|^q\}_n$ is bounded in $L^1(\O)$, to prove the claim we have just to show that $\{D_n\}$ is bounded in $L^1(\O_T)$.

As in the proof of the compactness part in Theorem \ref{th1}, we have
\begin{eqnarray*}
D_n(x,t) & = & \bigg(\dint_{0}^{t} \dint_{\Omega}g_n(y,\sigma)\frac{|\nabla_x P_{\Omega} (x,y, t-\sigma)|}{P_{\Omega} (x,y, t-\sigma)} P_{\Omega} (x,y, t-\sigma)\,dy\,d\sigma\bigg)^q\\
&\le & C\Big(\iint_{\{\O\times (0,t) \cap\{\d(x)>(t-\s)^{\frac{1}{2s}}\}\}\}}g_n(y,\sigma) \frac{P_{\Omega} (x,y, t-\sigma)}{(t-\s)^{\frac{1}{2s}}}\,dy\,d\sigma\Big)^q\\
&+ & \frac{C}{\d^q(x)}\dyle \Big(\iint_{\{\O\times (0,t) \cap\{\d(x)\le (t-\s)^{\frac{1}{2s}}\}\}\}}g_n(y,\sigma){P_{\Omega} (x,y, t-\sigma)}\,dy\,d\sigma\Big)^q\\ \\
&=& D_{n1}(x,t) + D_{n2}(x,t).
\end{eqnarray*}
Similar  to estimating the terms $J_{21}$ and $J_{22}$ in \eqref{j21}, \eqref{j21} respectively, we reach that
\begin{equation}\label{D21n}
\iint_{\O_T} D_{n1}(x,t)dxdt\le C(\O_T)T_0^{\frac{(r-(q-1))(\g_2+1)}{r}}||g_n||_{L^1(\O)}||\hat{u}||^{q-1}_{L^r(\O_{T_0})}
\end{equation}
and
\begin{equation}\label{dn22}
D_{n2}(x,t)=\frac{C}{\d^q(x)}\dyle \Big(\iint_{\{\O\times (0,t) \cap\{\d(x)\le t^{\frac{1}{2s}}\}\}\}}g_n(y,\sigma){P_{\Omega} (x,y, t-\sigma)}\,dy\,d\sigma\Big)^q\le C\frac{\hat{u}^q(x,t)}{\d^q(x)},
\end{equation}
where $r<\frac{N+2s}{N}$. Thus
$$
\iint_{\O_{T_0}} D_{n}(x,t)dxdt\le C(\O_{T_0})\iint_{\O_{T_0}}\frac{\hat{u}^q(x,t)}{\d^q(x)} dxdt +  T^{\frac{(r-(q-1))(\g_2+1)}{r}}||g_n||_{L^1(\O)}||\hat{u}||^{q-1}_{L^r(\O_{T_0})}.
$$
Using the fact that $\dfrac{\hat{u}^q(x,t)}{\d^q(x)}\in L^1(\O_{T_0})$ for all $q<\frac{N+2s}{N+1}$, $\hat{u}\in L^r(\O_{T_0})$ for all $r<\frac{N+2s}{N}$, and going back to \eqref{nnn}, it holds that
$$
\iint_{\O_{T_0}} |\n u_n(x,t)|^q dxdt\le C(\O_{T_0})||g_n||_{L^1(\O)}+C(\O_{T_0})\le C(\O_{T_0}) \iint_{\O_{T_0}}|\n u_n(x,t)|^\a dxdt + C(\O_{T_0}).
$$
Using the fact that $\a<q$ and by Young inequality we reach that $\dyle\iint_{\O_T} |\n u_n(x,t)|^q dxdt\le C(\O_{T_0})$ for all $n$. Hence $\{g_n\}_n$ is bounded in $L^1(\O_{T_0})$ and the claim follows. Therefore, using Theorem \ref{th1} we conclude that, up to a subsequence, $u_n\to u$ strongly in $L^{q}(0,T;W_{0}^{1,q}(\Omega))$ for all $q<\dfrac{N+2s}{N+1}$. Thus $u$ is the minimal solution to problem \eqref{uniq-grad} in $\O_{T_0}$.

{We prove now that the minimal solution $u$ can be defined in the set $\O_T$. According to Theorem \ref{maria}, the existence result holds for $L^1$ data in the set $\O\times (t_1,t_2)$ if $t_2-t_1\le \ddot{C}:=\ddot{C}(\O,s,N)$. Let $u$ the minimal solution obtained above in the set $\O\times (0,T_0)$ and suppose that $T_0<T$. Consider $T_1=T_0-\e$ with $\e>0$ is chosen such that $0<\e<\ddot{C}$. Then $u(.,T_1)\in L^1(\O)$ and then the problem
\begin{equation}\label{globTT}
\left\{
\begin{array}{rcll}
v_{t}+(-\Delta )^s v&=&|\nabla v|^{\alpha}+ f & \inn \Omega\times (T_1,T_1+\ddot{C}),\\ v(x,t)&=&0 &
\inn(\mathbb{R}^N\setminus\Omega)\times [T_1,T_1+\ddot{C}),\\ v(x,T_1)&=&u(x,T_1) & \inn\Omega,\\
\end{array}\right.
\end{equation}
has a minimal solution $v$. It is clear that $u$ is a solution of the same problem as $v$ in the set $\O\times [T_1,T_0)$. Hence $v=u$ in the set $\O\times [T_1,T_0)$. Setting
$$
\overline{u}(x,t)=
\left\{
\begin{array}{lll}
u(x,t) &\mbox{  if  }& (x,t)\in \O\times [0, T_1]\\
v(x,t) &\mbox{  if  }& (x,t)\in \O\times [T_1, T_1+\ddot{C})\\
\end{array}
\right.
$$
then $\overline{u}$ is the minimal solution to the problem
\begin{equation}\label{globTTTT}
\left\{
\begin{array}{rcll}
\overline{u}_{t}+(-\Delta )^s \overline{u}&=&|\nabla \overline{u}|^{\alpha}+ f & \inn \Omega\times (0, T_1+\ddot{C}),\\ \overline{u}(x,t)&=& 0 &
\inn(\mathbb{R}^N\setminus\Omega)\times (0,T_1+\ddot{C}),\\ \overline{u}(x,0)&=&u_{0}(x) & \inn\Omega.\\
\end{array}\right.
\end{equation}
Repeating the above argument in a finite time of steps we get the existence of a minimal solution $u$ to the problem \eqref{uniq-grad} defined in the set $\O_T$ with $u\in L^{q}(0,T;W_{0}^{1,q}(\Omega))$ for all $q<\dfrac{N+2s}{N+1}$.\\
}

{
Finally to show the uniqueness under the condition $\a<\dfrac{(N+2s)^2+2s(N+1)}{(N+1)(N+4s)}$. Notice that $\dfrac{(N+2s)^2+2s(N+1)}{(N+1)(N+4s)}<\dfrac{N+2s}{N+1}$ if and only if $2s>1$ which is our main hypothesis.\\
We will use the comparison principle in Theorem \ref{compa0}. If $u_1,u_2$ are two solutions to problem \eqref{uniq-grad} with  $u_1,u_2\in L^{q}(0,T;W_{0}^{1,q}(\Omega))$ for all $q<\dfrac{N+2s}{N+1}$. Then $v=u_1-u_2$ solves
\begin{equation}\label{uniq-grad000}
\left\{
\begin{array}{rcll}
v_t+(-\Delta )^s v&=& \langle B(x,t), \n v\rangle & \inn \Omega_T,\\ v(x,t)&=&0 & \inn(\mathbb{R}^N\setminus\Omega)\times [0,T),\\
v(x,0)&=& 0 & \inn\Omega,\\
\end{array}\right.
\end{equation}
where $v\in L^{q}(0,T;W_{0}^{1,q}(\Omega))$, $q<\dfrac{N+2s}{N+1}$ and $|B(x,t)|\le C(|\n u_1|^{\a-1}+|\n u_2|^{\a-1})$. It is clear that $B\in L^m(\O_T)$ for all $m<\dfrac{N+2s}{(N+1)(\a-1)}$. Recall that $\a<\dfrac{N+2s}{N+1}$, then $\a'<\dfrac{N+2s}{(N+1)(\a-1)}$. Since $\a<\dfrac{(N+2s)^2+2s(N+1)}{(N+1)(N+4s)}$, then
$$
\dfrac{N+2s}{(N+1)(\a-1)} >\max\{\a', \dfrac{2s}{(N+2s)-\a(N+1)}\}>\dfrac{N+2s}{2s-1}.
$$
Hence we can chose $m<\frac{N+2s}{(N+1)(\a-1)}$ such that $m>\max\{\a', \dfrac{2s}{(N+2s)-\a(N+1)}\}>\dfrac{N+2s}{2s-1}$. Hence by the comparison principle in Theorem \ref{compa0} we deduce that $v=0$ and then we conclude.
}
\end{proof}

{ Under additional regularity hypothesis on the solution, we can prove the next uniqueness result.
\begin{Theorem}\label{uniiq}
Assume that $\a>1$, then the problem \eqref{uniq-grad} has at most one solution $u$ such that $u\in L^{q}(0,T;W_{0}^{1,q}(\Omega))$ with $\dfrac{q}{\a-1}>\max\{\b', \dfrac{2s}{(N+2s)-\b(N+1)}\}$ for some $1<\beta<\dfrac{N+2s}{N+1}$. In particular, problem \eqref{uniq-grad} has at most one solution $u\in \mathcal{C}^{1}(\overline{\O_T})$.
\end{Theorem}
\begin{proof}
If $u_1,u_2$ are two solution with the above regularity, then wetting $v=u_1-u_2$, it holds that $v$ solves the problem
\begin{equation}\label{uniq-grad33}
\left\{
\begin{array}{rcll}
v_t+(-\Delta )^s v&=& B(x,t)\,\cdot\,  \n v & \inn \Omega_T,\\ v(x,t)&=&0 & \inn(\mathbb{R}^N\setminus\Omega)\times [0,T),\\
v(x,0)&=& 0 & \inn\Omega,\\
\end{array}\right.
\end{equation}
where $v\in L^{q}(0,T;W_{0}^{1,q}(\Omega))$ and $|B(x,t)|\le C(|\n u_1|^{\a-1}+|\n u_2|^{\a-1})$. According to the regularity hypothesis on $u_1,u_2$, we obtain that $B\in L^m(\O_T)$ with $m=\frac{q}{(\a-1)}>\max\{\b', \dfrac{2s}{(N+2s)-\b(N+1)}\}$ for some $1<\beta<\dfrac{N+2s}{N+1}$.
Thus using the comparison principle in Theorem \ref{compa0} we obtain that $u_1=u_2$ and then we conclude.
\end{proof}
}
\subsection{Some remarks on asymptotic behavior}
We now deal with asymptotic behavior of the  solutions. Let us begin by the next global existence result for the Cauchy problem given in \cite{DI}.
\begin{Theorem}\label{cauchy}
Assume that $\a<\dfrac{N+2s}{N+1}$ and $u_0\in W^{1,\infty}(\ren)\cap L^1(\ren)$, then the problem
\begin{equation}\label{cauchyeq}
\left\{
\begin{array}{rcll}
u_t+(-\Delta )^s u& = & |\n u|^\a & \inn \ren\times (0,T),\\ u(x,0)&=&u_0(x) & \inn\ren,
\end{array}\right.
\end{equation}
has a unique global solution $u$ such that $u\in \mathcal {C}([0,T], W^{1,\infty}(\ren))$ for all $T>0$. Moreover if $\a>\dfrac{N+2s}{N+1}$ and either
$||u_0||_{L^1(\ren)}$ or $||\n u_0||_{\infty}$ is small, then $||u(.,t)||_{L^1(\O)}\le C$ for all $t$.
\end{Theorem}
It is clear that if $u$ is a solution to problem \eqref{cauchyeq} with $u_0$ satisfying the conditions of Theorem \ref{cauchy}, then $u$ is globally  defined in $t$. We refer to \cite{DI} and \cite{W} for the proof.

In our case and according to the value of $\a$, we can prove the next partial blow up result.
\begin{Theorem}\label{blowup1}
Assume that $s\in (\dfrac{\sqrt{5}-1}{2}, 1]$ and suppose that $1+s<\a<\dfrac{s}{1-s}$. Then for all data $f\in L^\infty(\O\times (0,\infty))$, $f\ge 0$, the solution $u$ to
problem \eqref{grad} obtained in Theorem \ref{fix001} blows-up in a finite time in the sense that
$$
\io u(x,t)\d^s(x)dx\to \infty\mbox{  for }t\to T^*.
$$
\end{Theorem}
\begin{proof} Since $s>\dfrac{\sqrt{5}-1}{2}$ the interval of $\alpha$ is non empty.
We will use a convexity argument. Let $\phi_1$ be the first positive bounded eigenfunction of the fractional Laplacian, then $\phi_1$ satisfies
\begin{equation*}
\left\{
\begin{array}{rcll}
(-\Delta )^s \phi_1&=& \l_1\phi_1 & \inn \O,\\ \phi_1&>& 0 & \inn\O,\\ \phi_1 &=& 0 & \inn \mathbb{R}^N\setminus\Omega,
\end{array}\right.
\end{equation*}
and $\phi_1(x)\backsimeq \d^s(x)$. Using $\phi_1$ as a test function in the problem of $u$ and integrating in $x$, we reach that
$$
\frac{d}{dt}\io u(x,t)\phi_1(x) dx+\l_1 \io u(x,t)\phi_1(x) dx =\io |\n u|^\a\phi_1(x) dx\ge C\io |\n u|^\a\d^s(x)dx.
$$
Since $s<\a-1$, then using the weighted Hardy inequality \eqref{eq:super-hardy},  we conclude that
$$
\dyle \frac{d}{dt}\io u(x,t)\phi_1(x) dx+\l_1 \io u(x,t)\phi_1(x) dx \ge C(\O,s)\io \frac{u^{\alpha}(x,t)}{\d^{\alpha-s}(x)}dx.
$$
On the other hand $\a>s$, hence
$\dyle
\io \frac{u^{\alpha}(x,t)}{\d^{\alpha-s}(x)}dx \ge C(\O,s)\io u^{\a}(x,t)\phi_1(x) dx.
$
Therefore using Jensen's inequality it follows that
\begin{eqnarray*}
\dyle \frac{d}{dt}\io u(x,t)\phi_1(x) dx+\l_1 \io u(x,t)\phi_1(x) dx & \ge & \dyle C(\O,s)\io u^{\alpha}(x,t)\phi_1(x)dx\\
& \ge & \dyle C(\O,s,\alpha)\left(\io u(x,t)\phi_1(x) dx\right)^\alpha.
\end{eqnarray*}
Define $Y(t)=\io u(x,t)\phi_1(x) dx$, then
$$
Y'(t)+\l_1 Y(t)\ge C(\O,s,\alpha)Y^\a(t).
$$
A simple convex argument allows us to get the existence of $A$ such that if $Y(0)>A$, then $Y(t)\to \infty$ if $t\to T^*$ depending only on the data. Hence we
conclude.
\end{proof}

\end{document}